\documentclass[11pt,amsfonts, epsfig,reqno]{amsart}

\usepackage{amsmath, amscd, amssymb}
\usepackage{marginnote}
\usepackage{tikz}
\usepackage{hyperref}
\usepackage[utf8]{inputenc}
\usepackage{bbm}
\usepackage{yfonts}
\usepackage{float}
\usepackage{graphpap, color}
\usepackage{mathrsfs}

\usepackage{color}

\usepackage[top=1in, bottom=1in, left=1.3in, right=1.3in]{geometry}

\usepackage{cancel}
\usepackage[mathscr]{eucal}
\usepackage{pstricks}
\usepackage{color}
\usepackage{cancel}
\usepackage{verbatim}
\usepackage{latexsym}
\usepackage{tikz}
 \usepackage[arc,all]{xy}
\input xy
\xyoption{all}
\usepackage{array}
  \newcolumntype{C}[1]{>{\centering\arraybackslash$}m{#1}<{$}}
 \newlength{\mycolwd}                                         
\newlength{\mycolwdm}                                         
\newlength{\mycolwda}                                         
\newlength{\mycolwdb}                                         
\def\locg{\Theta}  \def\uev{^{\mathrm{ev}}} 
\def\npt{\n\mathrm{pt}}

\def\bipg{\Lambda}

\def\alp{\alpha}
\def\ep{\epsilon}
\def\bR{\mathscr R} 

\def\ureg{^{\mathrm{reg}}}

\def\ffp{\pt}
\def\ft{{t}}
\def\tt{\tau}
\def\tp{\varphi}
\def\sL{\mathscr L} \def\sN{\mathscr N}
\def\gnd{G_{g,n,\bd}}

\def\gndzo{G_{g,n,d}^{[0,1]}}

\def\lam{{\lambda}}

\usepackage{upgreek}
\usepackage{wasysym}
\usepackage{ esint }

\numberwithin{equation}{section}

\def\nmsp{\mathrm{NMSP}}

\newcommand{\id}{\mathrm{I}}

\newcommand{\bq}{\mathbf{q}}

\newcommand{\bt}{\mathbf{t}}

\def\PF{\mathcal{PF}}

\def\mS{{\mathscr S}}
\def\mT{\mathcal F^\bullet}

 \def\b{b}
\def\lgnd{_{g,n,\bd}}

\def\diag{{\mathrm {diag}}}
\def\sR{{ {\mathscr R}}}

\def\ff{{p}}

\def\vdim{\mathrm{vir}.\dim}
\def\bw{{\mathbf w}}
\def\virt{^{\vir}}

\newcommand{\tw}{\mathrm{tw}}

\def\lsta{_{\ast}}

\def\sO{\mathscr{O}}
\newcommand{\CC}{\mathbb{C}}
\newcommand{\NN}{\mathbb{N}}
\newcommand{\EE}{\mathbb{E}}
\newcommand{\LL}{L^\circ}

\newcommand{\PP}{\mathbb{P}}
\newcommand{\QQ}{\mathbb{Q}}
\newcommand{\RR}{\mathbb{R}}
\newcommand{\ZZ}{\mathbb{Z}}

\newcommand{\va}{{\vec{a}}}
\newcommand{\vb}{{\vec{b}}}

\newcommand{\vz}{{\vec{z}}}

\newcommand{\valp}{{\vec{\alp}}}
\newcommand{\vbeta}{{\vec{\beta}}}

\def\bS{\mathcal{S}}

\newcommand{\vir}{ {\mathrm{vir}} }
\newcommand{\ev}{ \mathrm{ev} }

\newcommand{\cal}{\mathcal}
\def\cC{{\cal C}}

\def\n{\mathrm{N}} 
\def\$\nmsp${{\n MSP}}

\def\cW{{\cal W}}

 \def\sC{{\mathscr C}}

\def\pr{\mathrm{pr}}

  \def\lalp{_\alpha}
\def\fl{{\mathrm{fl}}}

\def\dual{^{\vee}}

\def\sta{^\ast}
\def\ust{^{\mathrm{st}}}

\def\gnd{G_{g,n,\bd}}

 \newcommand{\Si}{\Sigma}
\newcommand{\Ga}{\Gamma}

\DeclareMathOperator{\End}{End}
\DeclareMathOperator{\Hom}{Hom}
\DeclareMathOperator{\Aut}{Aut}

\def\lggd{_{g,\gamma,\bd}}

\def\lev{{\mathrm{lev}}}

\newtheorem{prop}{Proposition}[section]
\newtheorem{convention}[prop]{Convention}
\newtheorem{remark}[prop]{Remark}
\newtheorem{definition}[prop]{Definition}
\newtheorem{corollary}[prop]{Corollary}
\newtheorem{example}[prop]{Example}

\newtheorem{dummy}{}[section]
\newtheorem{lemma}[prop]{Lemma}
\newtheorem{proposition}[dummy]{Proposition}
\newtheorem{theorem}[dummy]{Theorem}
\newtheorem{thm}{Theorem}
\newtheorem*{principle}{Principle}

\theoremstyle{definition}

\def\bd{\bold d}
\def\lloc{{\mathrm{loc}}}

\def\loc{^{\mathrm{loc}}}
\def\msp{^{M}}
\def\pt{\mathrm{pt}}
\def\hga{\Gamma}

\def\beq{\begin{equation}}
\def\eeq{\end{equation}}

\def\Pf{{\PP^4}}
 
\def\ti{\tilde}
\def\Lam{{\Lambda}}
\def\barM{{\overline{\mathcal M}}}

\def\sub{\subset}
\def\oT{\tilde T}

\let\ga=\Ga

 \DeclareMathOperator{\Cont}{Cont}
  \DeclareMathOperator{\Contr}{Contri}

\def\sH{{\mathscr H}}

\newcommand{\M}{\overline{\mathcal{M}}}

\renewcommand{\ev}{\mathrm{ev}}

\newcommand{\Res}{\mathrm{Res}}
\def\Pfn{\PP^{4+\n}}
\def\FF{\mathbb F}
\def\aA{\mathbb A}

\def\lra{\longrightarrow}
\def\and{\quad\mathrm{and}\quad}
\def\bl{\bigl(} \def\br{\bigr)}

\title[Polynomial structure  of GW potential via $\nmsp$]{Polynomial structure  of Gromov--Witten \\
 potential  of quintic $3$-folds via NMSP}

\author{Huai-Liang Chang}
\address{Department of Mathematics, Hong Kong University of Science and Technology, Hong Kong} \email{mahlchang@ust.hk}
\thanks{${}^1$Partially supported by Hong Kong grant GRF 16301515 and  GRF 16301717}

\author[Shuai Guo]{Shuai Guo}
\address{School of Mathematical Sciences and Beijing International Center for Mathematical Research, Peking University
}
\email{guoshuai@math.pku.edu.cn}
\thanks{${}^2$Partially supported by NSFC grants 11431001 and 11501013}

\author{Jun Li}
\address{Department of Mathematics, Stanford University,
USA; \hfil\newline 
\indent Shanghai Center for Mathematical Sciences, Fudan University, China} \email{jli@stanford.edu}
\thanks{${}^3$Partially supported by   NSF grant DMS-1564500 and DMS-1601211. }

\begin{document}

\begin{abstract} We use stable graphs to package the $\nmsp$ relations.
 Our tools are the $S$ matrix of the $\sO(5)$-twisted $\PP^{\n+4}$ equivariant GW theory, and the $R$ matrix obtained from the stablization of 
 the theory's localization formula.
\end{abstract}
\setcounter{section}{-1}
\setcounter{tocdepth}{1}
\maketitle
{\small
\tableofcontents
}

\section{Introduction}
 
This is the second of a three-paper series. In this paper, we will prove a structure result relating the 
Gromov-Witten (in short GW) potential function of quintic Calabi-Yau (in short CY) threefolds and the 
NMSP potential functions defined in the first paper. As a consequence, 
we will prove that the genus $g$ GW potential $F_g$ of quintic CY threefolds lies in
the Yamaguchi-Yau ring generated by five specified generators.

\medskip

In the first of this three-paper series, for a fixed integer $\n>1$, we construct the moduli space $\cW_{g,\gamma,\bd}$ 
of $\nmsp$ fields of numerical data $(g,\gamma,\bd)$; construct its 
$G=(\CC\sta)^\n$-equivariant virtual cycle $[\cW_{g,\gamma,\bd}]\virt$, and prove a vanishing that implies
that the $\nmsp$ invariants have a desired graph sum formula.

For notational simplicity, in this and the subsequent paper we will only consider 
$\nmsp$ fields with $n$ scheme markings decorated with
$(1,\rho)$. 
Accordingly we write such numerical data as $(g,n,\bd)=(g,n,(d_0,d_\infty))$, and write the associated moduli space
as $\cW_{g,n,\bd}$.

For the given $(g,n,\bd)$, an $\nmsp$ field is (cf. \cite[Definition 2.1]{NMSP1})
\beq\label{xi0}
\xi=(\sC,\Sigma^\sC,\sL,\sN,\varphi_1,\cdots,\varphi_5,\rho,\mu_1,\cdots,\mu_\n,\nu),
\eeq
consisting of a pointed twisted curve, two invertible sheaves $\sL$ and $\sN$, and various fields
(see Section \ref{sec2}). 
The group $G=(\CC\sta)^\n$ acts on $\xi$ by scaling   $(\mu_i)^\sigma=(\sigma_i\mu_i)$. 
Let $\Sigma_i\sub \Sigma^\sC$ be the $i$-th marking, then we have a $G$-equivariant evaluation morphism
\begin{align}\label{eva}\ev_i:\cW_{g,n,\bd}\lra \PP^{4+\n}, \quad
\xi\mapsto [\varphi_1,\cdots,\varphi_5,(\mu_1/\nu),\cdots, (\mu_\n/\nu)]\big|_{\Sigma_i},
\end{align}
where $G$ acts on $\PP^{4+\n}$ via scaling the last $\n$-homogeneous coordinates.
\smallskip

\noindent
{\sl Convention}. {\sl We 
let $\FF=\QQ(\zeta_\n)(t)$, where $\zeta_\n=e^{{2\pi\sqrt{-1}}/{\n}}\!$ and $t$ is a formal variable.
We set
\beq\label{A}
\aA=\FF[\![q]\!],
\eeq
which will be the coefficient ring to be used throughout this paper. 

The ring $H_G\sta(\pt)$ is generated by {\sl standard generators} 
$\ft_\alp$, with  $\alpha\in [\n]:=\{1,\cdots,\n\}$. Our convention is that after equivariant integration 
we always substitute $t_\alp$ by 
$$t\lalp=-\zeta_\n^\alpha t\in\FF.
$$ 
}

We will reserve Greek letters $\alpha$, $\beta$, etc., to mean indices in $[\n]$. 
One useful formula is 
$\prod_{{\beta\neq \alp}}(\ft_\beta-\ft_\alp)=\n \frac{t^\n}{- t_\alp}$. 
\medskip

For $\tau_i(z)\in H\sta_G(\PP^{4+\n})[z]$, $1\le i\le n$, 
we define the $\nmsp$ correlators:
\begin{align}\label{mspsum-0}
\bigl< \bigotimes_{i=1}^n \tau_i(\psi_i)\bigr>^M_{g,n,d_\infty}&=  \sum_{d\geq 0} (-1)^{d+1-g } q^d \int_{[\cW_{g,n,(d,d_\infty)}]\virt} 
\prod_{i=1}^n  \ev_i\sta \tau_i(\psi_i) 
\in\aA.
\end{align}
As $\vdim \cW_{g,n,(d,d_\infty)}=\n(d+1-g)+d_\infty$ is linear in $d$, 
this is a polynomial in  $q':=-q/\ft^\n$  with coefficients in $\FF$, after substituting $\ft\lalp$ by $-\zeta_\n^\alpha t$. 



\medskip
The structure of this series will be studied via virtual localizations \cite{GP}. For fixed numerical
data $(g,n,\bd)$, the localization relevant fixed locus has the following open and closed decomposition 
(cf. \cite[(4.4)]{NMSP1})
\begin{align*}
\bl\cW_{g,n,\bd}\br^G=\coprod_{\locg \in  \gnd^\fl} F_\locg.
\end{align*}
Here $\gnd^\fl$ is a class of flat decorated graphs of numerical data $(g,n,\bd)$.

Applying virtual localization, 
we obtain
\beq\label{vir-loc}
[\cW\lggd]\virt= \sum_{\locg \in  \gnd^\fl}  \frac{[F_\Theta]\virt}
{e(N\virt_{\Theta})}:= \sum_{\locg \in  \gnd^\fl}  \Cont_\Theta.
\eeq

The structure of these graphs will be recalled in Section 1. Here we briefly state some important features of such
graphs. 

Firstly, there is a subclass 
$\gnd\ureg\sub\gnd^\fl$ that has distinguished merits: 
when all $u_i\in H\sta_G(\Pfn)$, then \eqref{vir-loc} still holds when the summation is over the subclass 
$\Theta\in \gnd\ureg$.

Secondly, vertices of graphs in $\gnd^\fl$ have levels,  level $0$, level $1$ and level $\infty$. 
Any $\Theta\in \gnd\ureg$ will have no edge incident to a level $0$ vertex and a level $\infty$ vertex simultaneously. 
This way, every such $\Theta$
can be decomposed into its ``$[0,1]$" part and its ``$(1,\infty]$" part. 
We say a graph is supported on $[0,1]$(resp. $\infty$) if all its vertices have levels $0$ or $1$(resp. $\infty$).  
 The $\nmsp$-$[0,1]$ correlators are defined by the graphs
$$\gndzo=\{\Theta\in G_{g,n,(d,0)}\ureg\mid \Theta\ \text{is supported on} \ [0,1]\},
$$
via (compare \eqref{vir-loc})
\beq\label{01}
\bl [\cW\lggd]\virt\br^{[0,1]} = \sum_{\locg \in  \gnd^{[0,1]}}  \Cont_\Theta.
\eeq
 
\smallskip
We now fix the space of insertions. We form the fixed locus $(\PP^{4+\n})^G$, which is a union of 
$\Pf\sub\Pfn$ with $\n$ isolated 
fixed points $\pt_\alpha$, $\alpha\in [\n]$, where 
$\Pf=(z_{i>5}=0)$, and $\pt_\alpha$ has $z_j=\delta_{j,5+\alpha}$.
Let $Q\sub\Pf$ be the Fermat quintic threefold.  We abbreviate 
$\npt=\{\pt_1\}\cup\cdots\cup\{\pt_\n\}$ and  let
$$\aleph=Q\cup\npt =Q\cup \{\pt_1\}\cup\cdots\cup\{\pt_\n\}. 
$$
 We set our state space 
 $\sH=H\sta_G(\aleph,\FF).$
 
Let $p\in H^2_G(\Pfn)$ be the hyperplane class $(z_1=0)$; $H=p|_Q$. 
Let $(g,n)$ be in stable range, meaning that $2g-2+n>0$.  
In this paper, we adopt the convention that $\left< \tau_{\mathbf n}(\psi_{\mathbf n})\right>=\left<\bigotimes_{i} \tt_i(\psi_i)\right>$, 
where  $i =0,\cdots, n$ .


 
 \begin{definition}\label{01theory-def}
For $\tt_i(z) \in \sH[\![z]\!]$, we define the $\nmsp$-$[0,1]$ correlators
\beq\label{01correlator}
\bigl<\tau_{\mathbf n}(\psi_{\mathbf n})\bigr>^{[0,1]}_{g,n}
=  \sum_{d\ge 0}  (-1)^{d+1-g } q^d \int_{\bl [\cW_{g,n,(d,0)}]\virt\br^{[0,1]}} 
\prod_{i=1}^n \ev_i\sta \tt_i(\psi_i) \in \aA;
\eeq
where $\aA$ be defined in \eqref{A}.  We form the $\nmsp$-$[0,1]$ theory
$$
\bigl[\tau_{\mathbf n}(\psi_{\mathbf n})\bigr]^{[0,1]}_{g,n}
=  \sum_{d\ge 0}  (-1)^{d+1-g } q^d 
\bl \pr_{g,n}\br\lsta \Bigl(\prod_{i=1}^n \ev_i\sta \tt_i(\psi_i)\cdot \bl [\cW_{g,n,(d,0)}]\virt\br^{[0,1]} \Bigr),
$$
taking values in $H^*(\M_{g,n},\aA)$, where $\pr_{g,n}:\cW_{g,n,\bd}\to\M_{g,n}$ is the projection.
\end{definition}

 
\medskip

Pullback the psi-class on $\barM_{g,n}$, we get 
the ancestor classes $\bar\psi_i$. We define 
$\bigl[   \tau_{\mathbf n}(\bar \psi_{\mathbf n}) \bigr]_{g,n}^{[0,1]}$ 
by replacing $\tt_i(\psi_i)$ in Definition \ref{01theory-def} with $\tt_i(\bar\psi_i)$. The ancestor correlator is defined by
$$   
\bigl<   \tau_{\mathbf n}(\bar \psi_{\mathbf n}) \bigr>_{g,n}^{[0,1]}  =  \int_{[\M_{g,n}]}  \left[  \tau_{\mathbf n}(\bar \psi_{\mathbf n}) \right]_{g,n}^{[0,1]}.
$$ 
The $[0,1]$-theory $\Omega^{[0,1]}_{g,n}(\tau_{\mathbf n} )\!:=\left[\tau_{\mathbf n}  \right]_{g,n}^{[0,1]}$ takes a concise form when phrased 
as an $R$-matrix action on Cohomological Field Theory(CohFT), which is introduced in \cite{PPZ,NMSP3}. 
 
\begin{thm}\label{Omega01}
{Let $\Omega^{\aleph}$ be the CohFT associated to $\aleph=Q\cup\npt$, and $R(z)\in \End \sH  \otimes \aA[\![z]\!]$ be the R-matrix defined via the factorization between the local and global
S-matrices. Then
the $\nmsp$-$[0,1]$ theory gives a CohFT $\Omega^{[0,1]}$, relating to the local CohFT $\Omega^{\aleph}$ via}
\begin{align*} \Omega^{[0,1]} = R. \Omega^{\aleph}.
\end{align*}
\end{thm}
 The precise meaning of this theorem will be explained in the last part of this introduction. This theorem will be restated
in the third paper of this three-paper series, after the notion of CohFT is recalled.

\medskip

Further, the $\nmsp$-$[0,1]$ theory has a similar degree bound as the total $\nmsp$ theory.
\begin{thm}[{Polynomiality of $\nmsp$-$[0,1]$ correlators}]\label{thm1} 
For $(g,n)$ in stable range, and for $m_i\in [0,\n+3]$, the $[0,1]$-correlator
\begin{align}\label{01des-gen}
\ft^{\sum(n-m_i-k_i) -\n(g-1)} \bigl<   p^{m_1} \bar \psi_1^{k_1} ,\cdots,   p^{m_n}  \bar \psi_n^{k_n} \bigr>_{g,n}^{[0,1]}   \in \mathbb \QQ[q']
\end{align}
is a polynomial in $q':=-q/\ft^\n$, of degree bounded by
$$  
g-1+ \frac{3g-3+ \sum_{i=1}^n  m_i }{\n}. 
$$ \end{thm}

{
\begin{remark}
Theorem \ref{Omega01} and Theorem  \ref{thm1} together give us the explicit relations between global and local generating functions  via the  $R$-matrix, which improves the number version of the algorithm in \cite{CLLL}. For each $g>0$, these relations determine the GW potentials 
$$ 
F_g(\mathrm Q) := \sum_{d\geq 0} N_{g,d} \mathrm Q^d ,\qquad  N_{g,d}:=\int_{\M_{g}(Q,d)}  1$$ from the lower genus GW potential $\{F_h, h<g\}$, up to $g-1$ initial condition. This will be the starting point of our third paper \cite{NMSP3}.
\end{remark}}

\medskip
%
%
%

Let $I(z)$ be the $I$-function of quintic threefolds; we write 
\begin{align}\label{Q-Ifunc} I^Q(z) 
:=z    \sum_{d=0}^\infty q^d\frac{ \prod_{m=1}^{5d}(5 H +m z)}{\prod_{m=1}^d(H+ m z)^5  }= \sum_{i=0}^3  I_i(q) H^i z^{1-i} .\end{align}
Following \cite{YY}, we introduce $D:= q\frac{d}{dq}$ and
$$I_{1,1} = 1+D \Big(\frac{I_1 }{I_0}\Big),\quad A_k:=\frac{D^k  I_{1,1}}{I_{1,1}},\quad B_k:=\frac{D^k   I_{0} }{I_{0}},\quad Y=  \frac{1}{1-5^5 q}.
$$
It was argued in \cite{YY}, based on \cite{BCOV}, that the ring
\beq\label{R}
\bR=\mathbb Q[A_k,B_k,Y:k\ge 1]
\eeq
is generated by $A_1, Y, B_1, B_2$ and $B_3$, and is closed under the  differential operator $D$. Using the structure results proved, 
we provide a mathematical proof of

\begin{thm}\label{thm3}
[{Yamaguchi-Yau's Polynomial structure}] 
For $(g,n)$ in stable range, all 
$$P_{g,n}= \frac{Y^{g-1} }{I_0^{2g-2}  } \cdot D^n F_{g}(q \,e^{I_1/I_0})\in \bR=\QQ[A_1,B_1,B_2,B_3,Y].
$$
 \end{thm}
Note that, we have not used the  algebraic independence of the five generators here, though we expect it is true.  Indeed in the proof of Theorem \ref{thm3}, we give a canonical way to represent $P_{g,n}$ as an element in $\sR$.

As a direct consequence, we prove that these potentials are germs of analytic functions.

\begin{thm}\label{thm4}
For all $g>1$, the GW potential $F_g(\mathrm Q)$ is an analytic function of $\mathrm Q$ in an analytic neighborhood
of $0\in\CC$.
\end{thm}

The convergence result is shown in \cite{CI} for toric varieties. Our result provides  the first non-toric example.  

\vspace{0.8cm}
{\noindent
{\bf Restatement of Theorem \ref{Omega01}.} \black
Let $S\msp$, $S^{\ffp_\alp}$ and $S^Q$ be the $S$-matrix (cf.
Section \ref{JS}) of the $\nmsp$-theory, the GW theory of point $\ffp_\alp$ 
and of the GW theory of quintic threefolds at the base point $0$, {$\tau_\alpha(q)$ and $ \tau_Q(q)$, respectively ,}
where
\begin{align}\label{tauloc0} \textstyle
\tau_\alpha(q):={\zeta_\n^\alpha}\,\ft \cdot  \int_0^q \big(  (1-5^5x)^{\frac{1}{\n}} -1\big)   \frac{dx}{x}, \qquad  \tau_Q(q):=I_1(q)/I_0(q)H.
\end{align}  
Our $R$-matrix is defined by the following Birkhoff factorization (c.f. Section \ref{birkhoff}) 
  \begin{equation}\label{DefineR}
\  S\msp(z) \begin{pmatrix} \diag \{\Delta^{\ffp_\alp}(z) \}_{\alpha=1}^\n & \\ &1  \end{pmatrix} 
= R(z)    \begin{pmatrix} \diag \{S^{\ffp_\alpha}(z)  
\}_{\alpha=1}^\n& \\ & S^Q(z)  \end{pmatrix}\Big|_{q\mapsto q'} ,
\end{equation}
where 
by Quantum Riemann-Roch (c.f. \cite{CG})  \footnote{Here $B_{2k}$ are bernoulli numbers.}
\begin{multline} \label{deltapt}
\textstyle \Delta^{\ffp_\alp}(z):=  \exp \Big[  \sum_{k> 0 }     \frac{B_{2k}}{2k(2k-1)}   \Big(\frac{5}{(-\ft_\alpha)^{2k-1} }
 +\frac{1}{(5\ft_\alpha)^{2k-1}} 
 +\sum_{\beta\neq \alpha}\frac{1}{(\ft_\alpha-\ft_\beta)^{2k-1} } \Big) z^{2k-1}\Big] .
\end{multline}



\begin{definition} \label{locclass} 
We define the translated classes at fixed loci $F$ to be
 
\begin{equation*} 
{\small
[ - ]^{F,T}_{g,n}  :=\bigg.
 {
\left\{
  \begin{aligned}
 &  I_0(q')\black^{-(2g-2+n)}    \sum_{d\geq 0}   \frac{ q^d {e^{d( I_1(q)/I_0(q))} } }{ (- \ft ^\n)^{(d+1-g)}}
  \cdot\ff_{g,n*}  \big( - |_Q \cap  [\M_{g,n}(Q,d)]^{\vir} \big)\bigg.  &\text{  if  $F=Q $,  } \  \\
 &   L(q')\black^{\frac{ \n+3}{2}(2g-2+n)}  \! \frac{5^{1-g}}{(\n (- t_\alpha)^{ 3+\n})^{1-g} }   \sum_{k\geq 0}  \frac{1}{k!}     \bigg. 
\!\!\! \cdot\pr_{g,n *}   \big(- |_{ \ffp_\alpha} \!\!\! \cup   \oT_\alpha(\psi)^k \big)
 &  \text{  if  $F= \ffp_\alpha$,} \\
  \end{aligned}
\right.}}
\end{equation*}
where $\oT_\alpha(z):=z(1- L(q')^{\frac{ \n+3}{2}}  R  (z)^{-1} \mathbf 1)|_{\ffp_\alpha}$ 
and   $L(q):=(1-5^5q)^{1/\n}, q':=-q/t^\n$.
Notice that since each term in $\oT_\alpha(z)$ is divisible by $z^2$, 
the summation of $[-]^{\ffp_\alp,T}_{g,n}$ is finite. 

\end{definition}

\smallskip

\begin{thm}\label{R-action-thm}
 Let $G_{g,n}^\n$ be the set of stable (dual) graphs 
of $n$ legs and total genus $g$ so that each vertex $v$ 
is decorated by $F_v \in \{Q, \ffp_1,\cdots,\pt_\n\}$. Suppose $ \n\gg 3g-3+n$, then
 \begin{align*}
 \left[  \tt_1(\psi_1) ,\cdots,   \tt_n(\psi_n)   \right]^{[0,1]}_{g,n} 
=\sum_{\Gamma\in G_{g,n}^\n} \frac{1}{|\Aut \Gamma|} {\xi_\Gamma}_*\circ \Cont_\Gamma
\end{align*}
where $\xi_\Gamma: \M_\Gamma \rightarrow \M_{g,n}$ is the tautological map, and $\Cont_\Gamma$ 
is defined via: 
\begin{enumerate}
\item   at each stable vertex $v$, we place a linear map
$$\qquad
\bigotimes_{j=1}^{n_v} \gamma_j(z_j)  \longmapsto     \Big[\prod_{j=1}^{n_v} \gamma_j(\bar \psi_j)   \Big]^{{F_v},T}_{g_v,n_v}  \quad \in H^*(\M_{g_v,n_v}) ;
$$ 
\item  at each edge, we place a bi-vector valued series 
$$
\qquad \qquad \sum_\alpha  \frac{ e_\alpha \otimes  e^\alpha-  R(z)^{-1} e_\alpha \otimes   R(w)^{-1} e^\alpha}{z+w}
\in (\sH  \otimes \sH)\otimes \aA  [\![z,w]\!];$$
\item at each leg, we place a  vector valued series
$$   R(z)^{-1}[S^M(z^-)\tau_i(z)]_+\in \sH\otimes \aA[\![z]\!]. \qquad
$$
\end{enumerate}
\end{thm}}
Here in (3),   $f(z^-)$ means that we expand $f(z)$ near $z=\infty$, and $[-]_+$ means that we  
take the nonnegative-power part of a formal series. Note that if $\tau_i(z)=\tau_i\in\sH$, 
item (3) becomes $\displaystyle R (z)^{-1}   \tt_i$.



\vspace{0.7cm}

 The organization of this paper is as follows. In \S1, we  recall $\nmsp$ localizations and use Givental's theory to study $g=0$   generating functions of $\nmsp$-$[0,1]$ theory. 
In \S2, we package $\nmsp$ localization graphs to ``bipartie graphs" and  decompose (Theorem \ref{summation})  $\nmsp$ potentials into $[0,1]$ and $(1,\infty]$ potentials. In \S3, we package $[0,1]$ theory by using the stabilization to prove the $R$-matrix action formula (Theorem \ref{Omega01}, or more precisely Theorem \ref{R-action-thm}). In \S4,
we prove $[0,1]$'s polynomiality (Theorem \ref{thm1})  by  using $\nmsp$'s polynomiality established in \cite{NMSP1} and Theorem \ref{summation} in \S2. In \S5, we prove 
Yamaguchi-Yau's polynomial structure by using Theorem \ref{thm1}, Theorem \ref{Omega01} and the polynomiality of $R$-matrix. In \S6, we prove a key property of  the $R$ matrix.

 \vspace{0.5cm}

\subsection*{Acknowledgements}
{
The authors want to thank Weiping Li, Chiu-Chu Melissa Liu and Yongbin Ruan for early discussions. The first and second authors would like to thank Hsian-Hua Tseng for his lectures on Givental's formalism at Peking University. The second author would also like to thank Dustin Ross for early collaborations on  computing higher genus FJRW invariants by using MSP methods. 
}

\def\bd{\bold d}
\newpage

\section{The $\nmsp$ correlators and  genus zero structures}
\label{sec2}
\def\bd{\bold d}

We recall the technical details of the $\nmsp$ moduli space necessary for this paper. An $\nmsp$ field
(of the Fermat quintic) to be used in this paper are $\xi$ in \eqref{xi0} so that $\Si^\sC\sub \sC$ is a genus $g$,
$n$-pointed twisted curve where all markings are scheme points over which the field $\rho$ must vanishes; 
$\sL$ and $\sN$ are invertible sheaves of
$\sO_\sC$-modules so that $\sL\oplus\sN$ is representable; and the fields are
$$\varphi_i\in H^0(\sL),\ \rho\in H^0(\sL^{-5}\otimes\omega_\sC),\
\mu\lalp\in H^0(\sL\otimes\sN),\ \text{and}\ \nu\in H^0(\sN),
$$
satisfying properties that $(\varphi,\mu)$, $(\rho,\nu)$, and
$(\mu,\nu)$ are nowhere vanishing. It is called stable if $\Aut(\xi)$ is finite. It is called an MSP field if $\n=1$.

The field $\xi$ has numerical data: the genus $g$, the number of markings $n$, and the degrees $d_0=\deg\sL\otimes\sN$
and $d_\infty=\deg\sN$. Given $(g,n,\bd)$, $\bd=(d_0,d_\infty)$, the moduli of stable $\nmsp$ fields of given data forms a
DM stack $\cW_{g,n,\bd}$, acted on by $G=(\CC\sta)^\n$ via tautological scaling the $\n$ factors of the $\mu$
fields. It has a cosection localized $G$-equivariant virtual cycle 
$$[\cW_{g,n,\bd}]\virt\in A\lsta (\cW_{g,n,\bd}^-),
$$
where $\cW_{g,n,\bd}^-$ is the degeneracy locus of the cosection used to defined the cosection
localized virtual cycle $[\cW_{g,n,\bd}]\virt$. Further, $\cW_{g,n,\bd}^-$ is a proper $G$-invariant substack of 
$\cW_{g,n,\bd}$ (cf. \cite{NMSP1}).

Apply virtual localization formula \cite{GP, CKL}, for insertions from $H\sta_G(\PP^{4+\n})$, 
\begin{align*}
\int_{[\cW_{g,n,\bd}]\virt}(-)= \sum
\int_{[F_\Theta]\virt}\frac{(-)}{e(N_{\Theta}\virt )},
\end{align*}
where the summation is over all
regular graphs $\Theta\in \gnd\ureg$.

To proceed, we 
list the relevant part of the definition of a flat decorated graph $\locg\in\gnd^\fl$. For its precise definition, please
see \cite{NMSP1}.  The graph $\locg$ consists of vertices $V$, edges $E$, and legs $L$, plus decorations on them.
(We will use $E(\Theta)$ to emphasize the dependence of $E(\Theta)$ on $\Theta$; when $\Theta$ is understood,
we abbreviate it to $E$.) We say an edge $e\in E$ is incident to $v\in V$ if $e$ is attached to $v$, etc..
The decorations of $\Theta$ are level, hour, genus, and degrees.

{\bf level}:  Level is a map $\ell_{\cdot}: V\to \{0,1,\infty\}$; we agree $v\in V_{\ell_v}$;
when $e\in E$ is incident to $v,v'\in V$, we say $e\in E_{\ell_v\ell_{v'}}$.

{\bf hour}: Hour is a map $\alpha_\bullet: V_\infty\cup V_1\lra [\n]$, such that when
$e\in E_{aa'}$ is incident to $v\in V_a$ and $v'\in V_{a'}$, 
then $\alpha_v\ne \alpha_{v'}$ when $a= a'$, and $\alpha_v= \alpha_{v'}$ when $a\ne a'$.

{\bf genus and degree}: Genus is as usual; degree is a map 
$(d_{0\bullet},d_{\infty\bullet}): V\cup E\to (\QQ)^2$.

\smallskip

For $e\in E_{1\infty}$ we let $\alpha_e=\alpha_v$ for $v$ a vertex incident to $e$; it is well-defined.
We let $E_{1\infty}^\alp$ be the set of  $e\in E_{1\infty}$ with hour $\alp$, likewise $V^\alp_{\ell}$.
 
The total degree of a graph $\Theta$ is $(d_0,d_\infty)$, where $d_\bullet=\sum_{a\in V\cup E} d_{\bullet a}$; 
the total genus of $\Theta$
is $g=h^0(\Theta)+\sum g_v$.
We adopt the convention that for $v\in V$, $L_v$ (resp. $E_v$) is the set of legs 
(resp. edges) incident to $v$. 
 We say $v$ is unstable if $g_v=d_{0v}=d_{\infty v}=0$ and  $|L_v|+|E_v|\le 2$.  

\smallskip

When the graph is regular, among other things, its $E_{0\infty}=\emptyset$.
Thus every $\Theta\in \gnd\ureg$ can be decomposed along level $1$ vertices to its $[0,1]$-part and its
$(1,\infty]$-part. It is proved in \cite{NMSP1} that
only regular graphs will possibly have non-zero contribution to localizations.


\subsection{The state spaces}\label{statespace}

 
The state space, along with its even degree part, is
$$\sH=H\sta(\aleph,\FF),\ \and \sH\uev=H\uev(\aleph,\FF);
$$
they are $\FF$-vector spaces.  The $\sH$ has its unit $1=1_Q+\sum 1_\alp\in H^0_G(\aleph,\FF)$, where  $1_Q$
and $1_\alp$ are units of $H\sta(Q)$ and $H\sta(\ffp_\alp)$, respectively. Note that $\sH\uev$ is the image of
$H_G\sta(\Pfn,\FF)$ in $\sH$. We now introduce a bilinear pairing on $\sH$.

We denote by $\cW\lgnd^-$, $\bd=(d,d_\infty)$, the degeneracy locus of the cosection used to defined the cosection
localized virtual cycle $[\cW\lgnd]\virt$. Then over each $G$ fixed loci $F_\Theta\sub \cW\lgnd$ the evaluation map \eqref{eva}  restricts to  
\begin{align*}
\ev_i:  F_\Theta\cap \cW\lgnd^- \lra \aleph\sub\Pfn.
\end{align*}
Since the cycles $[F_\Theta]\virt\in A\lsta(F_\Theta\cap \cW\lgnd^-)$, we can integrate over $[F_\Theta]\virt$ the pullback $\ev_i\sta$ of any class in $\sH$, refining the definition \eqref{mspsum-0}.

We define, for $\tau_i(z)\in \sH(z)=\sH\otimes\FF(z)$: 
\begin{align}\label{mspsum0}
 \bigl<\tau_{\mathbf n}(\psi_{\mathbf n})\bigr>\msp_{g,n,d_\infty}
&:=\sum_{d\geq 0} (-1)^{d+1-g } q^d \sum_{F_\Theta} \int_{[F_\Theta]\virt}   \frac{1}{e(N\virt_{F_\Theta})}  
\prod_{i=1}^n   \ev_i\sta \tau_i(\psi_i),
\end{align} 
where the sum is over all $\Theta\in G\lgnd^\fl$ (cf. \cite{NMSP1}).  By localization formula, \eqref{mspsum0} reproduce \eqref{mspsum-0}.

In case $d_\infty=0$, we abbreviate
$\left< -\right> \msp_{g,n}:=\left< -\right> \msp_{g,n,0}$. If further $g=0$,
\beq\label{key-iso} \cW_{0,n,(d,0)}  \cong \barM_{0,n}(\PP^{4+\n},d), 
\eeq
equivariantly identifies the obstruction sheaf (bundle) as the dual of $\pi_{d\ast} f_d\sta L_p^{\otimes 5}$,
where $L_p:=\sO_{\PP^{4+\n}}(1)$ and $(\pi_d,f_d): \cC\to \barM_{0,n}(\PP^{4+\n},d)\times \PP^{4+\n}$  
are the universal family.
 Thus \eqref{key-iso} induces
$$[\cW_{0,n,(d,0)}]\virt=(-1)^{d+1} e_T(\pi\lsta f\sta L_p^{\otimes 5}) \cap [\barM_{0,n}(\PP^{4+\n},d)],
$$
together with their torus localization formulas. 

\begin{principle}[$\bigstar$]
The genus zero $\nmsp$ theory equals the genus zero  $L_p^{\otimes {5}}$-twisted $G$-equivariant GW theory of 
$\PP^{4+\n}$. In particular 
\begin{align}\label{prin} \bigl<  \tau_{\mathbf n}(\psi_{\mathbf n})\bigr>\msp_{0,n}=  \sum_{d\geq 0}   
q^d \int_{ \barM_{g=0,n}(\PP^{4+\n},d) }  e_T(\pi_{d\ast}  f_d\sta  L_p^{\otimes 5})  \prod_{i=1}^n  \tau_i(\psi_i).
\end{align} 
\end{principle}

This is how we employ Givental's theory in $\nmsp$.


\medskip
We now define the mentioned pairing on $\sH$:
\begin{align}\label{state-pairing} (x ,y )^\tw:=  q^0  \ \text{coefficient of }        
\left<1,x,y\right>^M_{0,3,0} \ \  \forall x,y\in \sH. \end{align}
Our state space is the triple $\big(\sH, (\,,\,)^{\tw},1\big)$ as an inner product space with unit. 
The pairing can be further calculated by   applying localization \eqref{state-pairing} to \eqref{mspsum0}:
\begin{align*}
 &(\cdot,\cdot)^{\tw}=(\cdot |_Q,\cdot|_Q)^{Q,\tw}+\sum_\alp(\cdot |_{\ffp_\alp},\cdot |_{\ffp_\alp})^{\ffp_\alp,\tw},
\end{align*} 
where $\cdot|_{Q}: \sH\to\sH_Q:=H\sta(Q)$ and $\cdot|_{\alp}: \sH\to\sH_{\ffp_\alp}:=H\sta(\ffp_\alp)$ 
are projections, and  
\begin{align*}
 &(x,y)^{Q,\tw}:=\int_Q\frac{   xy}{\prod_{\alp=1}^\n (H+\ft_\alp)}=\int_Q  \frac{xy}   {- \ft^\n},  \qquad\quad\qquad x,y\in \sH_Q \black \\
 &(x,y)^{\ffp_\alp,\tw}:=\frac{5}{   \ft_\alp^4\prod_{\beta:\beta\neq \alp}(\ft_\beta-\ft_\alp)}xy=    \frac{-5}{\n \ft_\alp^3 t^\n }xy, \qquad x,y\in \sH_{\pt\lalp}.
\end{align*}

\smallskip
 
Let $p$ be the equivariant-hyperplane class $(x_1=0)$ in $H^2_G(\PP^{4+\n})$. Then 
$$H\sta_G(\PP^{4+\n}) = \QQ(\ft) [p]/\bl p^5 \prod_\alpha (p+\ft_\alpha)\br,
$$
over which the pullback pairing equals (by applying \eqref{prin})
$$(x,y )^{\tw}   = \int_{\PP^{4+\n}}  x\cup y\cup 5 p, \qquad \forall x,y \in H^*_G(\PP^{4+\n}).
$$
In terms of the basis $\{ p^i\}$ of $\sH$, the pairing is given by (for $0\leq i,j\leq \n+3$)
$$\eta_{ij}:=  ( p^i, p^j)^{\tw}= \int_{\PP^{4+\n}} 5p  \cup p^i\cup p^j  =  \begin{cases}
 5    t^{\n}   \black   \black & \text{if } i+j = 2\n+3 \\
 5   \black \black & \text{if } i+j =  \n+3,   \\
 0 & \text{otherwise }
\end{cases}
$$

Denote by $Q\sub\PP^4\sub (\PP^{4+\n})^G$ the tautological embedding where the latter is given by $x_{i\ge 6}=\cdots=0$. 
Then the restriction gives a ring homomorphism preserving the intersection-pairing
\beq\label{Harrow}
H\sta_G(\PP^{4+\n}) = \QQ(\ft) [p]/\bl p^5 \prod_\alpha (p+\ft_\alpha)\br\lra \sH\uev:=H\uev(\aleph)\subset \sH
\eeq
that send $p$ to $p|_{Q}=H\in H^2(Q)$ and $p|_{\ffp_\alp}=-\ft_\alpha\in H^2(\pt\lalp)$.
Further, the kernel of \eqref{Harrow} is spanned by $p^4 \prod_\alpha (p+\ft_\alpha)$.  
Considering the natural decomposition
$$   \sH=\sH\uev\oplus H^3(Q).
$$
We take the basis $\{\phi_i:=  {p^i} \}_{i=0}^{\n+3}$ of $\sH\uev$; let $\{\phi^{i}\}$ be its dual basis. 
Then $\{\phi^{i}\}\sub \sH^{ev}$, and {\small
$$\{\phi^0,\cdots,\phi^{\n+3}\}= \Bigl\{  \frac{p^3}{5}(p^\n -\ft^\n) ,    \frac{p^2}{5 }(p^\n-\ft^\n) ,  \frac{p}{5 }(p^\n-\ft^\n) ,   \frac{1}{5 }(p^\n- \ft^\n), \frac{p^{\n-1}}{5}, \frac{p^{\n-2}}{5}, \cdots,  \frac{p^{0}}{5}    \Bigr\}.
$$}

\begin{convention}
In the remainder of this paper, we will view $p^i$ as its image in $\sH\uev $ under \eqref{Harrow}.
This way, we have {\small
$$\mathbf 1_\alp = \frac{p^4}{\ft_\alp^4}\prod_{\beta\neq \alp} \frac{\ft_\beta+p}{\ft_\beta-\ft_\alp} \text{  for }   \alpha\in [\n]
\and H^j = \frac{p^j}{\ft^\n}(\ft^\n -p^\n) \text{  for }  0\le j\le 3  \text{ with  }H^0= \mathbf 1_Q.$$}
\end{convention}
Under the twisted pairing, the Poincare dual of $\{1,H,H^2,H^3\}\cup \{1_\alp\}_{\alp\in[\n]}$ is
$$\{\frac{-\ft^\n}{5} H^3,\frac{-\ft^\n}{5} H^2, \frac{-\ft^\n}{5} H, \frac{-\ft^\n}{5} H^0 \}\cup\{  1^{\alpha}:=  
\frac{\n \ft_\alp^3 \ft^\n}{(-5)}1_\alp  \black   \}_{\alp\in[\n]}.
$$

\subsection{Correlators of $\nmsp$ theory and local theory} \label{JandS}
\begin{definition}
We introduce the qunitic twisted classes
\begin{align*}
 [\M_{g,n}(Q,d)]^{\tw}:=\,&\prod_\alp e_T\big(  R\pi\lsta \ev_{n+1}\sta\sO(1)\cdot \ft_\alp \big)^{-1} \cap  [\M_{g,n}(Q,d)]^{\vir} 
\end{align*}
we introduce the $\ffp_\alp$ twisted classes
 \begin{align*}
  [\M_{g,n}]^{\alpha,\tw}:= \,&  (-1)^{1-g} \frac{{5\ft_\alpha} \cdot e_T(\EE_{g,n}\dual\otimes (-\ft_\alpha))^5 }{(-\ft_\alpha)^5 \cdot e_T(\EE_{g,n}\otimes5\ft_\alpha)} 
      \frac{ \prod_{\beta\neq \alpha} e_T( \EE_{g,n}\dual\otimes (\ft_\beta-\ft_\alpha)) }{\prod_{\beta \neq \alpha} (\ft_\beta-\ft_\alpha)} \cap [\M_{g,n}]. 
\end{align*}
We define topological twisted classes $[\M_{g,n}(Q,d)]^{\omega}$ and $[\M_{g,n}]^{\alpha,\omega}$ to be the 
top degree component of corresponding twisted classes.
\end{definition}

By dimension reason, for quintic $Q$, 
\begin{align} 
 [\M_{g,n}(Q,d)]^{\omega}= { (-t^\n)^{-(d+1-g)}}  [\M_{g,n}(Q,d)]^{\vir}  =[\M_{g,n}(Q,d)]^{\tw} ;
\end{align}
for point $\ffp_\alpha$ ($\alpha = 1,\cdots,\n$),
\begin{align} \label{Top-pt} \ [\M_{g,n}]^{\alpha,\omega}=
\Big(  \frac{5}{\n (-t_\alpha)^{3+\n}}  \Big)^{1-g} [\M_{g,n}]
 =\Big(  \frac{-  5}{\n t_\alpha^3 t^\n}  \Big)^{1-g} [\M_{g,n}]. 
\end{align}
For $\tt_i\in \sH$, we define
$$\textstyle \left[ \tau_{\mathbf n}  \right] ^{Q,\tw}_{g,n}= \sum_{d\geq 0} q^d (\pr_{g,n})\lsta  \big(\prod_i \ev_i\sta(\tt_i|_Q)
\cap  [\M_{g,n}(Q,d)]^{\tw} \br ;
$$
$$ \textstyle\left[ \tau_{\mathbf n} \right] ^{\ffp_\alp,\tw}_{g,n}= \prod_{i=1}^n(\tt_i|_{\ffp_\alp}) \cdot  [\M_{g,n}]^{\alpha,\tw}.
$$
Removing ``$\tw$" in the above two lines define $[-]^Q_{g,n}$ and $[-]^{\ffp_\alp}_{g,n}$;
replacing ``$\tw$" by ``$\omega$" defines $[-]^{Q,\omega}_{g,n}$ and $[-]^{\ffp_\alp,\omega}_{g,n}$.    
We define the local classes to be
\begin{align*}
 \left[- \right] \loc_{g,n} = \left[- \right] ^{Q,\tw}_{g,n}+\sum_\alpha  
 \left[ - \right] ^{\ffp_\alp,\tw}_{g,n}
 \in H^*(\M_{g,n}, \FF[\![q]\!]).
\end{align*}
We define the correlators for \footnote{We regard the symbol ``$\ffp_\alp,\tw$" equal to ``$\alp,\tw$".} $\bullet=$``$\lloc$",``$Q,\tw$",``$\ffp_\alp,\tw$",``$Q$",``$\ffp_\alp$",
\begin{align*}
 \left<- \right>^{\bullet}_{g,n} =
 \int_{[\M_{g,n}]} \left[- \right]^{\bullet}_{g,n}  .
\end{align*}

\subsection{The genus zero theory and mirror formula}\label{JS}
For ``$\bullet$" to be ``$M$" or ``$\lloc$",  
$$J ^\bullet (\mathbf t,z) \in \aA \otimes \sH[\![z,z^{-1}]\!]  \and  S_\tau^\bullet(z)  \in \aA \otimes \End \sH  [\![z^{-1}]\!],
$$
with  $\mathbf t   \in \aA \otimes\sH[\![z]\!]$ and $\tau \in  \aA \otimes\sH$  are defined
via
\begin{align}
J^\bullet (\mathbf t,z) :=&\,z+ \mathbf t(-z)+\sum_{\alpha,n} \frac{1}{n!} e^{\alpha} \bigl< \frac{e_\alpha}{z-\psi}  , \mathbf t(\psi)^n \bigr> ^\bullet_{0,n+2} \label{Jfunction}\\
 S_\tau^\bullet (z) x  
:= &\,  x  + \sum_{\alp, n} \frac{1}{n!} e^{\alp}   \bigl< \frac{x}{z-\psi} ,e_\alp, \tau^n
\bigr> ^\bullet_{0,n+2}, \label{Sfunction}
\end{align}
where $\{e_{\alpha}\}$ is a basis of $\sH$, and $ \{e^{\alpha}\}$ is its dual basis
under the pairing $(,)^{\tw}$.
 We remark that     whenever $x\in \sH\uev$,
  $S_0\msp (z) x$ is even, namely $ S_0\msp (z) x \in    \aA  \otimes \sH\uev [\![z,z^{-1}]\!] $.   We brief $S^M(z)=S^M_0(z)$.   


Notice that $J^\bullet (\mathbf t,z)$  is  rational in $z$ whenever $\mathbf t$ is rational in $z$, so is 
$S_\tau^\bullet (z) \tt$   for any $\tt\in\sH$. We make the following very important conventions:

\begin{convention}
In this paper, we will always expand the $\nmsp$ functions $J\msp(\bt,z)$ and $S\msp_\tau(z)$ near $z=0$ when it is  computed by localizations, and we will always expand local functions $J\loc(\bt,z)$  and $S\loc_\tau(z)$ near $z=\infty$. 
\end{convention}

\begin{remark}\label{odd}  Applying localization,
for $x\in H^3(Q)$ one shows that, for $\bullet$ to be ``$M$" or ``\,{\rm loc}", $S^\bullet_\tau(z)x\in H^3(Q)$. This implies $S^\bullet_\tau(z)^{-1}$ preserves $H^3(Q)$ and thus $S^\bullet_\tau(z)$ preserves $\sH\uev$, . 
 When $\tau=0$ and $x\in H^3(Q)$, one has   $S^M(z)x=x$. 
 On the other hand, for $x\in H^3(Q)$, one  has $S\loc_\tau(z)x=x$ whenever $\tau|_Q=y H$ for some $y\in \aA\otimes H\sta(Q)$. 
\end{remark}

Set formally $\sH_{\circ}=\sH_{\circ,\tw}:=H\sta(\circ,\FF)$ for $\circ=Q$ or $\ffp_\alp$.
For $\bullet$ equals ``$Q$" , ``$Q,\tw$" , ``$\ffp_\alp$", or ``$\ffp_\alp,\tw$", \black
and for $\tau \in  \aA \otimes\sH_\bullet$ and $\bt \in \aA \otimes\sH_\bullet[\![z]\!]$, we define 
$$J ^\bullet (\mathbf t,z) \in \aA \otimes \sH_\bullet[\![z,z^{-1}]\!]  \and  
S_\tau^\bullet(z)  \in \aA \otimes \End \sH_\bullet  [\![z^{-1}]\!]
$$
via the same formula  \eqref{Jfunction} and \eqref{Sfunction}, where the dual basis are taken under the respective pairing 
$(,)^\bullet$, where $(,)^Q$ and $(,)^{\ffp_\alp}=(,)^{\ffp}$ are the ordinary Poincare pairings.

One calculates \begin{align}\label{twq'} S^{Q,\tw}_\tau(q,z)=S^Q_\tau(q',z)   \and S_\tau^{\ffp_\alp,\tw}(z) = \black S_\tau^{\ffp_\alp}(z)=e^{\tau/z}. \end{align}   Both match  those defined using $[-]^{Q,\omega}$ and $[-]^{\alp,\tw}$ respectively. \black



Applying principle ($\bigstar$), the $J\msp$, $S\msp_\tau$ are identical to the $J$ and $S$ function of the
genus zero $L_p^{\otimes 5}$-twisted $G$-equivariant theory of $\PP^{4+\n}$. Givental's mirror theorem applies immediately.

\begin{theorem}\label{mirror} ({\cite[{Thm 10.7}]{G96}})
We consider the $\nmsp$  $J$-function
   \begin{align*}
 \qquad J\msp(z) := J\msp(0,z)=&\,   z \mathbf 1+\sum_{\alpha} e_\alpha   \left< \frac{e^{\alpha}}{z-\psi_1} \right>_{0,1}\msp   \end{align*}
and define the $\nmsp$ $I$-function
    \begin{align} \label{nmspI}
  I\msp(q,z)=& z   \mathbf 1 +\sum_{d=1}^\infty z \, q ^d \frac{\prod_{m=1}^{5d}(5p+mz)}{ \prod_{m=1}^d (p+mz)^5 \prod_{m=1}^d \big((p+mz)^\n - \ft^\n \big)}.
 \end{align}
The following mirror theorem holds
\begin{equation*}
 J\msp(z)  = I\msp(q,z).
\end{equation*}
\end{theorem}

\begin{remark} 
By definition $zS\msp(z)\sta 1\,\, =\,  J\msp(z)$.
 \end{remark}

%

\subsection{Givental's Lagrangian cone} \label{birkhoff}  
The Lagrangian cone in GW theory was introduced in \cite{Ba,CG,Giv1}. 
We consider the Lagrangian cone defined by the local theory
$$
\mathcal L\loc:= \{ (\mathbf p,\bq):= \mathbf p(-z)+ \mathbf q(z) \in\sH[\![z,z^{-1}]\!] :  \mathbf p = d_\bq F_0\loc(\bt)  \} ,
$$
where $\bt(z) := \bq(z)+z \in \sH[\![z]\!]$, and 
$$ \textstyle
F_0\loc(\bt):= F_0^{Q,\tw} (\bt|_{Q}) +\sum_{\alp}   F_0^{\pt,\tw}(\bt|_{\ffp_\alp})
$$
is the genus zero twisted descendent potential.
Applying the localization formula for the MSP two point function,   we have
\begin{lemma}\label{cone} For any   $\ep\in \sH\uev \otimes \aA$ and $\tau\in\sH$ 
$$J\msp(\bt,-z)  \in  \mathcal L\loc   \and   zS_{\epsilon}\msp(z) ^{-1} \tt  \in  \mathcal L\loc \cap T_\epsilon\mathcal L\loc.  
$$ 
\end{lemma}
 
\begin{proof} 
 We compute the $J\msp(\bt,-z)$ by localization. This will give us the expansion of the   $J\msp(\bt,-z)$   as a Laurent series at $z=0$. The reason is as follows. Suppose $v$ is the vertex with the insertion in the localization graph. There are two possibilities:
  \begin{enumerate}
\item $v$ is an unstable vertex, i.e. there are at most one more edge adjacent to $v$ and $d_v=0$: then the $\psi$ class at $v$ is invertible, the localization graph will contribute terms with only non-negative powers of $z$ to $J\msp(\bt,-z)$;
\item $v$ is stable vertex, i.e. there are at least two more edges adjacent to $v$ or $d_v>0$: then the $\psi$ class at $v$ is not invertible, the localization graph will contribute terms with only negative powers of $z$ to $J\msp(\bt,-z)$.
\end{enumerate}
To summarize, the non-negative part 
\begin{align*} 
L\loc(\bt, z):=z \mathbf 1+[J\msp(\bt,-z)]_+  \in \sH\otimes \aA[\![z]\!]
\end{align*} is given by the sum of the contributions of the graphs which are tails with the unstable starting vertex $v$; the negative part $[J\msp(\bt,-z)]_-$ is the sum of the contributions of tails with stable starting vertex, which is given by
$$
\sum_{k,\alpha}  e^\alpha\big< L\loc(\bt,\psi)^k,  \frac{e_\alpha}{z-\psi}\big>\loc_{0,1+k}.
$$
Hence we have
\begin{align*} 
\qquad J\msp(\bt,-z) &=- z\mathbf 1+ L\loc(\bt,z)+  \sum_{k,\alpha}  e^\alpha
\big< L\loc(\bt,-\psi)^k,  \frac{e_\alpha}{-z-\psi}\big>\loc_{0,1+k} \\
 &=  J\loc(L\loc,-z)\ \in \ \mathcal L\loc.
\end{align*}
Moreover, for $\tt\in\sH$, $\epsilon := \epsilon(t)\in \sH\uev$ with some parameter $t$ 
(for example $t$ is the flat coordinate), we consider  $zS_\epsilon\msp(z) ^{-1} \tt$. 
Let $\star_\epsilon$ be the quantum product of the $L_p^{\otimes 5}$-twisted GW theory (Principle ($\bigstar$)). Then
 the $S$-function satisfies the QDE (see \cite{CG,LeeP})
$$
zd  S_{\epsilon}\msp(z)    =  d\epsilon *_{\epsilon} S_{\epsilon}\msp(z).
$$
Suppose $\tt $ lies in the subring of the  quantum coholomogy ring 
generated by $\epsilon':=\frac{d}{dt}\epsilon$.
Then it is of the form
$$zS_{\epsilon(t)}\msp(z) ^{-1} \tt  =  \sum_{k} a_k(\tt)z^k \frac{d^k}{d t^k} J\msp(\epsilon(t),-z).   $$
Since the Lagrangian cone $\mathcal L\loc$ is closed under the operation $z\frac{d}{d t} $ (c.f.  \cite{CG}, \cite{Giv3}), it lies in $\mathcal L\loc$.

Notice that by divisor equation, when $\epsilon = t p$ the operator $D_p:=p+zq\frac{d}{dq}$ plays the role of $z\frac{d}{dt}$. And it is clear in our case the hyperplane class $p$ generates even part of the quantum cohomology ring.
At the point $\epsilon=0$, the explicit form of the QDE is given in  Lemma \ref{QDEforMSPS}. 
\end{proof}

By Givental's result on Lagrangian cone  (cf. \cite{Giv1}),  the Lagrangian cone $\mathcal L{\loc}$ consists of a family of linear spaces
$$z S_\tau \loc(z)^{-1} \, \sH\otimes \aA[\![z]\!],
$$
which gives us the Birkhoff factorization, defining us the desired $R\loc $-matrix.

\begin{definition}\label{uoflocal}
There exists a unique
$R_{\tau\loc}\loc(z)  \in \End \sH \otimes \aA[\![z]\!] $ such that
\beq\label{sr}
S_\epsilon\msp(z) = R_{\tau\loc }\loc( z)  S_{\tau\loc }\loc( z),
\eeq
where $\tau\loc = \tau\loc(\epsilon)$ is defined by the Dijkgraaf-Witten map
\beq \label{taulocal}
\tau\loc(\epsilon):=  \sum_{\alpha,n} \frac{1}{n!} e^{\alpha} \big<  {e_\alpha} ,\mathbf 1,L\loc(\epsilon,  -  \black \psi)^n \big>\loc_{0,n+2} \in \sH\otimes \aA.
\eeq
\end{definition}
\vspace{-0.3cm}
  By the localization formula, we have $L\loc(0,   \psi)\in q\sH\otimes \aA[\![\psi]\!]$ and thus $\tau\loc(0)\in q\sH\otimes \aA$.

\begin{remark}  
We will see that,  with the formula of $\tau_Q$ and $\tau_\alpha$ given in \eqref{tauloc0},
\begin{align}\label{tauloc} \tau\loc(0)|_Q=\tau_Q(q'),\qquad  \tau\loc(0)|_{\ffp_\alp} = \tau_\alpha(q'). \end{align}
 See Lemma \ref{DWmapQ} and Corollary \ref{DWmappt} respectively for the proof. \black 

\end{remark}

\begin{convention} \label{conventionloc}
In the remainder of this paper,  we will abbreviate  
$$S\msp_{\epsilon=0},\quad S\loc_{\tau\loc(0)},\quad R\loc_{\tau\loc(0)},\quad S^{Q,\tw}_{\tau_Q(q')},\quad   S^{Q}_{\tau_Q(q)}  ,\quad  S^{\ffp_\alp,\tw}_{\tau_\alp(q') } ,\quad  S^{\ffp_\alp}_{\tau_\alp(q) } $$ 
 to
  $S\msp,S^{Q,\tw}(z), \ S\loc, \ R\loc, \ S^{Q}, \ S^{\ffp_\alp,\tw}, \ S^{\ffp_\alp}$  respectively.
\end{convention}

 By  convention \ref{conventionloc}, we see 
$$S^{Q,\tw}=S^Q |_{q\mapsto q'} \and S^{\ffp_\alp,\tw}(z) =S^{\ffp_\alp}(z) |_{q\mapsto q'}. $$

Since both the local and global $S$-matrices
$$S\msp \in\End(\sH) \otimes \aA[\![z,z^{-1}]\!] \quad \text{   and }  \quad S\loc \in \End(\sH) \otimes \aA[\![z^{-1}]\!]$$ are symplectic (see e.g. \cite{Giv2}), we see that $R\loc \in \End(\sH) \otimes \aA[\![z]\!]$ is symplectic as well.

\begin{remark} \label{oddpart}
 By Remark \ref{odd} and \eqref{tauloc}, \black for any $x\in H^3(Q)$,  $S\msp(z) x = x=S\loc_{\tau\loc(0)}(z)x$. Thus we have 
$R\loc(z) x = x$. The symplectic property then implies $S\msp, S\loc_{\tau\loc(0)}, R\loc$ all preserve $\sH\uev$. 
\end{remark}

\begin{lemma}\label{Define-R} The $R(z)$ defined in \eqref{DefineR}  matches 
 \begin{align*}   
R(z) := \,& R\loc (z)  \cdot  \Big(\id_{\sH_Q} \oplus \diag \{ \Delta^{\ffp_\alp}(z) \}_{\alpha=1}^\n \Big).   \end{align*} 
\end{lemma}
\begin{proof} Brief 
 $S^{\npt,\tw}=\diag \{S^{\ffp_\alpha,\tw}\} , S^{\npt}=\diag \{S^{\ffp_\alpha}\} $ and  $\Delta^{\npt}=\diag \{\Delta^{\ffp_\alp}\}_{\alpha=1}^\n$. By \eqref{sr}  
\begin{align*} 
\qquad S\msp \begin{pmatrix} \Delta^{\npt}  & \\ & 1 \end{pmatrix}
= & \ R\loc   \begin{pmatrix}  S^{\npt,\tw}  & \\ & S^{Q,\tw}   \end{pmatrix}  \begin{pmatrix}    \Delta^{\npt}  & \\ & 1  \end{pmatrix}\\
= & \ R\loc   \begin{pmatrix}  \Delta^{\npt}    & \\ & 1    \end{pmatrix}
  \begin{pmatrix}  S^{\npt,\tw}  & \\ &  S^{Q,\tw}  \end{pmatrix}
= R  \begin{pmatrix} S^{\npt,\tw}   & \\ & S^{Q,\tw}  \end{pmatrix},\qquad\quad
\end{align*}
since $\Delta^{\pt_\alp}$ commutes with   $S^{\pt_\alp,\tw} =  e^{\tau_\alp(q')/z} $.
\end{proof}

\vspace{1cm}

\section{Bipartie graphs and  $\nmsp$ generating functions}
\label{localizationgraph}

Our goal in this section is to show that the contribution of a $\nmsp$ localization graph can be obtained by
combining its contributions from ``$(1,\infty]$ parts" with  those from its ``$[0,1]$ parts". 
We use bipartie graphs  to organize such decompositions.

A {\sl bipartite} graph is a connected graph whose vertices are colored white or black, 
and the two vertices of each edge have different colors.

\begin{definition}\label{bip}
A decorated bipartite graph (referred to in this paper) is a bipartite graph with decorations:\\ 
{\bf vertices}: each vertex $v$ is decorated by an integer $g_v\ge 0$; in addition each black vertex $v$ is decorated by 
$d_{\infty [v]}\in \ZZ$; \\ 
{\bf edges}: each edge $e$ is decorated by an integer $a_e>0$
 ; and has its hour $\alpha_e\in [\n]$;\\
{\bf legs}: all legs are incident to white vertices;\\
{\bf constraint}: each white vertex $v$ having $g_v=0$ must have $|L_v|+|E_v|\geq 2$.
\end{definition}

For a decorated bipartite graph $\Lam$, we denote by $V_w(\Lam)$ (resp. $V_b(\Lam)$) its white (resp.
black) vertices. We define its genus $g=h^1(\bipg)+\sum_{v\in V(\Lam)} {g_v}$, degree $d_\infty =\sum_{v\in V_\b(\bipg)} d_{\infty [v]}$, and $n$ the number of its legs.

The bipartie graphs we are working with has one additional
constraint:  each black vertex $v$ must satisfy
\beq\label{dinf-bound}d_{\infty [v]}\geq \frac{1}{5}\bl 2-2g_v+\sum_{e\in E_v}(a_e-1)\br. \eeq
 
We let $\Xi_{g,n}^r$ be the set of all decorated bipartite graphs of genus $g$, degree $r=d_\infty$, $n$ legs, 
and satisfying \eqref{dinf-bound}. It is direct to check that, for given $g$, $r$ and $n$, the set $\Xi_{g,n}^{r}$ is finite.

For each bipartie graph $\Lambda$ we can perform the standard ``stablization" operation described in 
Appendix \ref{stablization}.\footnote{When we stabilize a bipartie graph, we treat it as a dual graph, with only genus
assignments considered.} The process induces a map $cl:L(\Lambda)\to V(\Lambda)$, 
sending each leg to its associated stable-graph-vertex of $\Lambda$ (cf. Appendix \ref{stablization}).
For each vertex $v$ of $\Lambda$, we define $\LL_v:=\{  l \in L(\Lambda), cl(l)=v\}$. 
Then $\LL_v\ne\emptyset$ is possible only if $v$ is a stable-graph-vertex.

 We now make sense of the ancestor classes in the localization formula. Like in \cite{CLLL}, we denote by
$F^S$ be the set of flags in $\Lam$: $(e,v)\in F^S$ if $e$ is an edge incident to $v$. 
As usual, we denote by $\psi_{(e,v)}$ the psi-class associated with the mentioned flag. 

Let $\Lambda$ be a bipartie graph; let $l$ be a leg of $\Lam$, incident to the vertex $v$.
The ``stablization" of $\Lam$ will make 
$l$ incident to $v^\circ_l=cl(l)\in V(\Lam\ust)$; as $v^\circ_l$ corresponds to a
stable vertex in $\Lam$, we view $v^\circ_l\in V(\Lam)$ as well. 
 There are two cases. In the first case,  $l$ is a leg incident to  $v^\circ_l$, we set $c(l)=l$.  In the second case, there is a unique pure subchain of
$\Lam$ connecting $v_l$ and $v^\circ_l$; let $e^\circ_l$ be the edge in this subchain that is incident to $v^\circ_l$;
we set $c(l)=(e^\circ_l,v^\circ_l)$. Then in both cases $\bar\psi_{c(l)}$ is the ancestor class that can be inserted at the vertex $cl(l)$.


%
%
%


\begin{theorem}\label{summation}
Suppose  each $\tt_i=\tt_i(z)$ is even. 
We have the following formula  that expresses the total $\nmsp$ correlators via
$\nmsp$-$[0,1]$ correlators and some $q$-polynomials:
\begin{align*} 
\left< \tt_1\bar \psi_1^{k_1}  ,\cdots,   \tt_n \bar \psi_n^{k_n} \right>\msp_{g,n,r}= & \sum_{\bipg\in \Xi^{r}_{g,n}}   \frac{1}{|\Aut \bipg |} 
\prod_{v\in V_b(\bipg) } \Cont^{\infty}_{[v]}  (\prod_{i\in\LL_v} \bar\psi_{c(i)}^{k_i}) \black
  \\
&   \prod_{v\in V_w(\bipg)} \Big\langle \bigotimes_{i\in L_v}   \tt _i        \prod_{i\in\LL_v}\bar\psi_{c(i)}^{k_i} \black    \bigotimes_{e\in E_v,f=(e,v)} \frac{ 1^{\alp(e),\tw}}{  \frac{5\ft_{\alp_e}}{a_e}-\psi_{f}  }     \Big\rangle^{[0,1]}_{g_v,n_v }. \nonumber
\end{align*} 
Here 
for each $v\in V_b(\Lam)$, its contribution is a $q$-polynomial with degree bound
$$ \Cont^{\infty}_{[v]}   (\prod_{i\in L_v^\circ}  \bar\psi_{c(i)}^{k_i})     \in  \QQ(\ft)[q],\quad \deg_q \le 
d_{\infty[v]}+ \frac{1}{5}\bl 2g_{v}-2 - \sum_{e\in E_v}   (a_e-1) \br
$$
and 
 for any $x,y\in\sH$ the unstable correlators are defined by
 \begin{align*}
\left< \frac{x}{z_1-\psi_1}, \frac{y}{z_2-\psi_2} \right>^{[0,1]}_{0,2} := \ &\frac{(x,y)^\tw}{z_1+z_2} +  \left< \frac{x}{z_1-\psi_1},\frac{y}{z_2-\psi_2}\right>_{0,2}\msp;
\\
\left< \frac{x}{z_1-\psi_1}, y\right>^{[0,1]}_{0,2}:= & \ \big(x , y \big)^\tw +\left< \frac{x}{z_1-\psi_1},y\right>_{0,2}\msp.  
\end{align*}
\end{theorem}


\subsection{From localization graphs to bipartite graphs}\label{Bipartite}

To each regular decorated graph $\locg \in \gnd\ureg$ (with $\bd=(d_0, d_\infty)$), we associate a decorated bipartie graph as follows. 

Let $v\in V_\infty(\Theta)$. A restricted-tail of $\locg$ rooted at $v$ is a rational tail of $\locg$ rooted at $v$ so that
the only vertex of the tail that lies in $V_\infty(\Theta)$ is the root $v$. To get the bipartie graph $\Lambda(\Theta)$,
we first remove all restricted-tails of $\Theta$ rooted at vertices in $V_\infty(\Theta)$,
resulting a graph $\Theta'$. We then collapse each maximal connected 
subgraph of $\Theta'$ supported\footnote{A graph is supported in $[0,1]$ (resp. $\infty$) if all its vertices has level $0$ or $1$ (resp. $\infty$).} in $[0,1]$ (resp. $\infty$) to a white vertex (resp. a black vertex);
we decorate the resulting vertices by the total genera of 
the subgraphs collapsed. These vertices form the vertices of the bipartie graph $\Lam(\Theta)$ we are constructing.

The edges of $\Lam(\Theta)$ is the same as $E_{1\infty}(\Theta')$, with the incidence relation the induced 
one. The legs of $\Lam(\Theta)$ is the same as $L(\Theta)$, with the incidence relation the induced one.
(If $l\in L(\Theta)$ is incident to $v\in V(\Theta)$ and $v$ lies in the subgraph collapsed to $\bar v\in V(\Lambda(\Theta))$, 
then in $\Lambda(\Theta)$ $l$ is incident to $\bar v$.)
This way, each edge is incident to one white and one black vertex, giving us a bipartie graph $\Lam(\Theta)$. 

For an edge $e$ of $\Lam(\Theta)$, we assign its hour $\alpha_e$ to be the same hour $\alpha_e$ by viewing
$e\in E_{1\infty}(\Theta)$, along the way we assign $a_e:=-5d_e\in \NN$. 
To a black vertex $v$ of $\Lam(\Theta))$, letting $\Theta_v\sub\Theta$ be the subgraph (supported in $\infty$)
contracts to $v$, we assign (cf. Definition \ref{bip})
\beq\label{[v]}
d_{\infty [v]}:=d_{\infty \Theta_v}+\sum  d_{\infty e},
\eeq
where the summation is over all edges $e\in E_{1\infty}(\Theta)$ incident to $\Theta_v$. 
One checks that the inequality \eqref{dinf-bound} holds true (cf. \eqref{estimate0}). This completes the construction of the
bipartie graph $\Lam(\Theta) \in \Xi_{g,n}^{d_\infty}$ associated to $\Theta$.

 \begin{example}\label{graph0} From a localization graph $\locg$ to a decorated bipartite graph $\bipg_\locg$ (we omit the decorations $d_0,d_\infty$ here). \black
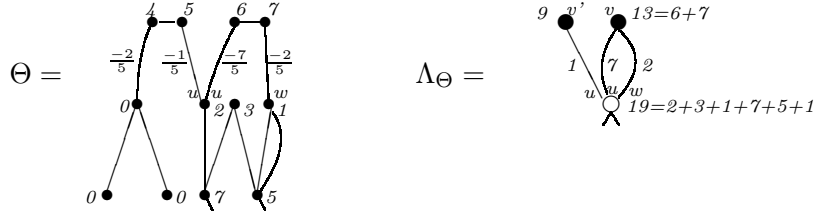
\vspace{-.5cm}
 \begin{center}
 \begin{picture}(50,29)
         \put(-30,12){$\locg=$}

   \put(0,20){\circle*{1.3}}
 \put(0,21){\tiny{6}}

  \put(-7,20){\circle*{1.3}}
  \put(-10,20){\line(1,0){2}}
  \put(-7,21){\tiny{5}}
    \put(-12,21){\tiny{4}}
   \put(-7,19.5){\line(1,-4){2.5}}

\put(4.5,9.2){\circle*{1.3}}
 \put(5.5,7.5){\tiny{1}}

\put(0,9.2){\circle*{1.3}}
   \put(0,8.5){\line(1,-4){2.9}}
   \put(-6.5,10){\tiny{u}}
\put(1,7.5){\tiny{3}}
\put(-4,9.2){\circle*{1.3}}
\put(-3,7.5){\tiny{2}}

\put(4,20){\circle*{1.3}}
\put(4.2,21){\tiny{7}}
\put(4,20){\line(-1,0){4}}
 \qbezier(4,19.2)(4.2,15)(4.5,9.2)
  \qbezier(0,19.5)(-2.5,15)(-4,9.2)
\put(4,14){\tiny{$\frac{-2}{5}$}}
 \put(5,10){\tiny{w}}
 \put(-2,14){\tiny{$\frac{-7}{5}$}}
 \put(-3.3,10){\tiny{u}}
 \put(-10,14){\tiny{$\frac{-1}{5}$}}

  \put(3,-3){\circle*{1.3}}
 \put(4,-4){\tiny{5}}
   \qbezier(3,-3)(3.5,-4)(4,-5)
 \put(3,-2.5){\line(1,6){2}}
    \qbezier(3,-3)(8,4)(5,8)

  \put(-4,-3){\circle*{1.3}}
    \put(-3,-4){\tiny{7}}
       \qbezier(-4,-3)(-3.5,-4)(-3,-5)
 \put(-4,-2.5){\line(1,3){4}}
  \put(-4,-2.5){\line(0,3){11}}

 \put(-13,9.2){\circle*{1.3}}
 \qbezier(-13,9.2)(-13,16)(-11,20)
\put(-17,14.2){\tiny{$\frac{-2}{5}$}}
\put(-11,20){\circle*{1.3}}
 \put(-15.5,8){\tiny{0}}

  \put(-9,-3){\circle*{1.3}}
    \put(-8,-4){\tiny{0}}
 \put(-9,-2.5){\line(-1,3){4}}

  \put(-17,-3){\circle*{1.3}}
    \put(-20,-4){\tiny{0}}
 \put(-17,-2.5){\line(1,3){4}}

      \put(24,12){$\bipg_\locg=$}

   \put(51,20){\circle*{2}}
      \put(49,21){\tiny{v}}
 \put(52.5,20.5){\tiny{13=6+7}}
\put(50,9.2){\circle{2}}
 \put(52,8){\tiny{19=2+3+1+7+5+1}}
        \qbezier(50,8)(49.5,7.2)(49,6.4)
        \qbezier(50,8)(50.5,7.2)(51,6.4)

 \qbezier(51,19)(55.5,15)(51,10)
  \qbezier(51,19.5)(47.5,15)(49.5,10.5)
\put(54,13.5){\tiny{2}}
  \put(52.2,10.2){\tiny{w}}
 \put(49.2,13.5){\tiny{7}}
  \put(49.3,10.5){\tiny{u}}
 \put(44,13.5){\tiny{1}}
 \put(46.3,10){\tiny{u}}
   \put(44,20){\circle*{2}}
    \put(44,21){\tiny{v'}}
   \put(40,20.5){\tiny{9}}
    \put(44,19.5){\line(1,-2){4.8}}

\end{picture}
\end{center}
\vspace{.5cm}
 where the integer next to a
vertex (an edge) is its genus (its $a_e$), and each vertex of $\Theta$ is indexed, including the indices $u,w\in \{1,\cdots, \n\}$.
\end{example}

\begin{example} \label{examplegraph}
We list  decorated bipartite graphs of some specified $(g,n,d_\infty)$, where
in the first three Figures each vertex has zero $d_\infty$ and is decorated by genus, and  every edge is decorated by some ``hour" $\alp\in\{1,\cdots,\n\}$ which we omit in the graphs. \black  

\noindent$\blacktriangleright$ {Figure: The list of all $(g,n,d_\infty)=(1,1,0)$ bipartite graphs.}
\medskip
 
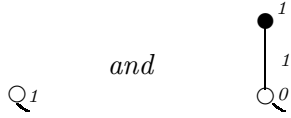
\begin{center}
 \begin{picture}(20,24)
\put(-26,13.3){\circle{2}}
 \put(-26,12.3){ \tiny{1}}
   \qbezier(-26,12)(-25.5,11.5)(-24.5,11)

\put(-14, 16){and}

 \put(7,24){ \tiny{1}}
 \put(7,23){\circle*{2}}
\put(7.5,17){ \tiny{1}}
  \put(7,22){\line(0,-1){8}}
\put(7,13.1){\circle{2}}
\put(8.5,12.5){\tiny{0}}
  \qbezier(8,12)(8.5,11.5)(9.5,11)
\end{picture}
\end{center}
\vspace{-1.2cm}

\noindent$\blacktriangleright$ {Figure: The list of all  $(g,n,d_\infty)=(2,0,0)$ bipartite graphs. }
\medskip
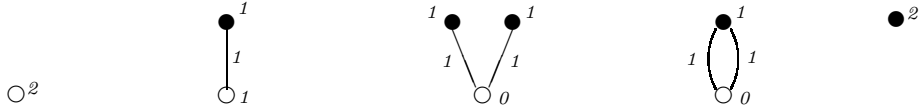
\begin{center}
 \begin{picture}(50,25)

\put(-36,13.3){\circle{2}}
 \put(-36,13.3){ \tiny{2}}

 \put(-8,24){  \tiny{1}}
 \put(-8,23){\circle*{2}}
\put(-7.5,17.5){\tiny{1}  }
  \put(-8,22){\line(0,-1){8}}
\put(-8,13.2){\circle{2}}
\put(-6.5,12.2){\tiny{1}  }

 \put(22,23){\circle*{2}}
 \put(18.5,23.5){\tiny{1} }
\put(20.5,17){\tiny{1} }
  \put(22,22){\line(2,-5){3.1}}

\put(30,23){\circle*{2}}
\put(31.5,23.5){\tiny{1} }
\put(29.5,17){\tiny{1} }
 \put(30,22){\line(-2,-5){3.1}}
\put(26,13.2){\circle{2}}
 \put(28,12){\tiny{0} }

 \put(58,23){\circle*{2}}
 \put(59.5,23.5){\tiny{1}}
\put(58,13.2){\circle{2}}
 \put(60,12){\tiny{0}}
 \qbezier(59,22)(61,18)(59,14)
  \qbezier(57,22)(55,18)(57,14)
\put(61,17.5){\tiny{1}}
 \put(53,17.3){\tiny{1}}

 \put(81,23.3){\circle*{2}}
 \put(81,23.3){ \tiny{2}}
  \end{picture}
\end{center}
\vspace{-1.cm}
\noindent$\blacktriangleright$ {Figure: Some bipartite graphs of $(g,n,d_\infty)=(10,0,0)$, where the numbers  in the top row denotes the $d_\infty$ of each vertex right below.}

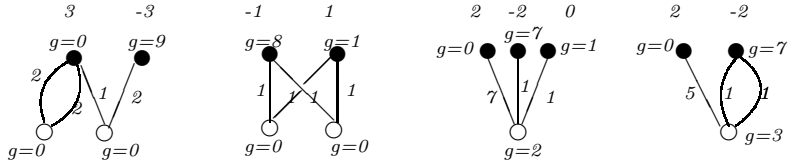
\begin{center}
 \begin{picture}(50,0)

 \put(-30,-14.8){\circle{2}}
 \put(-35,-18){\tiny{g=0}}
 \put(-26.5,-12.5){\tiny{2}}
  \put(-32,-8){\tiny{2}}
  \qbezier(-26,-5)(-23.5,-11)(-29,-14)
  \qbezier(-26,-5)(-32,-8)(-30,-13.5)
   \put(-22,-15){\circle{2}}
\put(-22.5,-18){\tiny{g=0}}
\put(-23,-10.5){\tiny{1}}
 \put(-22,-13.8){\line(-2,5){3.1}}
\put(-26,-5){\circle*{2}}
 \put(-29.5,-3.5){\tiny{g=0}}
  \put(-27.5,0.5){\tiny{3}}
\put(-17,-5){\circle*{2}}
 \put(-18,-5.8){\line(-2,-5){3.4}}
\put(-19,-3.5){\tiny{g=9}}
\put(-18,0.5){\tiny{-3}}
\put(-18.6,-11){\tiny{2}}

     \put(-3.5,0.5){\tiny{-1}}
     \put(7,0.5){\tiny{1}}

 \put(0,-14.3){\circle{2}}
 \put(-3.5,-17.5){\tiny{g=0}}
  \put(0,-4.5){\circle*{2}}
  \put(0,-5.5){\line(0,-1){8}}
  \put(-2,-10){\tiny{1}}
    \put(-3.5,-3.5){\tiny{g=8}}
\put(2,-11){\tiny{1}}
  \put(0.5,-13){\line(1,1){3.7}}
    \put(4.5,-8.5){\line(1,1){3.7}}
\put(8.5,-14.5){\circle{2}}
\put(8,-17.5){\tiny{g=0}}
\put(5,-11){\tiny{1}}
 \put(8.5,-13.3){\line(-1,1){9}}
\put(9,-4.5){\circle*{2}}
\put(9,-5.5){\line(0,-1){8}}
\put(10,-10){\tiny{1}}
 \put(7,-3.5){\tiny{g=1}}

 \put(26.5,0.5){\tiny{2}}
    \put(31.5,0.5){\tiny{-2}}
\put(39,0.5){\tiny{0}}

 \put(29,-4){\circle*{2}}
 \put(22,-4.5){\tiny{g=0}}
\put(28.5,-11){\tiny{7}}
  \put(29,-5){\line(2,-5){3.5}}
  \put(33,-4){\circle*{2}}
    \put(31,-2.5){\tiny{g=7}}
  \put(33,-5){\line(0,-5){8.7}}
  \put(33.1,-9.5){\tiny{1}}
\put(37,-4){\circle*{2}}
\put(38.5,-4){\tiny{g=1}}
\put(36.5,-11){\tiny{1}}
 \put(37,-5){\line(-2,-5){3.5}}
\put(33,-14.8){\circle{2}}
 \put(31,-18){\tiny{g=2}}

 \put(53,0.5){\tiny{2}}
 \put(61,0.5){\tiny{-2}}
 
   \put(62,-4){\circle*{2}}
 \put(63.5,-4.5){\tiny{g=7}}
\put(61,-14.8){\circle{2}}
 \put(63,-16){\tiny{g=3}}
 \qbezier(62,-5)(58.5,-9)(61,-14)
  \qbezier(62,-4.5)(68.5,-9)(62,-14)
\put(65,-10.5){\tiny{1}}
 \put(60.2,-10.5){\tiny{1}}
 \put(65,-10.5){\tiny{1}}
   \put(55,-4){\circle*{2}}
   \put(48.6,-4){\tiny{g=0}}
   \put(55,-4.5){\line(1,-2){4.8}}
 \put(55,-10.5){\tiny{5}}

\end{picture}
\end{center}
 
\end{example}
 





\vskip2cm
 
\subsection{The contribution from a black vertex}
We now construct the contribution from a black vertex to $\nmsp$ correlators. We first fix the notation we will be using.


Let $\bipg\in\Xi_{g,\ell}^r$; let $v\in V_b(\bipg)$, with $E_v=\{e_1,\cdots e_n\}$.
We let $[v]$ be a bipartie graph with one black vertex $v$, $n$ edges $e_1,\cdots,e_n$, $n$ 
genus 0 white vertices $v_1,\cdots ,v_n$, and $n$ $(1,\rho)$-legs $l_1,\cdots, l_n$, so that
each $v_i$ is incident to $e_i$ and each $l_i$ is incident to $v_i$. 
We then set the degree $d_{\infty [v]}$ to be the same $d_{\infty[v]}$ when viewing $v$ as the vertex in
$\Lambda$. 
We set $\alpha_i:=\alpha_{e_i}$, which is the hour of $e_i$, the hour $e_i$ as an edge in $\bipg$. 
This way, $[v]\in \Xi_{g_v,n}^{d_{\infty [v]}}$.

For the $[v]$ of the shape just described, we say a regular decorated graph $\Theta$
strongly contracts to $[v]$ if {\sl $\Lam(\Theta)\cong [v]$, and for each $i$ the subgraph of $\Theta$ that is collapsed 
to $v_i$ (in $[v]$) has total $d_0$-degree zero.}\footnote{If we denote the subgraph collapsed to $v_i$ by $\Theta_{v_i}$, then it is a one vertex no edge graph.}
We let $B_{[v]}$ be the set of all regular graphs $\Theta$ strongly contracting to $[v]$.
We define
\begin{align} \label{infinity-cont} 
\Cont^\infty_{[v]}:=\sum_{d\geq 0} (-1)^{d+1-g } q^d \sum_{\Theta\in B_{[v]}}  \int_{[F_\Theta]\virt}  
\frac{1}{e(N\virt_{\Theta})}  
\prod_{i=1}^n   \ev_i\sta 1_{\alpha_i}  \in \aA,\end{align}  
to be the sum of contributions to $ \left<1_{\alpha_1},\cdots,1_{\alpha_n}\right>\msp_{g_v,|E_v|,d_{\infty[v]}}$ from all
$\Theta\in B_{[v]}$.

\begin{example}\label{graph0-1} 



Following Example \ref{graph0}, the contributions of the graph $\Theta$ can be partitioned into a product of 
the contributions of three subgraphs. The smaller black dots are unstable vertices; $u$ and $w$ are hours.
 
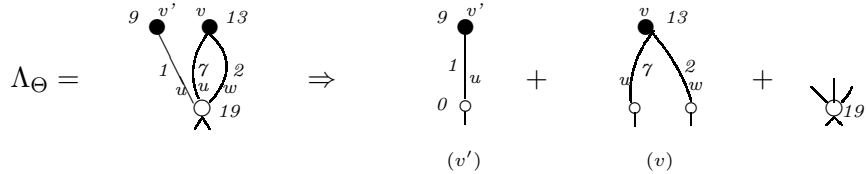
\begin{center}
 \begin{picture}(50,29)
         \put(-37,8.5){$\locg=$}

   \put(-10,20){\circle*{1.3}}
 \put(-10,21){\tiny{6}}

  \put(-17,20){\circle*{1.3}}
  \put(-20,20){\line(1,0){2}}
  \put(-17,21){\tiny{5}}
    \put(-22,21){\tiny{4}}
   \put(-17,19.5){\line(1,-4){2.5}}

\put(-5,9.2){\circle*{1.3}}
 \put(-4.5,7.5){\tiny{1}}

\put(-10,9.2){\circle*{1.3}}
   \put(-10,8.5){\line(1,-4){2.9}}
   \put(-16.5,10){\tiny{u}}
\put(-9,7.5){\tiny{3}}
\put(-14,9.2){\circle*{1.3}}
\put(-13,7.5){\tiny{2}}

\put(-6,20){\circle*{1.3}}
\put(-5.8,21){\tiny{7}}
\put(-6,20){\line(-1,0){4}}
 \qbezier(-6,19.2)(-5.8,15)(-5.5,9.2)
  \qbezier(-10,19.5)(-12.5,15)(-14,9.2)
\put(-6,14){\tiny{$\frac{-2}{5}$}}
 \put(-5,10){\tiny{w}}
 \put(-12,14){\tiny{$\frac{-7}{5}$}}
 \put(-13.3,10){\tiny{u}}
 \put(-20,14){\tiny{$\frac{-1}{5}$}}

  \put(-7,-3){\circle*{1.3}}
 \put(-6,-4){\tiny{5}}
   \qbezier(-7,-3)(-6.5,-4)(-6,-5)
 \put(-7,-2.5){\line(1,6){2}}
    \qbezier(-7,-3)(-2,4)(-5,8)

  \put(-14,-3){\circle*{1.3}}
    \put(-13,-4){\tiny{7}}
       \qbezier(-14,-3)(-13.5,-4)(-13,-5)
 \put(-14,-2.5){\line(1,3){4}}
  \put(-14,-2.5){\line(0,3){11}}

 \put(-23,9.2){\circle*{1.3}}
 \qbezier(-23,9.2)(-23,16)(-21,20)
\put(-27,14.2){\tiny{$\frac{-2}{5}$}}
\put(-21,20){\circle*{1.3}}
 \put(-25.5,8){\tiny{0}}

  \put(-19,-3){\circle*{1.3}}
    \put(-18,-4){\tiny{0}}
 \put(-19,-2.5){\line(-1,3){4}}

  \put(-27,-3){\circle*{1.3}}
    \put(-30,-4){\tiny{0}}
 \put(-27,-2.5){\line(1,3){4}}


  \put(1,8.5){$\Rightarrow$}

 \put(19,20){\circle*{1.3}}
  \put(16,20){\line(1,0){2}}
  \put(19,21){\tiny{5}}
    \put(14,21){\tiny{4}}
   \put(19,19.5){\line(1,-4){2.5}}

   \put(19,10){\tiny{u}}
   \put(22,9.2){\circle*{1}}
 \put(23,7.5){\tiny{0}}
\put(22,9){\line(0,-1){1.9}}

 \put(16,14){\tiny{$\frac{-1}{5}$}}

 \put(13,9.2){\circle*{1.3}}
 \qbezier(13,9.2)(13,16)(15,20)
\put(9,14.2){\tiny{$\frac{-2}{5}$}}
\put(15,20){\circle*{1.3}}
 \put(10.5,8){\tiny{0}}

  \put(17,-3){\circle*{1.3}}
    \put(18,-4){\tiny{0}}
 \put(17,-2.5){\line(-1,3){4}}

  \put(9,-3){\circle*{1.3}}
    \put(6,-4){\tiny{0}}
 \put(9,-2.5){\line(1,3){4}}


 \put(31,8.5){$+$}

\put(45,20){\circle*{1.3}}
\put(45,21){\tiny{6}}

\put(49,20){\circle*{1.3}}
\put(49.2,21){\tiny{7}}
\put(49,20){\line(-1,0){4}}
 \qbezier(49,19.2)(49.2,15)(49.5,9.2)
  \qbezier(45,19.5)(42.5,15)(41,9.2)
\put(49,14){\tiny{$\frac{-2}{5}$}}
 \put(50,10){\tiny{w}}
 \put(43,14){\tiny{$\frac{-7}{5}$}}
 \put(41.7,10){\tiny{u}}
 \put(41,9){\circle*{1}}
 \put(49.5,9){\circle*{1}}
\put(41,9){\line(0,-1){2}}
\put(49.5,9){\line(0,-1){2}}

 \put(50.5,7.5){\tiny{0}}
\put(42,7.5){\tiny{0}}


\put(60,8.5){$+$}

\put(80,9.2){\circle*{1.3}}
 \put(81.5,7.5){\tiny{1}}

\put(75,9.2){\circle*{1.3}}
   \put(75,8.5){\line(1,-4){2.9}}
   \put(72,12){\line(-1,-4){1}}
      \put(70,12){\line(1,-4){1}}
\put(76,7.5){\tiny{3}}
\put(71,9.2){\circle*{1.3}}
\put(72,7.5){\tiny{2}}
 \put(80,10){\line(0,1){2}}

  \put(78,-3){\circle*{1.3}}
 \put(79,-4){\tiny{5}}
   \qbezier(78,-3)(78.5,-4)(79,-5)
 \put(78,-2.5){\line(1,6){2}}
    \qbezier(78,-3)(83,4)(80,8)

  \put(71,-3){\circle*{1.3}}
    \put(72,-4){\tiny{7}}
       \qbezier(71,-3)(71.5,-4)(72,-5)
 \put(71,-2.5){\line(1,3){4}}
  \put(71,-2.5){\line(0,3){11}}



\end{picture}
\end{center}
\vskip10pt

\vspace{5mm} 

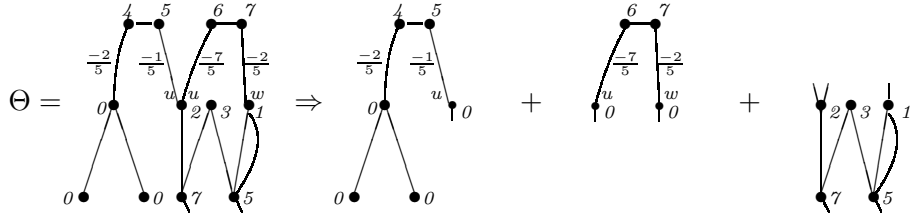
\noindent$\blacktriangleright$ {Figure: The R.H.S is the partition of $\Lambda(\Theta)$ into three parts:}
\vskip-15pt

\begin{center}
 \begin{picture}(50,29)

      \put(-38.5,12){$\bipg_\locg=$}

   \put(-12,20){\circle*{2}}
      \put(-14,21.5){\tiny{v}}
 \put(-10,20.5){\tiny{13}}
\put(-13,9.2){\circle{2}}
 \put(-11,8){\tiny{19}}
        \qbezier(-13,8)(-13.5,7.2)(-14,6.4)
        \qbezier(-13,8)(-12.5,7.2)(-12,6.4)

 \qbezier(-12,19)(-7.5,15)(-12,10)
  \qbezier(-12,19.5)(-15.5,15)(-13.5,10.5)
\put(-9,13.5){\tiny{2}}
  \put(-10.4,11.2){\tiny{w}}
 \put(-13.8,13.5){\tiny{7}}
  \put(-13.7,11.5){\tiny{u}}
 \put(-19,13.5){\tiny{1}}
 \put(-16.7,11){\tiny{u}}
   \put(-19,20){\circle*{2}}
    \put(-19,21.5){\tiny{v'}}
   \put(-23,20.5){\tiny{9}}
    \put(-19,19.5){\line(1,-2){4.8}}

    \put(1,12){$\Rightarrow$}


   \put(19.5,1.5){\tiny{$(v')$}}

     \put(22,20){\circle*{2}}
    \put(22,21.5){\tiny{v'}}
   \put(18,20.5){\tiny{9}}
    \put(22,19.5){\line(0,-1){9}}
    \put(19.5,14.2){\tiny{1}}
    \put(22,9.5){\circle{1.5}}
    \put(18,9){\tiny{0}}
    \put(22,8.8){\line(0,-1){1.7}}
    \put(22.4,12.5){\tiny{u}}


  \put(30,12){$+$}

     \put(46,1.5){\tiny{$(v)$}}

   \put(46,20){\circle*{2}}
      \put(45,21.5){\tiny{v}}
 \put(48.5,20.5){\tiny{13}}
\put(44.5,9.2){\circle{1.5}}
 \put(52,9.2){\circle{1.5}}
 \put(44.5,8.6){\line(0,-1){1.9}}
  \put(52,8.6){\line(0,-1){1.7}}

 \qbezier(47,19)(50.5,15)(52,10.2)
  \qbezier(47,19.5)(43.5,15)(44,10)
\put(51,14){\tiny{2}}
  \put(51.5,11.6){\tiny{w}}
 \put(45.2,13.5){\tiny{7}}
  \put(42.3,12.3){\tiny{u}}

  \put(60,12){$+$}

\put(71,9.2){\circle{2}}
 \put(72,8){\tiny{19}}


 \qbezier(70,10)(69.5,10.5)(68,12)
 \qbezier(71,10.1)(71,10.5)(71,13)
 \qbezier(72,10)(72.7,10.5)(73.4,12)

 \qbezier(71,8)(70.5,7.5)(70,7)
 \qbezier(71,8)(71.5,7.5)(72,7)

\end{picture}
\end{center}

 \vspace{1mm}

\end{example}

To proceed, we define $\Theta_\infty$ for any $\Theta\in G\ureg\lggd$. Indeed, the vertices of $\Theta_\infty$ 
is the set $V_\infty(\Theta)$, its edges are edges in $E_{\infty\infty}(\Theta)$, and its legs are elements
in 
$$E_{1\infty}(\Theta)\cup\{l\in L(\Theta)\mid \text{$l$ incident to a $v\in V_\infty(\Theta)$}\},
$$
with incidence relations the induced one. For $l_e\in L(\Theta_\infty)$ associated to $e
\in E_{1\infty}(\Theta)$, we assign its monodromy to be $\zeta_5^{w_e}$ satisfying $w_e+5d_{\infty e}\equiv 0(5)$.
Because $\Theta$ is regular, we can choose $w_e\in [1,4]$.

%
%
%

We now let $\Gamma\in G\ureg_{g,\gamma,(d_0',d_\infty')}$ so that $\Ga_\infty=\Gamma$. 
Let $l_1,\cdots,l_n$ be the legs of $\ga$, of (narrow) monodromy assignments $\gamma=
(\zeta_5^{w_1},\cdots, \zeta_5^{w_n})$,
with $\bw=(w_1,\cdots,w_n)$ and $w_i\in [1,4]$. To emphasize its dependence on $\Ga$, we write
$(d_0',d_\infty')=(d_{0\ga},d_{\infty\ga})$, and write $g_\ga$ the total genus of $\ga$. Then
$\Ga\in G\ureg_{g_\ga,\gamma,(d_{0\ga},d_{\infty\ga})}$.
We form $\cW_{(\ga)}$, which is the image of 
$$\cW_\ga\to\cW_{g_\ga,\gamma,(d_{0\ga},d_{\infty\ga})}.
$$
Over $\cW_{(\ga)}$ we denote the coarse-psi classes by  $\psi_i$, and its ancestors by $\bar\psi_i$,
which are pullback of the $i$-th psi class via the forgetful map $ \cW_{(\ga)} \to \barM_{g_\ga,n}$. 
For $(c_1,\cdots,c_n)\in (\ZZ_{\geq 0})^n$ and $\vz=(z_1,\cdots,z_n)$, 
we define the dual twisted FJRW $n$-point function
\begin{align*} F^{LG,\tw}_{\ga}(\vz;\prod_{i=1}^n \bar\psi_i^{c_i})   
  &:=  q^{d_{0\ga}} \int_{[\cW_{(\ga)}]\virt} \frac{1}{e(N\virt_{\ga})} \prod_{i=1}^m \frac{\bar\psi_i^{c_i}}{ z_i - \psi_i/5}.
\end{align*} 
Let $\alpha_i=\alpha_{v_i}$ be the hour of the vertex $v_i$ to which $l_i$ is incident; 
let $\va=(a_1,\cdots,a_n)\in (\ZZ_+)^n$ be such that $a_i\equiv w_i (5)$.
We define
\begin{align}\label{mtgamu} \mT_{\ga,\va}(\prod_{i=1}^n \bar\psi_i^{c_i}) := (-1)^{d_{0\ga}+1-g_\ga} \cdot F^{LG,\tw}_{\ga}(\frac{-\ft_{\alp_1}}{a_1},\cdots,\frac{-\ft_{\alp_n}}{a_n};\prod_{i=1}^n \bar\psi_i^{c_i}).
  \end{align}
In case all $c_i=0$, we denote \eqref{mtgamu} by $\mT_{\ga,\va}$.



%
%
%
%

\subsection{The specialized $S$-function}
We introduce the specialized $S$-functions:
\begin{align}\label{msa}
\mS_a^\alp:=   S\msp(z) 1^\alpha   \big|_{  z= \frac{5\ft_\alpha}{a}}, \and  \mS^\alp_{a;i}:= ( \mS^\alp_{a},  p^i)^\tw.
\end{align}

We agree $(a)_k:=a \cdot (a-1)\cdot \cdots\cdot  (a-k+1)$.

\begin{lemma} \label{specialS}
We have the identity
\begin{align*}
 & \mS^\alpha_{a;0}  =  \,     1 +\sum_{  d=1  }^{ \lceil a/5\rceil-1}  q ^d  \frac{(a-1)_{5d}}{ \big( (\frac{a}{5}-1 )_d \big)^5  }\frac{ \big( {a}/ {5\ft}\big)^{\n d} }{\prod_{m=1}^d \big((-\frac{a}{5}+m )^\n - (\frac{a}{5})^\n \big)}.
 \end{align*}
\end{lemma}
  
\begin{proof}  By definition of $ \mS^\alpha_{a;0}$ and the relation $J(z) =zS\msp(z)^* \mathbf 1$, we see
$$   \frac{5\ft_\alp}{a} \cdot  \mS^\alp_{a;0}  
 = J\msp(0,\frac{5\ft_\alp}{a})|_{\ffp_\alp}= \Big(   \frac{5\ft_\alp}{a} 1_\alp+ \black \left<   \frac{1^{\alpha}}{5\ft_\alp/a-\psi} \right>\msp_{0,1}  1_\alp \Big).
$$
Together with the Mirror theorem \ref{mirror}, we obtain
\begin{align*}
 & \frac{5\ft_\alp}{a}\cdot    (\mS^\alp_{a} ,\mathbf 1 )^\tw= \,   I\msp(q,z) \big|_{z=\frac{5\ft_\alp}{a}, \, p=-\ft_\alp} .
\end{align*}
A direct computation then proves this lemma.
\end{proof}

\begin{corollary}\label{coro-mS} The following properties hold
\begin{enumerate}
\item For any $\alpha$, we have $\mS_{1;0}^\alpha=1$ and 
$ \left<\frac{\mathbf 1^{\alpha}}{5\ft_{\alp}-\psi}\right>\msp_{0,1}=0$;
\item For any $\alpha$, $a$ and $i$, we have
$\mS^\alpha_{a;i} = \zeta^{i (\alpha-\beta)} \cdot \mS^\beta_{a;i}$;
\item  For any $\alpha$, $a$ and $i$,   $\mS^\alpha_{a;i}  \in  \QQ(\zeta_\n) \black [q/\ft^\n] $
is a polynomial in $q$ and
\beq \label{degofmS}
\deg_q  \mS^\alpha_{a;i}    \leq \begin{cases}
 \lceil a/5\rceil-1 & \text{ if } i< \n,\\
 \lceil a/5\rceil & \text{ if } i\geq \n .
 \end{cases}
\eeq
\end{enumerate}
\end{corollary}



\begin{proof} Item (1) is a direct consequence of Lemma \ref{specialS}. We now prove (2) and (3).
By the explicit QDE for the $S$-matrix that we will compute later (see Lemma \ref{QDEforMSPS}), we have
\beq \label{mSrec}
\mS^\alpha_{a;i} =\begin{cases}
 \Big(-\ft_\alpha+\frac{5\ft_\alpha}{a} q \frac{d}{dq} \Big) \mS^\alpha_{a;i-1} & \text{ if  }  i<\n, \\
\Big(-\ft_\alpha+\frac{5\ft_\alpha}{a} q \frac{d}{dq} \Big) \mS^\alpha_{a;i-1} - c_{i} \, q  \cdot\mS^\alpha_{a;i-\n}  & \text{ if  }  i\geq \n,
\end{cases}\quad
\eeq
where $(c_{j})_{j=\n,\cdots,\n+3} =  (120,770,1345,770)$. We see that (2) and (3) follow from inductions via \eqref{mSrec}, 
with initial conditions given by Lemma \ref{specialS}.
\end{proof}

\subsection{Proof of Theorem \ref{summation} } 
We first  look at the case  $k_1=\cdots=k_n=0$.  By the localization formula and applying \cite[Sect.\,6]{NMSP1},
we see that the decomposition of $\Theta$ into its $[0,1]$ part and $(1,\infty]$ part (c.f. Example \ref{graph0-1})  
is consistent with their localization contributions. 
  
Let $\Theta\in B_{[v]}$; it has the shape given by Figure 1. 
Here the infinity line represents  
the part of $\Theta$ lies at $\infty$. 
The edges $E_{1\infty}(\Theta)$ are divided into two
types: Type-I are edges $e_1,\cdots,e_m$ so that the vertex $v_i\in V_1(\Theta)$ incident to $e_i$ is unstable and has 
one leg $l_i$ incident to it; Type-I\!I are edges $e'_{1},\cdots,e'_{\ell}$ so that the maximal connected
subgraph $\Ga_{j}$ of $\Theta_{[0,1]}$ attached to $e'_{j}$ has total genus zero and no legs.

\vspace{-5mm}
 \begin{figure}[H]{\small
 \begin{center}
\begin{picture}(20,30)
\put(-26,23){\circle*{1.3}}
\put(-30,25){$\psi_1$}
 \put(-26,22){\line(0,-1){14.5}}
\put(-28,1.5){$\ti\psi_1$}
\put(-26,7){\circle*{1.5}}
\put(-30,7){$l_1$}
\put(-15,7){$l_m$}
\put(-25.5,12){$e_1$}

 \put(-19,15){$\cdots$}

\put(-10,23){\circle*{1.3}}
\put(-9.5,12){$e_m$}
\put(-14,25){$\psi_m$}
\put(-10,7){\circle*{1.5}}
\put(-12,1.5){$\ti\psi_m$}
 \put(-10,22){\line(0,-1){14.5}}

 \put(15,23){\circle*{1.3}}
\put(23,15){$\cdots$}
\put(13.5,3.7){$\ga_1$}
\put(11,25){$\psi_{m+1}$}
 \put(15,22){\line(0,-1){14}}
\put(15.5,12){$e'_{1}$}

\put(15,5){\circle{6}}

\put(34,23){\circle*{1.3}}
\put(32.5,3.6){$\ga_\ell$}
\put(34.5,12){$e'_{\ell}$}
\put(32,25){$\psi_{m+\ell}$}
 \put(34,22){\line(0,-1){14}}
\put(34,5){\circle{6}}

\black

 \put(-40,23){\line(1,0){85}}
\put(47,23){$ \infty$ }





\end{picture}
\end{center}

\vspace{-5mm}

\caption[1]{  }
}
\end{figure}

For $s=1,\cdots,m$ and $i=1,\cdots,\ell$, we  denote
  $$d_{\infty e_s}=\frac{a_s}{5}, \ \  \alp_{e_s}=\alp_s  \and d_{\infty e'_{i}}=\frac{b_i}{5}, \ \  \alp_{e'_{i}}=\beta_i.$$ 
 Then $\va=(a_1,\cdots,a_m)$,  $\vb=(b_1,\cdots,b_\ell)$  are sequences of integers. We denote \\
 $(a_1,\cdots,a_m,b_1,\cdots,b_\ell)$ by $(\va,\vb)$.



   We first make the following simplification. In case $b_j=1$, by Lemma \ref{specialS},
then the summation of the contributes of all possible $[0,1]$ tails $\ga_j$ is
$$\left<\frac{1^{\alp_j}}{5\ft_{\alp}-\psi}\right>\msp_{0,1}= 5\ft_{\alp_j} \big((\mS^{\alp_j}_{a=1},1^{\alpha})^\tw  -  1\big)1_{\alp}=0.
$$  
Thus from now on we assume $b_i\ge 2$ for all $i=1, \cdots \ell$.
 
We claim that 
\begin{align}\label{d-inf-black} 
d_{\infty[v]} = d_{0\Theta_\infty} + \frac{1}{5}(\sum_{r=1}^m a_r+\sum_{i=1}^\ell b_i)- \frac{1}{5}(2g_\ga-2+ m+\ell).
\end{align}
Indeed, because for any $(\sC,\Sigma^\sC,\cdots)\in\cW_{\Theta}$, we have
$\sL^{\otimes 5}\cong \omega_{\sC}^{\log}$ when restricted to $\sC_a$ for $a\in V(\Theta_\infty)\cup E(\Theta_\infty)$, thus
$d_{\Theta_\infty}=\frac{1}{5}(2g_\ga-2+ m+\ell)$. As $d_{\Theta_\infty}=d_{0\Theta_\infty}-d_{\infty\Theta_\infty}$,
\eqref{[v]} implies \eqref{d-inf-black}. 
As a consequence,
\begin{align}\label{estimate0}
0\leq \deg_q    \mT_{\ga,(\va,\vb)}=d_{0,\ga} = d_{\infty[v]}  + \frac{2g_\ga-2}{5}  -\frac{\sum_{r=1}^m  (a_r-1)}{5}. 
\end{align}

Let $\Cont([v],\Ga, \vb)$ be the sum of contributions to \eqref{infinity-cont} from all such $\Theta$'s with  prescribed $\Ga=\Theta_\infty$ and $\vb$ in Figure 1. Then we have 
\beq\label{Cont-sum}  \Cont^\infty_{[v]}=   \sum_{ \Ga, \vb } \Cont([v],\Ga, \vb) ,
\eeq
where the sum is over all $\Ga$, with the prescibed $g_\ga$, $\valp$, $\va$, $ d_{\infty[v]}$, and with varying $\vb$.
It is a finite sum since $d_{0\ga}$ and all $b_i-1$ are positive, and thus bounded using \eqref{d-inf-black}. 
   
\smallskip
We now calculate $\Cont([v],\Ga, \vb)$ according to the following subcases.

\smallskip
\noindent $\blacktriangleright$ {\sl Case 1: $(g_\ga,m+\ell,d_{0,\ga})\neq (0,2,0),(0,1,0)$}.

Recall the edge/flags contributions in \cite[ Sect.\,7\black]{NMSP1} and \cite[Lemm 4.5]{CLLL} give
$\ti A_{e_s}$ and $\ti A_{e''_s}$ as below. For any edge $e\in E^\alp_{\infty1}$, let $ u=d_e\in \frac{-1}{5}\NN$, and set {\small
$$\ti A_e:= \frac{\prod_{j=1}^{  \lceil-u\rceil-1}(-\ft_\alp-\frac{j\ft_\alp}{u} )^5  }{\prod_{j=1}^{-5u}(-\frac{j\ft_\alp}{u})
\prod_{j=1}^{\lfloor -u \rfloor}(\frac{j\ft_\alp}{u})  } ,
$$
$$c_{\va,\vb}(\valp,\vbeta):=\frac{1}{(\prod a_s)(\prod b_i)} \cdot (\prod_{s=1}^m \ti A_{e_s})(\prod_{i=1}^\ell \ti A_{e''_i}).
$$ } 

We denote   $|\vb|:=b_1+\cdots+b_\ell$. 
Applying the localization formula we have
\beq \label{blackpoly} {\small
     \begin{aligned}   
     & \Cont([v],\Ga, \vb)  
      =   (-1)^{d_{0,\ga}+1-g_\ga}   \frac{ c_{\va,\vb}(\valp,\vbeta)}{|\Aut \vb\,| }  \prod_{s=1}^m (5\ft_{\alp_s}) \black    \bigg[ \prod_{i=1}^\ell\Big( (5\ft_{\beta_i}) J^{msp}(0,\frac{5\ft_{\beta_i}}{b_i})|_{\ffp_{\beta_i}}  \Big) \bigg]   \cdot \\
     &  F^{LG,\tw}_{\ga}\big( \frac{-\ft_{\alp_1}}{a_1},\cdots,\frac{-\ft_{\alp_m}}{a_m},\frac{-\ft_{\beta_1}}{b_1},\cdots,\frac{-\ft_{\beta_\ell}}{b_\ell}\big) 
     =      \frac{c_{\va,\vb} (\valp,\vbeta)}{|\Aut \vb\,| }    \prod_{s=1}^m (5\ft_{\alp_s}) \black    \Big[ \prod_{i=1}^\ell (5\ft_{\beta_i}) (\frac{5\ft_{\beta_i}}{b_i}) \mS^{\beta_i}_{b_i;0}     \Big]       \mT_{\ga,\va,\vb}.
\end{aligned}}
\eeq
By Corollary \ref{coro-mS},  each $ \mS^{\beta_i}_{b_i;0}  $ is a  polynomial in ${q}/{\ft}$,  independent of $\beta_i$.

\smallskip
\noindent $\blacktriangleright$ {\sl Case 2: $(g_\ga,m+\ell,d_{0,\ga})= (0,2,0)$ or $(0,1,0)$}.

Notice the case $(m,\ell)=(0,2)$ is negligible. Using localization formula of \cite{CLLL,NMSP1} 
we obtain their contributions as follows. 

\smallskip
\begin{itemize}
\item \noindent { $(m,\ell)=(2,0)$}:
    The only nontrivial case has  $5\nmid a_1$,  $a_1+a_2\equiv 0(mod\, 5)$. Its contribution is calculated to be
    \begin{align*}   \frac{\ti A_{e_1}\ti A_{e_2}}{a_1a_2} \frac{1}{-\frac{\ft_\alp}{a_{1}}-\frac{\ft_\alp}{a_{2}}} (5\ft_\alp)^2 \cdot \frac{-5\ft_\alp^6}{\prod_{j\neq \alp}(\ft_j-\ft_\alp)} \black .\end{align*}
\item\noindent { $(m,\ell)=(1,1)$}:
The only nontrivial case is when $5\nmid a_1$ and $a_1+b_1\equiv 0(5)$.  
Its contribution is calculated to be
\begin{align*}
    &  (5\ft_\alp)   \frac{\ti A_{e_1}}{a_1}   \frac{1}{- \frac{\ft_\alp}{a_1}-\frac{\ft_\alp}{b_1}}  \frac{\ti A_{e''_1}}{b_1}  \cdot   \frac{-5\ft_\alp^6}{\prod_{j\neq \alp}(\ft_j-\ft_\alp)}    \black \cdot \bigg[    \frac{(5\ft_\alp)^2}{b_1}\cdot  \mS^\alp_{b_1;0}      \bigg].
\end{align*}

\item \noindent { $(m,\ell)=(1,0)$}:
  The only notrivial case is $n:=\frac{a_1}{5}\in \NN$. The contribution    is  calculated to be
   \begin{align*} 
    (5\ft_\alp)       \frac{5\ft_\alp}{ 1-a_1}   \frac{\ti A'_{e_1}}{a_1-1}  \frac{1}{\prod_{   j\neq \beta }(\ft_j-\ft_\alp)},  \black  \text{with} \ 
   \ti  A'_{e_1}=\frac{ \prod_{j=1}^{n-1}(-\ft_\alp-\frac{5j\ft_\alp}{-5n+1})^5 } { \prod_{j=1}^{5n-1}\frac{-5j\ft_\alp}{-5n+1} \prod_{j=1}^n \frac{5j\ft_\alp}{-5n+1}  }.
   \end{align*} 

\end{itemize}

\smallskip
\smallskip
 

\begin{lemma} The contribution
$$\Cont([v],\Ga, \vb)  \in A[\![q]\!]
$$
is a polynomial in $q$ whose degree is bounded by $d_{\infty [v]}+ \frac{1}{5}\bl 2g_v-2  - \sum_{i=1}^\ell  (a_i-1)\br$.
\end{lemma}
\begin{proof}
Following \eqref{blackpoly}, we have $\deg_q    \mT_{\ga,\va,\vb}=d_{0,\ga}$ and 
\begin{align*}
\sum_{i=1}^\ell \deg \mS^{\beta_i}_{b_i;0} &\,=\sum_{i=1}^\ell ( {\lceil b_i/5\rceil-1} ) \leq \frac{1}{5} \sum_{i=1}^\ell  (b_i-1)
\\ &\, = d_{\infty [v]} -d_{0,\ga}+ \frac{2g_v-2 }{5} - \frac{\sum_{i=1}^m  (a_i-1) }{5}.
\end{align*}
This proves the lemma.
\end{proof}

Then Theorem \ref{summation}(2), in case $k_1=\cdots=k_n=0$, follows from \eqref{Cont-sum}.
For case with nonzero $k_i$, 
one needs to integrate each ancestor $\bar\psi_i^{k_i}$ over the virtual cycle at the vertex $v^\circ_l=cl(l)$, regarded as the ancestor of the leg/flag $c(i)$ on $v^\circ_l$.  This amounts to add 
$ \prod_{i\in\LL_v}\bar\psi_{c(i)}^{k_i} $ to  each biparted vertex's contribution.   
 
As the term $\mT_{\ga,\va,\vb}$ in \eqref{blackpoly} is substituted 
by  $\mT_{\ga,\va,\vb}(\prod_{\substack{i=1,\cdots,n\\ cl(i)=v}}\bar\psi_{c(i)}^{k_i})$, whose $q$-degree being still 
$d_{0,\ga}$(c.f.\eqref{mtgamu}) implies the same degree bound in above lemma. This completes the proof of Theorem \ref{summation}.


\vspace{1cm}

\section{Proof of Theorem \ref{Omega01}}
In this section, we will prove Theorem \ref{R-action-thm}.
We begin with providing a stable graph description of $[0,1]$-class in descendent form.

\subsection{Stable graphs from localization graphs}
Let $\tau_i(z)\in\sH[z]$. Recall that
\begin{align}\label{des-01class}    \bigl[\tt_1 (\psi_1 ),\cdots,\tt_n (\psi_n )\bigr]^{[0,1]}_{g,n} \end{align}
is the sum of contributions of all $\nmsp$ localization $[0,1]$-graphs $\Theta$ of $n$ markings and genus $g$  (see Definition \ref{01theory-def}) 
{ 
$$\sum_{\locg\in G^{[0,1]}_{g,n,(d,0)}}  \Contr_{\locg} \big( \tau_{\mathbf n} (\psi_{\mathbf n})\big).
$$
Here (cf. \eqref{01})
$$\Contr_{\locg} \big(\tau_{\mathbf n}(\psi_{\mathbf n}) \big)
 := \sum_{d\ge 0}  (-1)^{d+1-g } q^d 
\bl \pr_{g,n}\br\lsta \Bigl(\prod_{i=1}^n \ev_i\sta \tt_i(\psi_i)\cdot  \Cont_\Theta \Bigr).
$$
}


Given a connected localization  $[0,1]$-graph $\locg$, we can {\sl stablize} it
(c.f. Appendix \ref{stablization}) to get a stable graph $\ga=\Theta\ust$, together with decoration $g_v$ for $v\in V(\locg)$, and 
its level $\lev_v \in\{0,1\}$ and $F_v\in\{Q\} \cup \{\ffp_\alp\}_{\alp=1}^\n$.  
Here since each $v\in \Theta\ust$ corresponds to a unique stable vertex, denoted by the same $v$, in $\Theta$, 
the level of $v$ is the same level of $v$, and $F_v=\{Q\}$ when $\lev_v=0$, and $F_v=\{\pt\lalp\}$ when
$\lev_v=1$ and $\alpha_v=\alpha$. Then $\ga\in G_{g,n}^\n$ by the definition of $G_{g,n}^\n$ above Theorem \ref{R-action-thm}.

%
%

\subsection{Tail contribution via $\nmsp$ $J$-function } \label{tailcont}
The tail contributions naturally appear at each stable vertex by the localization formula. For example,
among all decorated graphs appearing in the localization formula calculating
genus $g$ no marking $\nmsp$ correlator
$$
F\msp(q):=\left<  \right>\msp _{g,0} ,
$$
Then there is a class of graphs, call ``leading" graphs, each of which has a single genus $g$ vertex $v$ with ``tails" 
attached to it. We will view such graph as a vertex with legs so that each leg is then replaced
by a tail. 
By the argument in Lemma \ref{cone},  
the total contribution of all possible tails attached to a leg $l$ is ($\psi_l$ is the psi classes of the marking $l$)
$$L\loc(\psi_l) |_{F_v}= \psi_l +  [ J\msp(-\psi_l) ]_+  \big| _{F_v}    \in  \sH_{F_v}[\![\psi_l]\!].
$$

\subsection{Chain contributions via two-point functions } \label{chaincont}
  We introduce the following two-point function:
\begin{align*}
 W\msp(z_1,z_2) =  
 \frac{\sum_\alpha e_\alpha \otimes e^{\alpha}}{-z_1-z_2}+\sum_{\alpha,\beta} e_\alpha\otimes e_\beta   \Bigl< \frac{e^{\alpha}}{-z_1-\psi_1},\frac{e^{\beta}}{-z_2-\psi_2}\Bigr>_{0,2}\msp.
 \end{align*}
By Principle ($\bigstar$) it is equal to the two point function of $L_p^{\otimes 5}$ twisted $\PP^{4+\n}$ theory. 
 From the string and WDVV equations, a standard argument shows  %
\begin{lemma} We have
\begin{equation} \label{twopoinnt}
W\msp(z_1,{z_2}) =  -\sum_\alpha \frac{ S\msp(z_1)^{-1}e_\alpha \otimes S\msp({z_2})^{-1}e^{\alpha}}{z_1+{z_2}}.
\end{equation}
  \end{lemma}

The two point $\nmsp$ correlator can be computed by using $\nmsp$ localization as a graph sum, and we can see that each localization graph is a chain that connects two (localization) vertices $v_1$ and $v_2$ (which could be the same one).

At each vertex $v_i$ ($i=1,2$), there are two types of graph contributions:
\begin{enumerate}
\item when the vertex $v_i$ is unstable,  $\psi_i$ is invertible and the graph will contribute to $W\msp(z_1,z_2)$ of 
terms with non-negative power of $z_i$;
\item when the vertex $v_i$ is stable,  $\psi_i$ is not invertible and the graph will contribute to $W\msp(z_1,z_2)$ 
of terms with negative power of $z_i$.
\end{enumerate}

These contributions can be computed by expanding $W\msp(z_1,z_2)$ as a Laurent series of $z_i$ (at $z_i=0$), 
and taking the part of non-negative or positive powers of $z_i$. 

On the other hand, for each $z_i$ we can expand $W\msp(z_1,z_2)$ as a power series of $z_i^{-1}$, namely expanding
at $z_i=\infty$. The coefficient of $e_\alpha z_i^{-k-1}$ corresponds to the correlator with insertion 
$e^\alpha \psi_i^k$ at $v_i$. Recall that the notation   $f(z^-)$  means  we expand $f(z)$ near $z=\infty$.  

\smallskip
The following two situations will be used consistently.  
\begin{itemize}
\item[(a)] {\sl An edge in the stabilization}: Let $e$ be an edge in the graph after  stablization, with vertices $v$ and $v'$ incident to it.
We consider the two side truncation
\begin{equation}\label{edge324} [W\msp({ \psi_{(e,v)},\psi_{(e,v')}})]_{+,+}   \in \sH_{v}\otimes \sH_{v'}
[\![\psi_{(e,v)},\psi_{(e,v')}]\!],
\end{equation}
as a bi-vector insertion at the edge $e$. It gives total contributions from all possible chains 
that contracts to $e$. 
\item[(b)] {\sl A leg }: Let $l$ be a leg connected via a chain $e$ (in the localization graph) to a stable graph vertex $v$. 
{ As in Appendix B we call such chain an ``end".}
We can consider the class with one side truncation
\begin{equation} \label{legcont}
 \Res_{z=0}\Big([W\msp({ \psi_{v}},-z^-)]_{+} ,\tt(z)\Big)^{\tw} \in \sH_{v}[\![\psi_{v}]\!],
\end{equation} 
as an insertion at the vertex $v$. It gives the total contributions from all possible ``ends" between $v$ and $l$. 
Here for $\psi_l$ the descendent at $l$, $\tt(\psi_l)$ is the insertion at the leg $l$; and $\psi_{v}$ denotes the descendent of $v$ at the flag $(e,v)$. 
\end{itemize}

  By using  \eqref{twopoinnt}, the contribution \eqref{legcont} can be computed to be
  \begin{align}\label{A10}
&\Res_{z=0}\Big(\sum_\alpha \frac{ S\msp(\psi_v)^{-1}e_\alpha \otimes S\msp(-z^-)^{-1}e^\alpha}{\psi_v-z} \Big|_{+},\tt(z)\Big)^{\tw}   \\
&\qquad\,=  [S\msp(\psi_v)^{-1} [(S\msp(-\psi_v^{-})^{-1})^* \tt(\psi_v) ]_+]_+ \ 
 =   [S\msp(\psi_v)^{-1} [ S\msp(\psi_v^-) \tt(\psi_v) ]_+]_+. \nonumber
\end{align}
Especially, when the insertion $\tt$ does not contain $\psi_v$-class, it is
$$
[S\msp(\psi_v)^{-1} \tt]_+.
$$

\subsection{Stable graphs contribution in descendents}\label{by-des}
By Section \ref{tailcont} and \ref{chaincont}, the contribution from a stable $[0,1]$-graphs $\ga$ to
\eqref{des-01class}, which we denote by { $\Cont_{\hga}$}, 
is given by the following construction:
\begin{enumerate}
\item   at each vertex $v$ of $\ga$, we place a linear map
$$
\bigotimes_{j=1}^n \tt_j(z_j)  \longmapsto  \sum_{l\geq 0}\frac{1 }{l !} (\pr_l)_* \Big[\prod_{j=1}^n \tt_j(z_j)  , \prod_{i=1}^{l}  L\loc(\psi_{n_v+i})\Big]^{F_v,\tw}_{g_v,n_v+l},
$$
where $\displaystyle L\loc(z):=z (1-S\msp(z)^{-1})\mathbf 1 |_{+} \in \sH[\![z]\!]$;
 \footnote{The infinite summation is convergent in the $q$-adic topology. See remark \ref{convergenceL}.}
\item  at each edge of $\ga$, we place a bi-vector valued series(viewing $z=\psi_{(e,v)},w=\psi_{(e,v')}$) 
$$
\qquad \qquad \sum_\alpha \frac{ S\loc(z) ^{-1} e_\alpha \otimes S\loc(w)^{-1} e^\alpha}{z+w}
  -\frac{  S\msp(z)^{-1} e_\alpha \otimes  S\msp(w)^{-1} e^\alpha}{z+w}  \Big|_{+,+};
$$
{
this is the contribution of chains in 
the localization graph that stabilize to $e$:
the second term comes from \eqref{twopoinnt} and \eqref{edge324},   and the first term comes from the case that $v=v'$ is the chain itself, which is not allowed;}
\item at each leg $l$ incident to $v$ of $\ga$, we place a vector valued series
$$[S\msp(\psi_v)^{-1} [ S\msp(\psi_v^-) \tt_i(\psi_v) ]_+]_+ \in {\sH\otimes \aA[\![\psi_{v}]\!]};
$$
this is from \eqref{A10}, the total contribution of chains connecting $v$ and the leg $l$, with insertions $\tt_i$ at $l$ in the localization graph. (Note that unlike (2), the ``length" of this chain can be $0$.) \end{enumerate}

This way the $\frac{1}{|\Aut_{\hga}|}(\xi_\hga)_* \Cont_\hga$ is the sum of   $\Contr_\Theta(\tau_1(\psi_1),\cdots,\tau_n(\psi_n))$ over all $\Theta$ with $\Theta\ust=\ga$.  We conclude that
$$
 \sum_{d\geq 0} q^d  (\pr_{g,n})_* \Bigl(\prod_{j=1}^n \tt_j  \cap \bl [ {\mathcal W}_{g,n,(d,0)} ]\virt\br^{[0,1]} \Bigr)
= \sum_{\hga}  \frac{1}{|\Aut_{\hga}|}(\xi_{\hga})_*{\Cont_{ \hga}}.
$$
Note here that the automorphisms of the tails are part of the definition of $\Cont_{\locg}$, hence only $|\Aut \hga|$ 
is in the identity.

\begin{remark}
The argument here is essential the idea of Givental \cite{Giv3}, and is close to the treatment in \cite{CGT}.
\end{remark}

\subsection{Stable graph contribution in ancestors}
To prove Theorem \ref{R-action-thm}, we need to convert descendent classes to
an ancestors classes $\bar\psi_i  \in H^*(\M_{g_v,n_v})$, at stable-graph-vertex $v$. 

{  Let $\{\tt_i(\psi)\}_{i=1}^{n_v}$ be the insertions to $v$ from the two types of  chains \eqref{edge324} and \eqref{A10}. Here $n_v$ is the valence of $v$ in the graph after stablization.  By Kontsevich-Manin's descendent-ancestor formula \cite{KM}, 
\begin{align} \label{Tlocclass}
& \, \sum_{s\geq 0} \frac{1}{s!} (\pi_s)_* 
\big[ \bigotimes_{i=1}^{n_v}\tt_i(\psi_i), L\loc(\psi)^{s} \big]^{F_v,\tw}_{g_v,n_v+s} \\ \nonumber
  =&\,\sum_{k,l\geq 0} \frac{1}{k!\,l!} (\pi_{k+l})_* \big[ \bigotimes_{i=1}^{n_v}\tt_i(\psi_i), 
  (L\loc(\psi)-u)^k,u^l\big]^{F_v,\tw}_{g_v,n_v+k+l} , \qquad u=\tau\loc(0) \nonumber
\\ 
 =&\,\sum_{k,l\geq 0} \frac{1}{k!\,l!} (\pi_{k+l})_*  \big[ \bigotimes_{i=1}^{n_v} S_u\loc(\bar \psi_i  ) \tt_i(\bar\psi_i)_+,
T_u\loc(\bar\psi)^k  ,u^l\big]^{F_v,\tw}_{g_v,n_v+k+l}, \nonumber
\end{align} 
}
where  $\psi_i$'s are the psi-class of   $\M_{g_v,n_v+k+l}(F_v,d)$'s; 
$\bar \psi $'s are the ancestor classes pullback from $\M_{g_v,n_v+k}$. Here we have used
\begin{align} \label{Tloc}
T_u\loc(z)  :=  &\,S_u\loc(z) (L\loc(z)-u)_+  \noindent\\= &\, S_u\loc(z)\Big(z (1-S\msp(z)^{-1})\mathbf 1 \big|_{+}  -u\Big) = z(1-R_u\loc(z)^{-1}) \mathbf 1. \nonumber
\end{align}
 Here we choose the coordinate   $u=\tau\loc(0)$ (c.f. Definition \ref{uoflocal} and \eqref{tauloc}), and as in Convention \ref{conventionloc} we omit the subscript in $S\loc$,$R\loc$ and $T\loc$ with this choice of $u$.

Since $T\loc(z)$ has no $z^{n\le 0}$ terms, namely $z | T\loc(z)$, the formula still holds if we replace the first  $n_v$. $\bar \psi_i$'s by the ancestor classes $\bar{\bar{\psi}}_i$'s from $\M_{g_v,n_v}$, and \eqref{Tlocclass} is equal to
$$ \,\sum_{k,l\geq 0} \frac{1}{k!\,l!} (\pi_{k+l})_* 
\Big[ \bigotimes_{i=1}^{n_v} S\loc(\bar {\bar\psi}_i) \tt_i(\bar {\bar\psi}_i)_+,  
T\loc(\bar\psi)^k  ,u^l\Big]^{F_v,\tw}_{g_v,n_v+k+l}.
$$


\begin{remark}\label{convergenceL}
Note that by the localization formula, each term in $L\loc$ has positive $q$-degree. Since $u = \tau\loc(0)$ also has the same property (c.f. Definition \ref{uoflocal}),   \black
$T\loc$ also only has positive $q$-degree. 
Hence for fixed $d$, there are only finite terms that contribute to the coefficient of $q^d$ in the infinite sum. 
This implies that the infinite sum   is well-defined. 
\end{remark}

By using the Birkhoff factorization (cf. \eqref{sr}) 
$$
S\msp(z)
 = R_{u }\loc( z)  S_{ u }\loc( z),
$$  
we see the contribution of a stable graph
${\hga}$
is given by the following construction which can be realized as an $R$-matrix action:
\begin{enumerate}
\item at each vertex $v$, we place a linear map
\begin{align}\label{vertex-assign} 
\otimes_{j=1}^{n_v} \tt_j(z_j)  \longmapsto  \sum_{k,l\geq 0} \frac{1}{k!l!} (\pi_{k+l})_* 
\Big[\bigotimes_{j=1}^{n_v}\tt_j(\bar{\bar \psi}_i) , T\loc(\bar \psi)^k  ,u^l\Big]^{F_v,\tw}_{g_v,n_v+k+l} \end{align}

\item  at each edge $e$ with $E_v=\{v_1,v_2\}$, we place  
a bi-vector valued series 
\begin{align*}
\qquad V(z,w) :=& \sum_\alpha  \frac{ e_\alpha \otimes   e^\alpha-  R\loc(z)^{-1} e_\alpha \otimes  R\loc(w)^{-1} e^\alpha}{z+w}   \in (\sH  \otimes \sH)\otimes \aA  [\![z,w]\!]
\end{align*}
\item at each leg $l$ incident to $v$, we place a  vector valued series 
$$ R\loc (z_v)^{-1} \big[ S\msp (z_v^-) \tt_i(z_v)  \big] _{+} \in \sH\otimes \aA[\![z_v]\!].$$
 \end{enumerate}

Recall that Grothendieck-Riemann-Roch theorem is used by Mumford to express the twisted class $[-]^{\ffp_\alp,\tw}$ in terms of its degree $0$ component  $[-]^{\ffp_\alp,\omega}$ (c.f. \eqref{Top-pt}).   
{ And Mumford-Faber-Pandharipande's formula  (c.f. \cite[Sect.\,2.2]{Giv3}) 
can be applied to conclude} that the above graph sum remains the same if we do the following: \\
(i) at edge and leg,
replace all $R\loc$ {by (cf. Lemma \ref{Define-R})} \footnote{Here $\Delta^{\ffp_\alp}$ are defined in \eqref{deltapt};  they can be viewed as the $R$ matrices of Grothendieck-Riemann-Roch formula reducing $[-]^{\ffp_\alp,\tw}$ to its topological part $[-]^{\ffp,\omega}$.}
\begin{align} \label{rmatrixdelta}
R(z) = \, R\loc (z)  \cdot  \Big(\id_{\sH_Q} \oplus \diag \{ \Delta^{\ffp_\alp}(z) \}_{\alpha=1}^\n \Big);  
\end{align}
(ii) at each vertex replace  $[-]^{F_v,\tw}$ by $[-]^{F_v,\omega}$, (cf. \eqref{vertex-assign})
\begin{itemize}
\item if $F_v=Q$, replace $[-]^{Q,\tw}$ by $[-]^{Q,\omega}$, notice $[-]^{Q,\tw}=[-]^{Q,\omega}$;
\item if $F_v=\ffp_\alp$,  replacing   $[-]^{\ffp_\alp,\tw}$ by   $[-]^{\ffp_\alp, \omega}$, 
 \end{itemize}  
and replace $T\loc$ in \eqref{vertex-assign} by \begin{align}  
T(z) : =  \, z(1-R^{-1}(z))\mathbf 1.
\end{align} 

\begin{remark}  The formula $T(z)$ comes from, for $\Delta(z):=\Big(\id_{\sH_Q} \oplus \diag \{ \Delta^{\ffp_\alp}(z) \}_{\alpha=1}^\n \Big)$,  the tail of composed $R$ matrices is (c.f. \cite{PPZ,NMSP3}) 
\beq\label{comp-tail}\Delta(z)^{-1}T\loc_u(z)+ z[1-\Delta(z)^{-1}1]= z-\Delta(z)^{-1}R\loc(z)^{-1} 1=z(1-R(z)^{-1})\mathbf 1. 
\eeq
 \end{remark}

\begin{remark} \label{RwithGRR}
The equation \eqref{rmatrixdelta} can also be understood as a composition of $R$-matrix actions on CohFTs. By the theorem of \cite{PPZ}, the $R$-matrix action on CohFT is a left group action. Hence by first applying the constant $R$-matrix action $ \Delta^{\ffp_\alp}(z)$ at the fixed loci $\ffp_\alp$, and then applying the $R\loc$ which is from the localization, the composition gives us the $R$-matrix $R(z)$. For more details of this point of view, see Section 1.5.1 in the sequent paper \cite{NMSP3}.
\end{remark}

\begin{proof}[Complete the proof of Theorem \ref{R-action-thm}]
We will finish the proof by applying the Divisor and Dilaton equations. 
Notice that we need to evaluate the vertex contribution
\begin{align}\label{vcon}  \sum_{k,l\geq 0} \frac{1}{k!l!} (\pi_{k+l})_* \big[\bigotimes_{j=1}^{n_v}\tt_j(\bar{\bar \psi}_i) , T(\bar \psi)^k  ,u^l\big]^{F_v,\tw}_{g_v,n_v+k+l}. 
\end{align}


At $F_v=Q$, the Divisor equation implies that for $f(q) \in F[\![q]\!]$, $u = f(q) H$, and $\tt_i \in \sH_Q$,  we have $$
  \sum_{l\geq 0} \frac{1}{l!} (\pi_{l})_*   \big[  \tt_1 ,\cdots,  \tt_{n_v}  ,u^l\big]^{Q,\omega}_{g_v,n_v+l} =\big[  \tt_1 ,\cdots,  \tt_{n_v}  \big]^{Q,\omega}_{g_v,n_v} \Big|_{q\mapsto q\cdot e^{f(q)}}.
$$
For our case, this gives the mirror map $q\mapsto q \cdot e^{I_1/I_0}$.  

At $F_v=\ffp_\alp$, we have
$$(\pi_{l})_*   \big[  \tt_1 ,\cdots,  \tt_{n}  ,u^l\big]^{\ffp_\alp,\omega}_{g_v,n_v+l}=0, \quad \text{when}\ l>0.
$$ 
Let $T_v(z):= T(z)|_{F_v}$ and set
\beq  \label{defofoT}
\oT_v(z):=z+ \delta_{F_v} \cdot ( T_v(z)  -z1_{F_v})
\eeq
with $q':=-q/t^\n$ 
and
$$ \textstyle
 \delta_{F}  :=\left\{
  \begin{aligned}
&    I_0 ^{-1}(q') \black  &  &\text{  if  $F=Q$; }\\
&   L^{\frac{\n+3}{2}}(q')  \black  &  &\text{  if  $F= \ffp_\alpha$}.
  \end{aligned}
    \right.
$$ 
Then $\oT_v(z)=\oT_\alp(z)$ in Definition \ref{locclass} if $F_v=\ffp_\alp$. Also for any $v$  
$$T_v(z)=(1-\delta_{F_v}^{-1})z+ \delta_{F_v}^{-1}\oT_v(z).
$$

Using the class-version of Dilaton equation,\footnote{Let $\pi_1:\barM_{g,n+1}(X,d) \to \barM_{g,n}(X,d)$ be the forgetful
map, let $\eta$ be a homology class of $\barM_{g,n}(X,d)$. Then
$\pi_{1\ast}(\pi_1\sta \eta \cap \psi_{n+1}) = (2g-2+n)  v$.}
\beq\label{Guo-dila} {\small
\begin{aligned} T_v
  &    \big[  \otimes_i^{n_v}\tt_i \big]^{F_v,\omega}_{g_v,n_v} \black  :=\,\sum_{s\geq 0} \frac{1}{s!} (\pi_{s})\lsta   
  \big[  \otimes_i^{n_v}\tt_i  ,T_v(\bar\psi)^s \big]^{F_v,\omega}_{g_v,n_v+s}   \\
=\,&\sum_{\ell,m\geq 0} \frac{1}{\ell!m!} (\pi_{\ell+m})\lsta   \big[  \otimes_i^{n_v} \tt_i ,  [(1-\delta_{F_v}^{-1})\bar\psi\mathbf 1]^\ell ,  [\delta_{F_v}^{-1}\oT_v(\bar\psi)]^m \big]^{F_v,\omega}_{g_v,n_v+\ell+m}    \\
=\,&\sum_{\ell,m\geq 0} \frac{  (1-\delta_{F_v}^{-1})^{\ell} }{\ell!m!} \binom{{2g_v-2+n_v+\ \atop \ \ +m+\ell-1}}{\ell}  (\pi_{m})\lsta   \big[  \otimes_i^{n_v} \tt_i,  [ \delta_{F_v}^{-1} \oT_v(\bar\psi)]^m \big]^{F_v,\omega}_{g_v,n_v+m} \\
=\,&\sum_{ m \geq 0} \frac{\delta_{F_v}^{2g-2+n_v}}{m!}   (\pi_{m})\lsta   \big[   \otimes_i^{n_v}\tt_i, [\oT_v(\bar\psi)]^m \big]^{F_v,\omega}_{g_v,n_v+m}.
\end{aligned}}
\eeq
Together with the explicit formula \footnote{$T(z) |_Q = T\loc(z)$ because $\Delta\sta|_Q$ is identity and $\Delta\sta$ preserves $H\sta(\npt)$ in \eqref{comp-tail}. 
 }
\beq \label{TzQ}
T(z) |_Q = T\loc(z) |_Q =   \big(1- I_0(q') \big) \cdot  \mathbf 1_Q z +O(z^{\n-2}),  
\eeq
$$
T(z) |_{\ffp_\alp}=   \,     \big(  1-L(q')^{\frac{-\n-3}{2}}+O(z)\big) \black \cdot \mathbf 1_\alpha z,
$$
(c.f. Lemma \ref{DWmapQ} and Corollary \ref{DWmappt}) we obtain (using  \eqref{defofoT} for definition of $\oT_v$)
$$
\oT_v(z) = \begin{cases}
 O(z^{\n-2})\qquad &\text{if}\ F_v=Q, \\
   O(z^2) \qquad & \text{if} \ F_v=\ffp_\alp.
 \end{cases}
$$
Applying the discussions to \eqref{vcon}, we see that when
$${\n \gg 3g-3+n\geq3g_v-3+n_v},
$$
\eqref{vcon} matches the vertex contribution in Theorem \ref{R-action-thm}.  This completes the proof of Theorem 
\ref{R-action-thm}.

Note that in the above proof,  the condition $\n \gg 3g-3+n$ is only used for \eqref{TzQ}. This allows us to define the translated classes at $Q$ as in  Definition \ref{locclass}, which has no tails contributions. For general $\n$ the proof still works, with the translated classes defined by the translation action on the local CohFT (see \cite{PPZ, NMSP3} for the translation action, see also Remark \ref{RwithGRR}). The CohFT formula in Theorem \ref{Omega01} holds for general $\n$.
\end{proof}

\vspace{1cm}

\section{Proof of Theorem \ref{thm1}}
 
First, we replace the descendants in 
$$
\Big[ \bigotimes_{i\in L_v}   \tt _i  , \bigotimes_{e\in E_v} \frac{ 1^{\alp_e}}{  \frac{5\ft}{a_e}- \psi_{(e,f)}  }     \Big]^{[0,1]}_{g_v,n_v }
$$
(cf. Theorem \ref{summation}) by their ancestors.  There are two cases to consider: the stable ones and 
the unstable ones.

\subsection{Unstable cases}Let $v\in V_w(\bipg)$ be unstable, which is when $(g_v,n_v)=(0,2)$.
\begin{lemma}
Let $\bS$ be defined by
$\displaystyle \bS f(z) :=    S\msp(z^{-})f(z)   \in \sH[\![z,z^{-1} ]\!]$, where $ f(z)   \in \sH[\![z ]$. Then
$$
\qquad \bigl< \frac{\tt_1}{z-\psi_1},\tt_2(\psi_2) \bigr>^{[0,1]}_{0,2}  =    {\big(S\msp(z)\tt_1, \bS\tt_2 (-z)_+ \big)^{\tw}}.
$$
\end{lemma}

\begin{proof}
Using the definition of $W$ and \eqref{twopoinnt}, we have
\begin{align*}
 &\bigl<\frac{\tt_1}{z-\psi_1},\tt_2(\psi_2) \bigr>^{[0,1]}_{0,2} 
 =  \Res_{w=0} \Big(\tt_1 \otimes \tt_2(w) ,\frac{S\msp(z)^* e_\alpha \otimes S\msp(w)^*e^{\alpha} }{z+w}  \Big)^{\tw}\\
&\qquad= \Res_{w=0} \Big(S\msp(z) \tt_1 \otimes \bS\tt_2(w) ,\frac{ e_\alpha \otimes  e^{\alpha} }{w+z}\Big)^{\tw}   =   { (S\msp(z) \tt_1, \bS\tt_2(-z)_+)^\tw},
\end{align*} 
where every ``$w$" in the formula means ``$w^-$".
Also, in the last step we have used  $\Res_{w=0} \frac{f(w)}{w-z} =f(z)_+$.
 \end{proof}

\begin{example}\label{unex}
For every 
genus zero white vertex  with two edges $e_1$ and $e_2$, of hours $\alpha_1$ and $\alpha_2$, respectively,
the unstable contribution is  (cf. \eqref{twopoinnt})
\begin{align*}
 \qquad \Bigl< \frac{1^{\alpha_1}}{  \frac{5\ft_{\alpha_1}}{a_{e_1}}-\psi_{1}  } , \frac{1^{{\alpha_2}}}{  \frac{5\ft_{\alpha_2}}{a_{e_2}}-\psi_{2}  }   \Bigr>^{[0,1]}_{0,2}
=&\,   \frac{  (\mS^{\alpha_1}_{a_{e_1}},\mS^{\alpha_2}_{a_{e_2}})^\tw}{\frac{5\ft_{\alpha_1}}{a_{e_1}}+\frac{5\ft_{\alpha_2}}{a_{e_2}}} .
 \end{align*}
For genus zero white vertex  with one edge and one insertion, we have
\begin{align*}
 \qquad \Bigl<\tau, \frac{1^{{\alpha}}}{  \frac{5\ft_{\alpha}}{a_{e}}-\psi_{2}  }   \Bigr>^{[0,1]}_{0,2}
  =&\,   {  (\tau, \mS^{\alpha}_{a_{e}})^\tw} .
 \end{align*}   \end{example}

\subsection{Stable cases}  This is done by the Kontsevich Manin formula. Indeed we have a cycle
version of the KM formula for the $\nmsp$-$[0,1]$ classes. 
\begin{lemma}[KM formula for classes] For stable vertex
 $v\in V_w(\bipg)$ (i.e.  $2g_v-2+n_v>0$)
 \begin{align*}
 \bigl[   \tt_1 (\psi_1 ),\cdots,   \tt_n (\psi_n )  \bigr]^{[0,1]}_{g_v,n_v}=  \left[ \bS \tt _1 (\bar \psi_1)_+  ,\cdots,\bS \tt _n (\bar \psi_n)_+    \right]_{g,n}^{[0,1]}.
 \end{align*}
 \end{lemma}
 \begin{proof}
Notice that by equation \eqref{A10}, { for an insertion $\tt_j(\psi_j) \in \sH[\psi_j]$}, the corresponding contribution  in the stable graph localization formula at the fixed loci is given by
$$
 [S\msp(\psi)^{-1} [ \bS \tt_j(\psi) ]_+]_+.  \nonumber
$$
Via the descendent-ancestor relation for the fixed loci, as an ancestor insertion it becomes
$$
 R(\bar \psi)^{-1} [ \bS\tt_j(\bar \psi) ]_+.
$$
On the other hand, if we consider an ancestor insertion $\bar \tt_j'(\bar \psi)$ in the master space, the corresponding contribution at the fixed locus is
$$
 R(\bar \psi)^{-1}  \tt_j'(\bar \psi).
$$
By setting $\bar \tt_j'(\bar \psi)= \bS \tt _j (\bar \psi_j)_+ $, we finish the proof.
\end{proof}

By this lemma and the definition of $\mS_{a}^\alpha$ (cf. \eqref{msa}), we see that the ancestor version of the white vertex contribution in Theorem \ref{summation},  after letting $c_{\alpha(e)}\in \FF\sta$ be the constant making $1^\alp=c_{\alpha(e)}1_\alp$, becomes
\beq \label{whitevertexans}
    \Big<  \bigotimes_{i\in L_v}   \tt _i  \bigotimes_{i\in \LL_v}   \bar \psi_i^{k_i} \black  \bigotimes_{e\in E_v,f=(e,v)}  c_{\alpha(e)} \cdot \frac{ \mS_{a_e}^{\alpha(e)}}{  \frac{5\ft_{\alp(e)}}{a_e}- \bar\psi_{f}  }     \Big>^{[0,1]}_{g_v,n_v }.
\eeq
Here we use that, for $\frac{x}{u-z}=\frac{x}{u}(1+\frac{z}{u}+\cdots)$, one has 
$\big[S^M(z^-)\frac{x}{u-z}\big]_+= \Big(\frac{S^M(w)x}{u-z}\Big)|_{w=x}$.
\subsection{Finish the proof of Theorem \ref{thm1}} 
By dimension reasons  \eqref{01des-gen} lies in $\QQ(\zeta_\n)[\![q']\!]$. \black Observe that the localization \ref{01correlator} used to define \eqref{01des-gen} is 
symmetric under the permutations group $S_\n$, which 
acts on the last $\n$ coordinate of $\PP^{4+\n}$.
This implies \eqref{01des-gen} lies in $\QQ[\![q']\!]$. Therefore it is sufficient to show that 
\eqref{01des-gen} lies in $\FF[q]$ with the desired $q$-degree bound. 

Recall our convention that $\deg  p^i= i$, $i\le \n+3$. In this proof we will use the convention that the degree of 
$\sum_{i=0}^{\n+3} c_i p^i$, $c_i\in \FF$, is $\max\{i\mid c_i\neq 0\}$.\footnote{This degree definition 
here is only used in this proof.}
By definition, for each $i$,
 \begin{align}\label{degree} \deg \phi^i\leq \n+3-i.\end{align}

As stated in \cite{NMSP1}, $\deg[\cW_{g,n,(d,0)}]\virt = \n( d+1-g) +n$. Thus for 
pure degree insertion $\{\tau_i\}$, 
\begin{align}\label{MSPco}
\bigl<   \tt_1 \bar \psi^{k_1} ,\cdots,   \tt_n \bar \psi^{k_n} \bigr>_{g,n}\msp
\end{align} vanishes
when $\sum_i  \deg\tt_i  +k_i <  \n(d+1-g) +n$; because $\bar \psi$ are ancestor classes, the   vanishing of \eqref{MSPco} also holds when
$3g-3+n < \sum_i k_i$. Adding the two inequalities, the vanishing holds when  $$\sum_i   \deg  \tt_i   <  \n(d+1-g)  -(3g-3).$$
This proves that \eqref{MSPco} is a polynomial in $\FF[q]$ of degree bounded by
\beq\label{degree01}
g-1+ \frac{3g-3+\sum_i \deg \tt_i }{\n}.
\eeq

We now prove the desired bound of the $q$-degree of $ \left<  -\right>_{g,n}^{[0,1]}$ by induction on $g$. 
When $g=0$, because
$$
 \bigl<  \tt_1 \bar \psi^{k_1} ,\cdots,   \tt_n \bar \psi^{k_n}  \bigr>_{0,n}^{[0,1]} =  \bigl<   \tt_1 \bar \psi^{k_1} ,\cdots,   \tt_n \bar \psi^{k_n}  \bigr>_{0,n}\msp,
$$
the stated bound holds.

We now suppose that the stated bound holds for all genus $h<g$ with arbitrary many insertions.
We apply Theorem \ref{summation} to the genus $g$, $n$ insertions case. Note that the set of bipartite graphs $\Xi_{g,n}$
contains a ``leading" one, the graph $\Lam_g$ with a single genus $g$ white vertex and $n$ markings. 
We let $(\Xi_{g,n})^\circ=\Xi_{g,n}-\{\Lam_g\}$, its complement. 
By Theorem \ref{summation}, 
\begin{align*}
 \bigl<   \tt_1 \bar \psi^{k_1} ,\cdots,   \tt_n \bar \psi^{k_n}   \bigr>_{g,n}\msp=
 \bigl<   \tt_1 \bar \psi^{k_1} ,\cdots,   \tt_n \bar \psi^{k_n}  \bigr>_{g,n}^{[0,1]}+\sum_{\Lam\in (\Xi_{g,n})^\circ}(\ast).
\end{align*} 
Here the first term on the R.H.S. of the identity is the contribution from $\Lam_g$. 
As was argued at the beginning, the L.H.S. of the identity is a $q$-polynomial of degree bound by \eqref{degree01}, thus
to prove the theorem, we only need to show that each term in the summation is a $q$-polynomial of degree 
bounded from above by the same quantity.

Let $\bipg\in (\Xi_{g,n})^\circ$. The contribution from $\Lam$ is, up to a constant multiple,
\begin{align}\label{Con}
\prod_{v\in V_b(\bipg) } \Cont^{\infty}_{[v]}  (\prod_{i\in \LL_v} \bar\psi_i^{k_i}) 
\prod_{v\in V_w(\bipg)} \Big\langle \bigotimes_{i\in L_v}   \tt _i        \prod_{i\in\LL_v}\bar\psi_i^{k_i} 
\bigotimes_{e\in E_v} \frac{\mS_{a_e}^{\alp(e)}}{  \frac{5\ft_\alp}{a_e}- \bar\psi_{(e,v)}  }     \Big\rangle^{[0,1]}_{g_v,n_v }.
\end{align}
 If \eqref{Con} vanishes there is nothing to prove. Assume \eqref{Con} is not zero. Then $B_{[v]}\neq \emptyset$ (c.f. \eqref{infinity-cont}). We substitute $\mS_{a_e}^{\alp(e)}=\sum_{i=0}^{\n+3} \mS_{a;i}^{\alp(e)}\phi^i$ and apply inductions. 
\begin{itemize}
\item
At each stable vertex $v\in V_w(\bipg)$, an easy argument shows that the total genus of all 
black vertices of $\Lam$ is at least one, thus we have $g_v<g$. Consequently, by induction hypothesis
the term $\bigl<-\bigr>^{[0,1]}_{g_v,n_v}$  is a polynomial of degree no more than
$$
g_v-1+ |E_v| + \sum_{e\in E_v}\frac{a_e-1}{5}  +   \frac{3g_v-3+3|E_v|+\sum_i \deg \tt_i }{\n}
$$
where we have used that \eqref{degofmS} and \eqref{degree}  imply that each insertion from edge contributes degree
$$\leq
\deg_q \mS_{a_e;i}^{\alp(e)}+\frac{\n+3-i}{\n} \leq  \frac{a_e-1}{5}+1+ \frac{3}{\n}.
$$
  \item At each unstable white vertex $v\in V_w(\bipg)$, by Example \ref{unex}:\\
(1) if there are one edge $e$ and one insertion $\tau(\bar \psi)$,  the $[0,1]$-correlator is a polynomial of degree
$$
  \lceil \frac{a_e}{5}\rceil-1  +  \frac{\deg \tau }{\n}  \leq  \frac{a_e-1}{5}  +   \frac{  \deg \tt }{\n} ;
$$
(2) if there are two edges $e_1,e_2$,  the $[0,1]$-correlator is a polynomial of degree
$$
 \lceil  \frac{a_{e_1}}{5} \rceil + \lceil  \frac{a_{e_2}}{5}\rceil-1\leq \frac{a_{e_1}-1}{5} +\frac{a_{e_2}-1}{5} +1 +  \frac{3}{\n}
$$
where we have used  \eqref{degofmS}  and that $(\phi^i,\phi^j)^{\tw}=0$ whenever $i,j\geq \n$. 
 
\item  At each black vertex $v\in V_b(\bipg)$
\begin{align*}
\deg \Cont^\infty_{[v]} \leq \,& d_{\infty[v]}+ \frac{2}{5}(g_{v}-1) - \sum_{e\in E_v}  \frac{1}{5} (a_e-1).
\end{align*}  \end{itemize}
  Using  $\sum_{v\in V_b(\bipg)} d_{\infty [v]}  =0$ and $B_{[v]}\neq \emptyset$, we have 
$\sum_{v\in V_b(\bipg)} (g_v-1)  \geq 0$, after applying $d_{0\Theta_\infty}\geq 0$ in \eqref{d-inf-black}, summing over $v$.
Hence
$$
\quad \sum_{v\in V_b(\bipg)} \frac{2}{5}(g_v-1) \leq  \sum_{v\in V_b(\bipg)} (g_v-1).
$$
This proves that the total $q$-degree is no more than
\begin{align*}
\,&\sum_{v\in V_w(\bipg)} (g_v-1+|E_v|)  +\sum_{v\in V_b(\bipg)}\frac{2}{5} (g_v-1) 
 + \frac{1}{\n} \big(\sum_{v\in V_w(\bipg)}3 (g_v-1+|E_v|)+\sum_i \deg \tt_i \big)\\
\leq \,&\sum_{v\in V_w(\bipg)} (g_v-1+|E_v|)  +\sum_{v\in V_b(\bipg)}  (g_v-1)+ \frac{1}{\n} \big(\sum_v 3(g_v-1)+3|E(\bipg)|+\sum_i \deg \tt_i \big)\\
= \,& \quad g-1+ \frac{3  (g-1)+\sum_i \deg \tt_i }{\n} ,
\end{align*}
where in the last step we have used
$$
\quad |E(\bipg)|+ \sum_v (g_v-1 )=
 \sum_{_{v\in V(\bipg)}}g_v +|E(\bipg)|-|V(\bipg)|= \sum_{v }g_v+g(\bipg)-1 = g-1.
$$
This proves the theorem.

\vspace{1cm}

\section{Proof of Theorem \ref{thm3} and \ref{thm4}}
\subsection{Explicit formula for $\n$MSP $S$-function} \label{formulaforS}

 \begin{lemma} \label{QDEforMSPS}
Let $D_p:=p+ z q\frac{d}{d q}$. The $S$-matrix of $\nmsp$ theory $S^M(z)\sta$ satisfies 
the following quantum differential equation 
\begin{equation}\label{QDEforMS}
D_p S\msp(z)^* =   S\msp(z)^*  \cdot A\msp.
\end{equation}
Here we use the same notation $S^M(z)\sta$ to mean the restriction of $S^M(z)\sta$ to $\sH\uev$, which thus is
from $\sH\uev$ to $\sH\uev$. And, w.r.t. the basis $\{\phi_i=p^i\}$ in order $i=0,\cdots,\n+3$,
\beq \label{Anmsp}
\qquad \qquad A\msp ={\small
\   \arraycolsep=1.8pt\def\arraystretch{1.6}
 \left[
  \begin{array} {*{11}{@{}C{\mycolwd}@{}} c}
 0 & &&&&&&  120q\\
1  &0 & &&&&&&770q& \\
&1 & 0 &  &&&&&&1345 q  \\
 && 1 &0 & &&&&&&770q  \\
 &&&1&0 & &&&&&& 120q-\tiny{t^\n} \!\! \\
 \hline
& &&& 1 &0 & \\
 &&&&& 1&0 &  \\
  & &&&&&1&0 &  \\
&  & &&&&&\cdots&\cdots &  \\
& & & &&&&&1&0 &  \\
& &   &  &&&&&&1&0 &  \\
& &   & & &&&&&&1& {0 \,  \quad} 
  \end{array}\right]. \quad}
\eeq
\end{lemma}  
\begin{proof}
The QDE matrix can be computed via Birkhoff factorization, which is an algorithm starting from $J$-functions to get $S$-matrices. The $S^*(z) \phi_i$ can be computed recursively.
\begin{enumerate}
\item By definition, for $\phi_0=1$, we have $ S\msp(z)\sta 1\,\, =\, z^{-1} J\msp(z).$
\item  Suppose we already obtain closed formulas of 
$$\textstyle
S\msp(z)\sta 1, S\msp(z)\sta \phi_1,\cdots, S\msp(z)\sta \phi_k
.$$ Apply $D_p:=z q\frac{d}{dq} + p$ to  $S\msp(z)\sta \phi_k$.  It will keep the element in the Lagrange cone.  We search for a linear combination of
$$\textstyle
S\msp(z)\sta 1,\quad S\msp(z)\sta \phi_1,\quad\cdots,\quad S\msp(z)\sta \phi_k,\quad D_p S\msp(z)\sta \phi_k
$$ (with coefficients in $\aA$)
such that the combination takes the form
$$
\phi_{k+1} + O(z^{-1}).
$$
By Coates-Givental's result  such linear combination is exactly $S\msp(z)\sta \phi_{k+1}$.

\item  The  process  stops when no new $\phi_k$ appears. Whenever $\phi_1$ generates the even part of quantum cohomology ring, this algorithm  provides us the full $S\msp(z)^*$.
\end{enumerate}
Applying this algorithm to the $I$-function $I\msp(z)$, we obtain exactly \eqref{Anmsp}.
\end{proof}

Recall that the tail contribution at the fixed loci $Q$ is given by (c.f. Sect.\,\ref{tailcont})
\begin{align*}
L\loc(z)|_Q =&\  z\mathbf 1+  J\msp(0,-z) |_{Q,+}\\
= &\ -\sum_{d=1}^\infty z \, q ^d \frac{\prod_{m=1}^{5d}(5H-mz)}{ \prod_{m=1}^d (H-mz)^5 \prod_{m=1}^d \big((H-mz)^\n + \ft^\n \big)} \Big|_{+}.
\end{align*}
By $H^4=0$ and the explicit formula of $I^Q$ \eqref{Q-Ifunc}, we have
\begin{align}\label{LLoc}
L\loc(z)|_Q =&  -\sum_{d=1}^\infty z \, (q') ^d   \frac{\prod_{m=1}^{5d}(5H-mz)}{ \prod_{m=1}^d (H-mz)^5 } \cdot \Big(1+  O(z^{\n})  \nonumber \\
&\qquad+ H\cdot O(z^{\n-1})+ H^2\cdot O(z^{\n-2})+ H^3\cdot O(z^{\n-3})\Big)  \Big|_{+} \nonumber\\
  =  &\quad   z+I^Q(q',-z)  \big|_+ + O(z^{\n-2}) \nonumber\\
  = & \quad z \cdot \big(1- I_0(q') \big)+ I_1(q') H +O(z^{\n-2}).  
\end{align} 

\begin{convention}\label{conv-q}{ In this and next section, we will always regard the symbols 
$$
I_0,\,I_1,\,I_2,\,I_3 ,\, A_1,\, B_1,\, B_2,\, B_3,\,Y
$$
as their original definition with $q$ substituted by $q'= -q/\ft^\n$. For example, in section five and six,
$I_1=I_1(q')$, $Y=Y(q') = (1+5^5 q/\ft^\n)^{-1}$, e.t.c..}
\end{convention}

\begin{lemma} \label{DWmapQ}
Suppose $\n>5$. \footnote{Indeed, a more careful analysis shows $\n>2$ in enough.  Further, for the cases $\n=1,2$ one can still compute the Dijkgraaf-Witten map and prove certain polynomiality of entries of the $R$-matrix, however the computations will be much more involved. 
For example, for the original MSP (i.e. $\n=1$) we have 
$$\tau\loc(0)|_Q =  5!q+\tau_Q  +\tau_2 H^2 +\tau_3 H^3$$
with $\tau_2 = - \frac{1}{2} \, I_{1,1} I_{2,2}$ and $\tau_3 =  I_{1,1}^2I_{2,2} \big( -\frac{5}{12}+\frac{Y}{4}-\frac{B}{2}-\frac{A}{4} \big)$ (here $I_{2,2}$ is defined in \eqref{I22}). }
We have the following formula for $\tau\loc(0)|_Q$  and $T\loc(z)|_Q$ (see \eqref{taulocal} and \eqref{Tloc} for their definitions)
\begin{align*}
\tau\loc(0)|_Q = \,& \tau_Q(q'):=\frac{I_1(q')}{I_0(q')}H,\quad \text{and}\quad
T\loc(z) |_Q = \big(1- I_0(q') \big) \cdot  \mathbf 1_Q z + O(z^{\n - 2}) .
\end{align*}
\end{lemma}
\begin{proof}
By  \eqref{taulocal}, \eqref{LLoc}, and the method deriving \eqref{Guo-dila},  we have  
\begin{align*}
\tau\loc(0)|_Q= & \sum_{\alpha,n} \frac{1}{n!} e^{\alpha} \left<  {e_\alpha} ,\mathbf 1,L\loc(-\psi)|_Q^n \right>^{Q,\tw}_{0,n+2} \\
 = & \sum_{\alpha,n} \frac{1}{n!} e^{\alpha} \left<  {e_\alpha} ,\mathbf 1, \big((1-I_0(q')) \psi + I_1(q') H  \big)^n \right>^{Q,\tw}_{0,n+2} \\
= & \sum_{\alpha,n} \frac{1}{n!} e^{\alpha} \left<  {e_\alpha} ,\mathbf 1, \big(  \tau_Q(q') \big)^n \right>^{Q,\tw}_{0,n+2}  = \tau_Q(q').
\end{align*}   
 For the second equality, we use \eqref{LLoc}, 
\eqref{quinticS}, and \eqref{twq'} to obtain
\begin{align*}
T\loc(z) |_Q =\,&  S_{\tau_Q}^{Q,\tw}(z) \big( L\loc(z)-\tau_Q\big) |_{Q,+}  +  O(z^{\n - 2})   \\
=\,& \big(1- I_0(q') \big) \cdot  \mathbf 1_Q z  +  O(z^{\n - 2}).
\end{align*}
This proves the lemma.
  \end{proof}
\begin{example} \label{leadingR} Similar to  \ref{LLoc}, we have
$$
S\msp(z)^* \mathbf 1|_{Q,+} =
z^{-1} J\msp(z)|_{Q,+} = I_0(q')+O(z^{\n-3})
$$
Applying \eqref{QDEforMS} to $\phi_0=p^0=\mathbf 1$, we get  \begin{align*}
S\msp(z)^* p|_{Q,+} &= \big( D_p S\msp(z)^* 1 \big) |_{Q,+}
  =     H  I_0(q')    + zD(I_0(q'))  +  D(I_1(q'))H+O(z^{\n-2})  .
\end{align*}
Recall  \eqref{DefineR} and \eqref{twq'} give formally  
\begin{align}\label{formally}
&\big( R(z)^{-1} x \big) |_Q =  S^Q(q',z)  \big(  S^{M}(z)^{-1} x \big)|_Q,  \qquad \forall x\in\sH.
\end{align} 
 Together with \eqref{quinticS} 
  they lead to the first two $R$ matrix entries:
\begin{align*}
R(z)^* \mathbf 1 |_{Q} 
   &=     I_0(q')+O(z^{\n-3}),  \\
R(z)^* p|_{Q} 
   &=     zD(I_0(q'))  +  H\, I_0(q')  \,I_{1,1}(q')+O(z^{\n-2})  ,
\end{align*}
where in the last equality we have used $I_{1,1} = 1+D(I_1/I_0)$.
\end{example}

\subsection{Key Lemmas}

We define the normalized basis for $H^*(Q)$ and its dual by \footnote{The same normalized basis has been used in \cite[Sect.\,6.2]{GR2}, which naturally appear in the computation of the canonical basis for the twisted theory (see also \cite[Sect\,6.3]{GR1}).}
$$
 \tp_b:=  { I_0(q') I_{1,1}(q')\cdots I_{b,b}(q')} H^{b}  \,,
\qquad
 \tp^b:= \frac{-\ft^\n}{5\,I_0(q') I_{1,1}(q')\cdots I_{b,b}(q')} H^{3-b}  ;
$$
We define the normalized basis for $H^*(\pt_\alpha)$ and its dual by 
$$\bar {\mathbf 1}_\alpha:=L^{-\frac{\n+3}{2}}{\mathbf 1}_\alpha,\qquad  \bar {\mathbf 1}^\alpha:=L^{\frac{\n+3}{2}}{\mathbf 1}^\alpha.$$
We set $L_\alpha:=\xi_\n^\alpha \, \ft \cdot L(q') = \xi_\n^\alpha  ( \ft^\n+5^5 q )^{1/\n} $  and introduce
$$
{(  R_k)_{j}}^b  := (R_k \tp^b, p^j)^\tw,\qquad{(  R_k)_{j}}^\alpha  := L_\alpha ^{-(j-k)}\cdot   (R_k \bar{\mathbf 1}^\alpha , p^j)^\tw.
$$

\begin{lemma}\label{keylem} Suppose $k<\n-3$.  At the fixed loci $Q$, we have
\beq  \label{vanishRQ}
{(  R_k)_{j}}^b  =0,  \quad \text{ if  }  j \not \equiv  b+k \!\mod \n \,,
\eeq  and for $b=0,1,2,3$
\begin{equation} \label{polyofRQ}
\textstyle
{(  R_k)_{b+k}}^b  \ \ \text{and}\ \  \frac{Y}{\ft^\n} \! \cdot \! {(  R_k)_{b+\n+k}}^b  \ \in  \  \mathbb Q[A,B,B_2,B_3,Y].
\end{equation}
At the fixed loci $\ffp_\alpha$, we have
\begin{equation} \label{polyofRpt}
{ (R_k)_{j}}^{ \alpha} \in  \mathbb Q[Y]_{k+\lfloor\frac{j}{\n} \rfloor}
\end{equation}
  which is independent of $\alpha$.
\end{lemma}
\begin{proof}
First we prove the vanishing property \eqref{vanishRQ} and polynomiality \eqref{polyofRQ} of the $R$-matrix restricted at the fixed loci $Q$. 
By Example \ref{leadingR}, we have
\beq \label{initial} 
{(  R_k)_{0}}^b =   \delta_{b,0} \delta_{k,0}   \quad  \ \text{for}\quad  k<\n-3.
\eeq
To compute the other columns,  recall that
 the QDE for $S\msp$ and $S^Q$ are (c.f. \eqref{QDEforMS}, \eqref{qQDE-star})\footnote{By Convention \ref{conv-q}, to entries of $S^Q,A^Q$(\eqref{quintic-QDE}, \eqref{quinticS}), we apply $q\mapsto q'$, and still denoted them as $S^Q,A^Q$ in QDE.} 
 $${ 
\begin{aligned}
\big( p+ z\,q\frac{d}{dq} \big)S\msp(z)^{*} = & \ S\msp(z)^{*} \cdot A\msp
,\\
\big( H+ z\,q\frac{d}{dq} \big)S^Q(z)^{*}  = & \  S^Q(z) ^{*} \cdot   A^Q. 
\end{aligned}}$$
Together with the Birkhoff factorization \eqref{formally},
we obtain the QDE for $R^0(z)^*$:
\beq
(z D +A^Q    )\,  (R(z)^*  x)  \big|_Q  =   ( R(z)^* \cdot A\msp x )    \big|_Q , \qquad \forall x\in\sH.
\eeq
This implies, under the basis $\{ \tp_i\}$, 
  for $j=1,\cdots ,\n+3$
\beq  \label{induc}
{(  R_k)_{j}}^b =  (D+C_b)  {(  R_{k-1})_{j-1}}^{b}   +{(  R_k)_{j-1}}^{b-1} -c_j  q \, 
 {(  R_k)_{j-\n}}^b,
\eeq
\beq
 C_b:= D \log (I_{0}I_{1,1}\cdots I_{b,b}) \in \sR, \qquad b=0,1,2,3 ,
\eeq
\beq \label{constantscj}
  (c_{j})_{j=1,\cdots,\n+3} :=  (0,\cdots,0,120,770,1345,770).
\eeq
Recall the ring $\sR$ is closed under  $D$, hence  \eqref{polyofRQ} follows from \eqref{induc} by recursion.  

For \eqref{vanishRQ}, just notice that in the inductive formula \eqref{induc}, the difference of the index $j-b-k \mod \n$ is preserved. Hence  \eqref{vanishRQ} is implied by \eqref{initial}  recursively.

\medskip

Next, we prove the property \eqref{polyofRpt} of the $R$-matrix restricted at the fixed loci $\npt$. The $j=0$ case  is proved in Section \ref{keylemRpt}. We now prove $j>0$ case
.  Apply   \eqref{QDEforMS} to \eqref{DefineR}:
$$
(\mathbf 1^\alpha,  R(z)^* \phi_j  )  =   e^{u_\alpha/ z}  \Delta^{\pt_\alpha}(z)^* (\mathbf 1^\alpha, S\msp(z)^* \phi_j).
$$
we see that for $j>1$, the entries $(\mathbf 1^\alpha,  R(z)^* \phi_j )$ can be recursively computed  via
$$
(\mathbf 1^\alpha, R(z)^* \phi_{j}) =  D_{L_\alpha} (\mathbf 1^\alpha, R(z)^* \phi_{j-1}) - c_j \,  q \,(\mathbf 1^\alpha, R(z)^* \phi_{j-\n} )
$$
where $(c_{j})_{j=1,\cdots,\n+3}$ is defined as in \eqref{constantscj}. Namely, \footnote{Note we have used $L^{-\n} \, q  = {(1-Y)}/{5^5}$, $L_\alpha = \zeta_\n^\alpha t\,L$ and $D L  =  \frac{1}{\n} L  (1-Y)$. }
$$
{(R_k)_j}^\alpha \! = \Big(D-{\textstyle \frac{1}{\n}\big(\frac{\n+3}{2}-j+k\big)(1-Y)}\Big)  {(R_{k-1})_{j-1}}^\alpha \!+ {(R_{k})_{j-1}}^\alpha\!+ \frac{c_j}{5^5} \, (1-Y) \, {(R_{k})_{j-\n}}^\alpha .
$$ 
By induction on $j$, using $D(Y)=Y(Y-1)$ and the initial result ($j=0$), we conclude ${(R_k)_j}^\alpha$ does not depend on $\alpha$ and  ${ (R_k)_{j}}^{ \alpha} \in  \mathbb Q[Y]_{k+\lfloor\frac{j}{\n} \rfloor}
$.
\end{proof}

\begin{corollary} \label{DWmappt}
The Dijkgraaf-Witten map at $\ffp_\alp$ is given by
$$
 \tau\loc(0) \black |_{\ffp_\alp} = \tau_\alp(q') =  -t_\alpha \int_0^{q'}  ( L(x) -1 )  \frac{dx}{x}\black
,\qquad 
$$
and the tail contribution at $\ffp_\alp$ is given by{\footnotesize
\begin{align*}
\,T(z) |_{\ffp_\alp}\!\!&=  \mathbf 1_\alpha z - L^{\frac{-\n-3}{2}} \, \mathbf 1_\alpha z \cdot
\bigg[1- \bigg(  {  \Big( \frac{\n}{24}+{\frac{43}{120}} \Big) }+ {\frac {Y-1}{\n} \Big( {\frac{47}{24}}+\frac{23 \n}{24}+\frac{{\n}^{2}}{12} \Big) }
 \bigg) \frac{z}{L_\alp} +O(z^2)  \bigg].
 \end{align*}}
\end{corollary}

\begin{lemma}\label{keylem2}
Suppose
$$
V(z,w)  = \sum_{k,l \geq 0} V_{kl}  z^k w^l.
$$
 Then the coefficients $V_{kl} $ can be written in the following form
\begin{align} \label{vkl}
\begin{aligned}
 V_{kl} = &\textstyle \sum_{a,b=0}^3 (  V_{kl})^{ab} \tp_a\otimes \tp_b
+ \sum_{b=0}^3 \sum_{\alpha=1}^\n L_\alpha^{2-b-k-l}  \cdot (  V_{kl})^{\alpha b} \bar{\mathbf 1}_\alpha\otimes \tp_b+
\\
  & +\textstyle \sum_{\alpha,\beta =1}^\n  \sum_{j}  
  L_\alpha^{j-k} L_\beta^{2-j-l}   \cdot (  V_{kl})^{\alpha \beta;j} \,\bar{\mathbf 1}_\alpha\otimes \bar{\mathbf 1}_\beta 
 \end{aligned}
\end{align}
such that
\begin{align*} \textstyle
\frac{Y}{\ft^\n}\! \cdot \! (  V_{kl})^{\alpha \beta;j}  \in    \, 
 \mathbb Q [Y]
,\qquad 
\frac{Y}{\ft^\n}\! \cdot \! (  V_{kl})^{\alpha b}   \and  \frac{Y}{\ft^\n}\! \cdot \! (  V_{kl})^{ab}   \in  \,    \mathbb Q[A,B,B_2,B_3,Y]
\end{align*}  
are indepedent of $\alpha,\beta$. \black Further, the edge contribution $V(z,w)$ is homogeneous of degree $2$.
\end{lemma}
\begin{proof}
Recall that we have the relations
$$
V(z,w) = \frac{1}{z+w}
 \Big(  \sum_{j=0}^{\n+2}  \tp_j  \otimes   \tp^j - R(z)^{-1}\tp_j  \otimes R(w)^{-1} \tp^j    \Big)_{i=0}^{k+l}.
$$
Hence for each $k,l$, $V_{kl} $ is a linear combination of
$$
 \Big\{  \sum_{j=0}^{\n+2}  R_i^*\tp_j  \otimes   R_{k+l+1-i}^* \tp^j    \Big\}_{i=0}^{k+l}.
$$
Then the lemma follows from Lemma \ref{keylem}. 
\end{proof}

\subsection{Proof of Theorem \ref{thm3}}   To simplify the computation, we pick $t$ such that $t^\n=-1$. This makes $q'=q$, 
By definition, it is clear that for
$2g-2+n>0$
\beq \label{Pgnrec}
P_{g,n+1} =  \Big(q\frac{d}{dq} +  ( g-1) (2B+1-Y ) -n \,A \Big) P_{g,n}.
\eeq
Hence if $P_{g,n} \in \sR$, so does $P_{g,n+1}$.

By Theorem \ref{R-action-thm},  for any $[0,1]$-theory, it is equal to a summation over stable graphs. We define the leading graph of the stable graphs in $G_{g,0}^\n$ to be the single genus $g$ vertex labeled by $Q$ (quintic fixed loci).
We now prove the theorem by induction:

First  by using \eqref{3point} and genus $1$ mirror theorem \cite{Zi} \footnote{For genus $1$ case, Zinger's  theorem can be recovered by considering the genus $1$ $\nmsp$-theory with one $\phi_1$-insertion (see \cite{NMSP3}).}
$$ \textstyle
P_{0,3} = 1 \and  P_{1,1} = - \frac{1}{2}A - \frac{31}{3} B + \frac{1}{12}Y - \frac{13}{6}
$$
are  both in $\sR$.  Thus by \eqref{Pgnrec} 
$$P_{0,n} (n\geq 3), \ \text{and}\  \ P_{1,n} (n\geq 1)  \  \in \ \sR.$$ 
We next assume $g\geq 2$. 
Assume for any genus $h<g$, and any $2h-2+n>0$ one has $P_{h,n} \in \sR$.   
We consider the normalized $\nmsp$-$[0,1]$ potential $Y^{g-1}\cdot F^{[0,1]}_g$.  By Theorem \ref{thm1},
$$
Y^{g-1}\cdot F^{[0,1]}_g \in \mathbb Q[Y]_{g-1}.
$$ 
On the other hand, we have
the graph sum formula by Theorem \ref{R-action-thm}. The leading stable graph of the leading bipartite graph is a single genus $g$ vertex labelled by $0$, with contribution $\frac{Y^{g-1}}{I_0^{2g-2}} F_g =P_g$.
For the rest of the graphs, via the relation
$$
\sum_v (g_v-1) + E = g-1
$$
we put the factor $Y^{g-1}$ into vertices and edges. Together with Lemma \ref{keylem} and \ref{keylem2}, the contribution of each non-leading graph is given by the followings:

$\bullet$ At each edge, the contribution is of the form (by \eqref{vkl})
\begin{align}\label{edge-summand}
Y \cdot (  V_{kl})^{a b} ,\quad Y \cdot (  V_{kl})^{\alpha b}  \quad  \text{  or  }  \quad   Y \cdot (  V_{kl})^{\alpha \beta;j}
\end{align}
which lies in the ring $\sR$. We pick (any) one of them from \eqref{vkl}, and vary the hour $\alpha=1,\cdots,\n$ of each level $1$ vertex. The variation provides a multiplicative factor below, since  \eqref{edge-summand} are independent of hours (c.f. Lemma \ref{keylem2}). 

$\bullet$ At each quintic vertex $v$, the contribution is
\begin{align*}
P_{g_v,\vec a, \vec k} :=   Y^{g_v-1} \int_{\M_{g_v,n_v}}  \left[  \tp_{a_1} \bar \psi_1^{k_1} ,\cdots, \tp_{a_{n_v}} \bar \psi_{n_v}^{k_{n_v}} \right]_{g_v,n_v}^{Q,T}.
\end{align*}
Here we recall the translated correlators are defined by the equation in Definition \ref{locclass}. 
It vanishes unless  $\sum a_i+k_i = n_v$. By using string and dilation equations once and again, this correlator will reduce to $P_{g_v,m} $  multiplied by a constant \footnote{For $g=1, \vec a=0^{n_v}$ case, the correlator will reduce to 
$$
\left< \bar\psi \right>_{1,1}^Q =  \frac{\chi}{24}\quad  \text{  with   }   \quad  \chi = -200.
$$
}
. Since $g_v<g$, by induction hypothesis we have $
P_{g_v,\vec a, \vec k}  \in \sR. $

$\bullet$ At each $\ffp_\alpha$ vertex, the contribution is
\begin{align*}
 & Y^{g_v-1}  \int_{\M_{g_v,n_v}}  \left[ L_\alpha^{j_1-k_1}  \bar \psi_1^{k_1} ,\cdots,  L_\alpha^{j_{n_v}-k_{n_v}}  \bar \psi_{n_v}^{k_{n_v}} \right]_{g_v,n_v}^{\ffp_\alpha,T}\\
\qquad\qquad = \,& \sum_m\frac{L^{\frac{3}{2} (2 g_v-2)}}{m!} \left< L_\alpha^{j_1-k_1}  \bar \psi_1^{k_1} ,\cdots,  L_\alpha^{j_{n_v}-k_{n_v}}  \bar \psi_{n_v}^{k_{n_v}} , \prod_{s=1}^m \oT_\alpha(\bar \psi_{n_v+s})  \right>_{g_v,n_v+m}.  \nonumber
\end{align*}

We claim that, after summing over $\alpha=1, \cdots,\n$, the contribution lies in $\sR$. 
\begin{enumerate}
\item For $s=1,\cdots,m$, if each monomial in 
$$ \textstyle
 \oT_\alpha(\bar\psi_{n_v+s})={\sum_{l_s} (\oT_\alpha)_{l_s} \bar \psi_{n_v+s}^{ l_s+1}}$$ contribute $(\oT_\alpha)_{l_s} \bar \psi^{ l_s+1}$ ,  the correlator is non-zero only if
$$  \textstyle
\sum_{i=1}^{n_v} k_i  +\sum_{s=1}^m  (l_s+1) = 3 g_v-3+ n_v+m.
$$
\item
Together with the fact $L_\alpha^{ l} \cdot (\oT_\alpha)_l = (R_l)_{j \alpha} \in  \mathbb Q[Y] $ (for all $l$), we see
that the total factor involving $L_\alpha$ is $L_\alpha^{(\sum_i j_i) -n_v}$.
This makes $\sum_{\alp=1}^\n L_\alpha^{(\sum_i j_i) -n_v}$ a multiplicative factor of the contribution.
 Since   $\n$ may be chosen to be a prime, we do such assumption in the beginning.  Then this multiplicative factor is non-zero only if 
$$ \textstyle \sum_{i=1}^{n_v} j_i = n_v \mod \n. $$
 The total factors in the graph $\Gamma$ becomes
$$ \textstyle
\prod_{v\in\Gamma} L_{\alpha(v)}^{\sum_{i=1}^{n_v} j_i -n_v}  = 
\prod_{v\in\Gamma} Y^{(\sum_{i=1}^{n_v} j_i -n_v)/\n}  = 1.
$$
Here we have used that for each edge, if at one end it contributes $L_{\alpha}^{j}$ then in the other end it  contributes $L_{\beta}^{2-j}$.
\end{enumerate}

We conclude that the summation of the rest graph contributions lies in the ring $\sR$, thus
$$
P_{g} \in \sR.
$$
Finally, by \eqref{Pgnrec}
we have $
P_{g,n} \in \sR$ {for any} $n>0$.
This proves Theorem \ref{thm3}.
\vspace{0.5cm}
\subsection{Proof of Theorem \ref{thm4}}
Notice that the quintic $I$-functions $\{I_i\}_{i=0}^3$ (see \eqref{Q-Ifunc}) are analytic functions in the disk $\{|q|<\frac{1}{5^5}\}$ (c.f. \cite[(3,14)]{CdGP}). So are the mirror map $I_1/I_0$ and the generators $A_k$ and $B_k$. Further, the map
$$
q \rightarrow Q(q) = q\cdot e^{I_1(q)/I_0(q)}
$$
is an analytic homeomorphism between neighborhoods of zeros. 
Hence, we have any element in $\sR$ is an analytic function near $Q=0$ as a function of $Q$.
 
\vspace{1cm}

\section{Proof of a key property} \label{keylemRpt}
We consider the $R(z)^*1$ restricted at the fixed loci $\npt$. We will prove
\begin{equation} \label{polyofRpt1}
{ (R_k)_{0}}^{ \alpha} = L^{\frac{\n+3}{2}}L_\alpha ^{k}\cdot   (R_k \mathbf 1^\alpha , \mathbf 1)^\tw \quad  \in  \quad  \mathbb Q[Y]_{k}
\end{equation}
in this section. The idea is to use the Picard-Fuchs(PF) equation to solve $R_k$ recursively and to use 
Givental's oscillator integral to determine their initial values.

\subsection{Applying Picard-Fuchs equation}
The first columns ${ (R_k)_{0}}^{ \alpha}$ can be solved from the PF equation for $\nmsp$ I-function.  It is clear $I\msp$ (see \eqref{nmspI}) satisfies the PF equation:
\begin{equation} \label{PFofMSP} 
\bigg(  D_p ^5 \prod_{\alpha=1}^\n(D_p+\ft_\alpha)  -q  \prod_{k=1}^5(5D_p+kz)   \bigg) I\msp(q,z) = 0,
\end{equation}
where $D:= q\frac{d}{dq}$ and $D_p:=z D+p$.  By \eqref{DefineR}, \eqref{twq'} and \eqref{tauloc} we also have\footnote{ in this section we always brief $\tau_\alp=\tau_\alp(q')$; \black}
$$
  \Delta^{\pt_\alpha} (z)^* J\msp(0, z)|_{\pt_\alpha}= z\Delta^{\pt_\alpha} (z)^*  S\msp(z)^{*} \mathbf 1|_{\pt_\alpha} = z  e^{\tau_{\alpha}/z} R^*(z)  \mathbf 1|_{{\pt_\alpha}}.
$$
Via the   mirror theorem (Theorem~\ref{mirror}) we obtain
\begin{equation} \label{asymptofI_1}
\Delta^{\pt_\alpha} (z)^{*}I\msp(q,z)  |_{{\pt_\alpha}} =  e^{\tau_{\alpha}/z}z  R^*(z)  \mathbf 1|_{{\pt_\alpha}}.
\end{equation}
The LHS of equation \eqref{asymptofI_1} satisfies the PF equation \eqref{PFofMSP} as well since $\Delta^{\pt_\alpha}(z)$ is a constant in $q$ . Hence we see {$R^*(z) \mathbf 1|_{\pt\lalp}$}  satisfies
\begin{equation}  \label{PFforRpt} 
\Big( D_{{L_\alpha}}^5   \big( D_{{L_\alpha}}^\n -\ft^\n \big)  -q \prod_{k=1}^5  (5D_{{L_\alpha}}+kz)   \Big) R^*(z) \mathbf 1|_{\pt_\alpha}= 0,
\end{equation} 
where $D_{{L_\alpha}} := zD+    {L_\alpha}$ and $ {L_\alpha}:=  -t_\alpha+   q\frac{d}{dq} \tau_\alpha$.
Note that  for any $k>0$ 
\begin{align} \label{DLk}
 D_{L_\alpha}^k \!=& \, L_\alpha^k \cdot \Big[ 1 + \frac{z}{L_\alpha}  \cdot \!\Big(k\,  D+ \frac{k(k-1)}{2} \frac{D { L_\alpha}}{L_\alpha}   \Big)+ \cdots \Big]
\end{align}
By solving this equation we obtain both ${L_\alpha}$ and $R^*(z)  \mathbf 1|_\pt$.

{We solve them recursively as follows:}
 First we look at the coefficient of $z^0$ of \eqref{PFforRpt}, which gives us
$$ 
L_\alpha^5 (L_\alpha^\n-t^\n) -q \cdot 5^5 L_\alpha^5 =0.
$$
This equation has two types of solutions: $L_\alpha=0$ with multiplicity $5$ and  $L_\alpha=\xi_\n^\alpha \, \ft ( 1+5^5 q/\ft^\n)^{1/\n} $ with multiplicity $1$.
The first solution corresponds to the matrix $S^{Q}$ which is not diagonalizable, while the second solution corresponds to the diagonalizable part $S^{\pt_\alp}=e^{\tau_\alp/z}$,  which is what we need. 
 We then obtain
\begin{align*} 
\tau_{\alpha}  = \tau_{\alpha} (q')  = \int_0^{q'}  ( L_\alpha +  \ft_\alpha )  \frac{dq}{q}
\end{align*} 
where we have used the initial condition $\tau_\alpha |_{q=0}=0$.

Next we look at the coefficient of $z^1$ of \eqref{PFforRpt}, and use the value of $L_\alp$ we compute
\begin{equation*}
 \textstyle
L_\alpha=\zeta_\n^\alpha \,\ft( 1+5^5 q/\ft^\n )^{1/\n} , \and  D L_\alpha =   {(L_\alpha^\n-t^\n)}/\n {L_\alpha^{\n-1}}.
\end{equation*} 
Recall $R(z)$ is a symplectic transformation such that 
$$
R^*(z) : = R^*_0+R^*_1 z+R^*_2 z^2+\cdots \in \Hom(\sH ,  \sH_Q\oplus \sH_{\npt} )[\![z,z^{-1}]\!].
$$  
The coefficient of $z^1$ of the equation \eqref{PFforRpt} becomes 
$$
\frac{3+\n}{2} 5^5  q\,  R_0^* \mathbf 1 |_{\pt_\alpha}+ \n ( 5^5 q+\ft^\n)  \, D(R_0^* \mathbf 1 |_{\pt_\alpha}) =0.
$$
Solving it we obtain
$$
R_0^* \mathbf 1 |_{\pt_\alpha} =   L ^{-\frac{\n+3}{2}},
$$
where $L=   ( 1+ 5^5 q/\ft^\n)^{1/\n}$. Here we have used the initial condition $R_0^* \mathbf 1 |_{\pt_\alpha, q=0}  = [z^0] \Delta^{\pt_\alpha}(z) =1$.

Then we look at the coefficient of $z^2$ of this equation, and solve $R_1 \mathbf 1 |_{\pt_\alp}$ up to a constant. Repeating the steps, we can solve $R_k^* \mathbf 1 |_{\pt_\alp}$ for any $k$. Note that at each step we have one constant to fix. The constant can be fixed by the $\Delta^{\pt_\alpha}$-matrix. 

\medskip
By using the above idea, we now prove   \eqref{polyofRpt1}.
We write
$$ \textstyle
R(z)^* \mathbf 1 |_{\pt_\alpha}   =    L ^{-\frac{\n+3}{2}} \cdot \big(1+\frac{r_1}{L_\alpha}  {z} +\frac{r_2}{L_\alpha^2}  {z^2}  +\cdots\big).
$$
 The  equation \eqref{PFforRpt} becomes of the form
\begin{equation} \label{PFR}
L_\alpha^{\n+5}\cdot \PF  \Big(  L ^{-\frac{\n+3}{2}} \cdot \big(1+\frac{r_1}{L_\alpha}  {z} +\frac{r_2}{L_\alpha^2}  {z^2}  +\cdots\big)  \Big)=0,
\end{equation}
where we denote the operator on LHS of  \eqref{PFforRpt} by $L_\alpha^{\n+5} \cdot \PF$ . Recall $q'=-q/\ft^\n$ and let
$$X := 1-Y= {(-5^5 q')}/{(1-5^5 q')} =  1- L^{-\n}   .$$
By the explicit formula of \eqref{PFforRpt}, we see the operator 
\begin{align*} 
  \ \textstyle  \PF
  =& \ \textstyle\frac{1}{L_\alpha^{5}}  \big(\frac{1}{L_\alpha^{\n}}  D_{{L_\alpha}}^\n +1-X \big)
D_{{L_\alpha}}^5    +\frac{X}{5^5 L_\alpha^5} \prod_{k=1}^5  (5D_{{L_\alpha}}+kz)  
\end{align*}
 is of the  form $
 \PF =  \sum_{k=1}^{\n+5}  \frac{z^k}{ L_\alpha^{k}} \cdot \PF_k$ ,
with 
\beq  \label{PFk} \textstyle
\PF_k = {\small \text{$\sum_{j=0}^k$}} f_{k,j}(X)D^{k-j} \and f_{k,j} \in \mathbb Q[X]_j \quad \forall k,j.
\eeq
 Here we have used  the following observations:
\begin{enumerate}
\item acting on $L_\alpha$ or $X$, the operator $D$ increases the $X$-degree by $1$:
\beq \label{DLDX}
D L_\alpha  =  L_\alpha \cdot X /\n \and D X = X(1-X) ;
\eeq
\item in each monomial of $\PF_n$, the operator $D_{L_\alpha}$ appear at most $n$-times;
\item for any $n$, by using \eqref{DLDX} the operator $L_\alpha^{-n} D_{L_\alpha}^n$ has the following form
\beq \label{ckjdegree} \textstyle
L_\alpha^{-n} D_{L_\alpha}^n =\sum_{k=0}^n \ z^{k}L_\alpha^{-k} \cdot   \sum_{j=0}^k     c_{k,j}(X) D^{k-j}
\eeq
where $c_{k,0} =\binom{n}{k}$ and for $j>0$, $c_{k,j} \in X \mathbb Q[X]_{j-1}$ ( See   \eqref{DLk} for the leading terms).
\end{enumerate}
For example, we have 
\begin{align*}
& {\small \text{$  \PF_1 =  \n\,D+\frac{(\n+3)X}{ 2  },
$}}
\qquad {\small \text{$
\PF_2 = \frac{\n (\n+9)}{2} D^2 + \frac{(\n^2+12\n+23)X}{2 } D+\cdots.$}}
\end{align*}
Now let $\widetilde \PF:= L ^{\frac{\n+3}{2}} \, \PF (L ^{-\frac{\n+3}{2}}) = \sum_{k=1}^{\n+5} z^k \cdot \widetilde \PF_i$, with $\widetilde \PF_1 = \n D$. Then  \eqref{PFR} becomes
\begin{align*}
& \textstyle \widetilde \PF_1 \big(\frac{r_1}{L_\alpha}\big) +  \frac{1}{L_\alpha}\!\cdot \widetilde \PF_2 1 =0, \\
& \textstyle \widetilde \PF_1  \big(\frac{r_2}{ L_\alpha^2}\big)  +  \frac{1}{L_\alpha}\!\cdot \widetilde \PF_2 \big(\frac{r_1}{L_\alpha}\big)  +  \frac{1}{L_\alpha^2}\!\cdot \widetilde \PF_3 1  =0, \\
&\cdots  \  \cdots. 
\end{align*}
By the shape of \eqref{PFk}, we know that any transformation 
$$D \mapsto D+a \cdot X,\qquad  a\in \mathbb Q$$ will not change the degree estimate of the coefficients $f_{k,j}$. Hence the operator $\widetilde \PF_1$ has exactly the same properties as \eqref{PFk}. 
We observe that the coefficient of $z^{k+1}$ in   \eqref{PFR} can be always written in the form
$$ \textstyle
  \n\cdot D \big( \frac{r_k}{L_\alpha^k} \big) =\frac{1}{L_\alpha^k}  \mathcal P_k(r_0,r_1,\cdots,r_{k-1})
$$
where $r_0=1$ and $\mathcal P_k$ are differential polynomials in $r_i$ and $X$: \footnote{We can see if we let
$\deg X=1, \quad \deg r_i = i, \quad \deg D = 1$, 
then we have 
$\textstyle
\deg \mathcal P_{k} = \deg \widetilde \PF_{k+1} = \deg   \PF_{k+1}  = k+1.$}
$$ \textstyle
\mathcal P_k (r_0,\cdots,r_{k-1}) = -\sum_{j=1}^{k-1}  \widetilde \PF_{k+1-j}\big|_{D \mapsto D-\frac{j}{\n} X} (r_{j}).
$$
Hence equation \eqref{PFR} can be solved by induction: Suppose for any $i<k$, $r_i$ is a polynomial in $X$ of degree $i$, and does not depend on $\alpha$. Then $\mathcal P_k (r_0,\cdots,r_{k-1}) $ is a polynomial  degree $k+1$. Further by the property of $c_{k,j}$ in \eqref{ckjdegree}, and by the fact $Df(X)$ is divided by $X$ for any polynomial $f(X)$, it is divided by $X$.
By using
$DL_\alpha = \frac{1}{\n} L_\alpha \cdot X $  again one has
$$
\frac{r_k}{  {L_\alpha^k} }  = \int   \frac{\n}{L_\alpha^k}\mathcal P_k(r_1,\cdots,r_{k-1}) \,  \frac{dL_\alpha}{ X\cdot L_\alpha}.
$$
Notice $X^{-1}\mathcal P_k \in \mathbb Q[X]_k$ and $X=1+L_{\alpha}^{-\n}$. Thus the integrand is a polynomial of $L_\alpha^{-1}$. Further, we obtain $r_k \in Q[X]_{k}$ and independent on $\alpha$. By initial data $r_0=1$ and the following initial condition used in the integration
\beq \label{initialcond} 
\text{``  $ {r_k} \cdot{  {L_\alpha^{-k}} } $ vanishes when $L_\alpha^{-1} \rightarrow 0$ "} .
\eeq
This completes the induction. We will prove \eqref{initialcond} in Section \ref{asymptI}.

\def\fS{\mathfrak S}
\subsection{Asymptotic expansion} \label{asymptI}

Following \cite{G96,CG}, we introduce the Landau-Ginzburg potential $W: (\mathbb C^*)^{n+1} \rightarrow \mathbb C$ for the  equivariant GW theory of Fano hypersurface $X_m\subset \mathbb P^n$, or equivalently, $\mathcal O(m)$-twisted GW theory of $\mathbb P^{n}$ with $m<n+1$:
$$ \textstyle
W(x_0,\cdots,x_n):=\sum_{i=0}^{n}  (x_i - \lambda_i \ln x_i)  +  y,\quad  y^m q =  \prod_{i=0}^{n} x_i .
$$
We consider the critical points of $W$
$$
\Big\{ \mathbf x_\alpha :  \frac{\partial W}{\partial x_i} (\mathbf x_\alpha)=0 \Big\} 
$$
There are $n+1$-critical points for generic equivariant parameters $\{\lambda_i\}$:
\beq \label{critpt}
(x_i)_\alpha = L_\alpha +\lambda_i,\qquad y_\alpha = -m L_\alpha
\eeq
with critical value $u_\alpha = W(\mathbf x_\alpha)$,
where $\{L_\alpha\}$ are $n+1$-solutions \footnote{At the critical point, the condition $ x_i\partial_{x_i} W = 0$ implies  $ x_i - \lambda_i =  - y /m$ for all $i$.  Hence by using $\prod_i( -\frac{y}{m} + \lambda_i) = y^m q$ we solve $y_\alpha$ and then $(x_i)_\alpha$.  } of $\prod_i(L+\lambda_i)=(-mL)^m q$. 
Near each critical point $\mathbf x_\alpha$ we have a \emph{Lefschetz thimble $\gamma_\alpha$}, which is a real $n$-dimensional cycle in $\mathbb C^n$ such that restricted to the Milnor fibre $W^{-1}(u)$ it is the vanishing cycle.

We consider the oscillatory integral
$$
I_\alpha(q,z):=\int_{\gamma_\alpha \subset (\mathbb C^*)^{n}}  e^{W/z}  \frac{d x_0 \wedge \cdots \wedge d x_{n}}{x_0 \cdots x_{n}}.
$$
We assume $z<0$ and study its asymptotic expansion as $z\rightarrow 0^-$, from the negative real axis.
By the result of \cite{G96,CCIT2}, the asymptotic expansion of $I_\alpha$ coincides with\footnote{Here the critical value $u_\alpha=\int (L_\alpha+t_\alpha) \frac{dq}{q}$ differ from the $\tau_\alpha$ defined in \eqref{tauloc0} by 
some constant.}
$$
I_\alpha(q,z) \asymp
 e^{u_\alpha /z} (-2\pi z)^{\frac{n+1}{2}}R^\lambda(z)^* \mathbf 1 |_{\pt_\alpha} .
$$
Here let $S^\lambda(z)$ be the $S$-matrix of the equivalently $\mathcal O(m)$-twisted GW theory of $\mathbb P^{n}$, and  $ R^\lambda(z)$ is defined via the Birkhoff factorization
$$
\Delta^\lambda(z) S^\lambda(z)  = R^\lambda(z) e^{u/z}
$$
with $\Delta^\lambda(z) = \diag \big\{\exp  \sum_{k> 0 }     \frac{B_{2k}}{2k(2k-1)}   \big( \frac{1}{(-m \lambda_\alpha)^{2k-1}}
 +\sum_{\beta\neq \alpha}\frac{1}{(\lambda_\alpha-\lambda_\beta)^{2k-1} } \big) z^{2k-1} \big\}_{\alpha=0}^n$.

\def \cQ{\mathcal Q}
We now use the saddle point method to compute the asymptotic expansion. We consider the Taylor expansion of $W$ near critical point $\mathbf x_\alpha$ :
$$
W = u_\alpha + \frac{1}{2} \cQ(\xi)+\sum_{k\geq 3} \frac{1}{k!}\sum_{i_1,\cdots,i_k} \partial_{x_{i_1}}\cdots \partial_{x_{i_k}}W (\mathbf x_\alpha) \xi_{i_1} \cdots \xi_{i_k} 
$$
where $\xi = \mathbf x- \mathbf x_\alpha$ is the local coordinate and
 $$\textstyle
  \cQ(\xi)=\sum_{i,j} \partial_{x_i}  \partial_{x_j}W (\mathbf x_\alpha) \xi_i \xi_j $$
 is a non-degenerate quadratic form.

 For $I=(i_1,\cdots,i_k)$, we denote by
$$
\partial_{x_I}:= \partial_{x_{i_1}}\cdots \partial_{x_{i_k}}  ,\qquad  \xi_I:= \xi_{i_1} \cdots \xi_{i_k} 
$$
Then the oscillatory integral has the following form (note $z=-s<0$ is real) 
\begin{align*}
 I_\alpha(q,z) &\, =  e^{u_\alpha /z} 
\int_{\gamma_\alpha}  e^{\sum_{k\geq 3} \sum_{I=(i_1,\cdots,i_k)}   \frac{1}{z} \black \partial_{I}W(\mathbf x_\alpha) \xi_I} e^{\frac{1}{2}\cQ(\xi)   \frac{1}{z} \black }  d\omega
\\
&\, \asymp e^{u_\alpha /z}  \Big( {\textstyle \prod_i (x_i)^{-1}_\alpha \int_{\mathbb R^n}   e^{\frac{1}{2z}\cQ(\xi)}  d\xi_1\cdots d\xi_n }\Big) \cdot \Big( 1+  \sum_{l>0} f_l \cdot (-z)^l\Big)
\end{align*}
 where $d\omega :={\textstyle \frac{d \xi_0 \cdots d \xi_n}{ (\xi_0+(x_0)_\alpha)\cdots (\xi_n+(x_n)_\alpha)}} $.
We prove
\begin{proposition} \label{initialconditionofR}
For $l>0$, $f_l$ are rational functions of $L_\alpha$. Further
\beq \label{limitcond}
\lim_{L_\alpha^{-1} \rightarrow 0} f_l(L_\alpha)  =0.
\eeq
\end{proposition}
\begin{proof} It suffices to prove that for any $|I_j|\geq 3$
\begin{align}\label{goal}
{\small\text{$\Delta^{-1}\partial_{x_{I_1}}\!\!W (\mathbf x_\alpha)\cdots \partial_{x_{I_r}}\!\!W (\mathbf x_\alpha)  \cdot \int_{\mathbb R^n}  \xi_{I_1}\cdots  \xi_{I_r} e^{\frac{1}{2z}\cQ(\xi)}  {\textstyle \frac{d \xi_0 \cdots d \xi_n}{ (\xi_0+(x_0)_\alpha)\cdots (\xi_n+(x_n)_\alpha)}} = O(L_\alpha^{-1}).$}}
\end{align}
where  $\Delta:= {\textstyle \prod_i (x_i)^{-1}_\alpha \int_{\mathbb R^n}   e^{\frac{1}{2z}\cQ(\xi)}  d\xi_1\cdots d\xi_n } $.
 By using $x_i\partial_{x_i} y = m^{-1} y$,
we have
$$ \textstyle
\frac{\partial^2 W}{\partial x_i^2} =  \frac{\lambda_i }{x_i^2} +\frac{(1-m) }{m^2 x_i^2} y,\quad
\frac{\partial^2 W}{\partial x_i \partial x_j} = \frac{y }{m^2 x_i x_j} 
\and
\frac{\partial^k W}{\partial x_{i_1}\cdots \partial x_{i_k}}  \in  \frac{1}{x_{i_1}\cdots x_{i_k}}\mathbb Q[y]_1
$$
 for $k\geq 3$.
At  the critical point, by \eqref{critpt} we see $\partial_{x_I} W (\mathbf x_\alpha)$ are rational function of $L_\alpha$ with the following property
\begin{align}\label{Wxalp}
\partial_{x_{i_1}}\cdots \partial_{x_{i_k}}W (\mathbf x_\alpha) = O(L_\alpha^{-k+1}) \qquad  \text{ when }  L_\alpha^{-1} \rightarrow 0
\end{align}
 By combining the estimate proved by the lemma below, one obtain the LHS of \eqref{goal} has leading term (lowest order of $L_{\alp}^{-1}$ expansion)  
 $$(-z)^{\frac{n+1+\sum_j |I_j|}{2}} L_\alp^{\frac{n+1+\sum_j |I_j|}{2} - \frac{n+1}{2} +\sum_j(1-|I_j|)}  = (-z)^{\frac{n+r+1}{2}?+l} L_\alp^{-l}  $$
where $l:=-\sum_j (1- \frac{|I_j|}{2})>0$. The Gaussian integral vanishes if $l$ is not an integer.
Hence this proves \eqref{goal}.
 \end{proof}

\begin{lemma} \label{gaussint}Write $k=k_0+\cdots+k_n$. Then as expansions near $L_{\alp}^{-1}=0$ we have 
\begin{align}\label{Gau} 
 &\int_{\mathbb R^n}  \xi_0^{k_0}\cdots  \xi_n^{k_n} e^{\frac{1}{2z}\cQ(\xi)} {d \xi_0 \cdots d \xi_n}  
\quad \in \quad  (-z)^{\frac{n+1+k}{2}}  L_\alpha^{\frac{n+1+k}{2} } \RR[\![(L_{\alp}^{-1})]\!]
\end{align}
when $k$ is even; while \eqref{Gau} vanishes if $k$ is odd.
 In case $k=0$, the \eqref{Gau} is of form 
 $  L_\alp^{\frac{n+1}{2}}(c+ O(L_\alp^{-1}))$ for some $c\neq  0$ independent of $L_\alp$.
\end{lemma}
\begin{proof} 

At each critical point \eqref{critpt} we have
 $$
\frac{\partial^2 W}{\partial x_i^2} = \frac{1}{(L_\alp+\lam_i)^2} \Big(  \lambda_i  +\frac{(m-1) }{m} L_\alp \Big) 
\and
\frac{\partial^2 W}{\partial x_i \partial x_j}   =-\frac{L_\alpha}{m(L_\alpha+\lambda_i)(L_\alpha+\lambda_j)}
$$
Denote $B_{ij}= \sum_{i,j} \partial_{x_i}  \partial_{x_j}W (\mathbf x_\alpha)$, we have $B_{ij}=B_{ji}$.
Further, we compute
\begin{align*}
\det B&=m^{n+1}L_\alp^{-(n+1)}   \Big( 1  - \frac{n+1}{m}  +O(L_\alp^{-1})  \Big)
 \end{align*}
In our case, which is Fano, making $1-\frac{n+1}{m}\neq 0$.
 This  implies $$ (B^{-1})_{ij} =  O( L_\alp) \qquad L_\alpha^{-1} \rightarrow 0.$$
 
 Gaussian integral says \eqref{Gau} vanishes if $k:=\sum k_i$ is not even. Assume $k=2\ell$ with $\ell\in\NN$.
By Wick's theorem (see e.g. \cite[Sect.\,0.2.1]{CPS}),  we obtain 
  \begin{align*} 
  \int_{\RR^n}  \xi_{d_0}\cdots  \xi_{d_k} e^{\frac{1}{2z}\cQ(\xi)} {d \xi_0 \cdots d \xi_n}  \ &
=  (-z)^{\frac{n+1+k}{2}}   \frac{(2\pi)^{n/2}}{\sqrt{\det B}} \frac{1}{2^\ell \ell !} \sum_{\sigma\in \fS_{2\ell}} \prod_{i=1}^l (B^{-1})_{d_{\sigma(2i-1)} d_{\sigma(2i)}}
  \\
 &\in  (-z)^{\frac{n+1+k}{2}}   L_\alp^{\frac{n+1}{2}} L_\alp^{\ell} O((L_{\alp}^{-1})^0) .
 \end{align*}
where $\fS_{2\ell}$ be the set of permutations of $\{1,2,\cdots,2\ell\}$ that is product of $\ell$ disjoint transpositions.  
\end{proof}
 \begin{remark}
For CY case, \eqref{limitcond} is no longer true, 
since the estimate of $\det B$ is changed to $O(L_\alpha^{-n-2})$. Indeed, for CY  case with $\sum_i \lambda_i \neq 0$, we  have
$$
\lim_{L_\alpha^{-1} \rightarrow 0} f_l(L_\alpha)  = \text{constant}.
$$
\end{remark}

\subsection{Specialization}
In our case, we have $m=5$ and $n=\n+4$. Further
the equivariant parameters are given by the following special values
$$(t_i)=(0,0,0,0,0,-\xi_{\n}  t,-\xi_{\n}^2  t,\cdots,-\xi_{\n}^\n  t). $$
At critical points, the equation for $L_\alpha$ becomes $y^5(y^\n/5^\n-\ft^\n) = -5^5  q \cdot y^5$. We consider
$$
\big\{ y_\alpha = 5 \zeta_\n^\alpha (\ft^\n +5^5 q)^{1/\n} = 5L_\alpha, \quad (x_i)_\alpha = \lambda_i +L_\alpha \big\}, \qquad \alpha=1,\cdots,\n.$$
At these $\n$ critical points, 
 the Hessian $(\partial_{x_i}\partial_{x_j} W)$ is non-degenerate and then all the argument in Section \ref{asymptI} can be applied. The equation \eqref{limitcond} gives us exactly \eqref{initialcond}.

  \newpage

 \begin{appendix}

\addtocontents{toc}{\protect\setcounter{tocdepth}{1}}

\section{Quintic's QDE}  \label{quinticqde}
Let $S^Q(\tau,z) =\sum_{k \geq 0} S^Q_k(\tau) z^{-k}$ be the $S$-function for the quintic, then under the flat basis we have the following quantum differential equation (QDE)
\begin{equation*}
z\, d S^Q (\tau,z) = d\tau *_\tau S^Q (\tau,z).
\end{equation*}
When $\tau$ is the mirror map $\tau_Q$ \eqref{tauloc0},   by using Divisor equation the QDE becomes
  \begin{align}\label{qQDE-star} (H+ zD) \, S^Q(z)\sta  =   S^Q (z)\sta * \dot \tau_Q . \end{align} 
where  $D=q \frac{d}{dq}$ and $\dot \tau_Q:=  H+D\tau_Q = I_{1,1} H$. 
 Suppose the quantum product $* \dot \tau_Q $ is given by a matrix $A^Q$  under the flat basis, namely
  $H^k *\dot \tau_Q = H^j (A^Q)_{j}^k $.
 Then by using
\beq \label{3point}
\left< 1,H,H^2 \right>_{0,3}^Q = 1,\qquad
\left< H,H,H \right>_{0,3}^Q = I_{2,2}/I_{1,1},\qquad
\eeq
where for $k=1,2,3$ we define $J_k:= I_k/ I_0$ and
\beq \label{I22}  
I_{2,2} = \text{ \small $ I_{1,1}\cdot \frac{d^2}{d \tau_Q^2} J_2 = D\Big(\frac{D (I_2/I_0)}{D(I_1/I_0)}\Big)$},
\eeq 
we can easily deduce the above matrix
\beq \label{quintic-QDE}
A^Q = {\small \begin{pmatrix}0 &&& \\I_{1,1}&0&&\\&  {I_{2,2}} &0&\\&& I_{1,1} &0  \end{pmatrix} }.
\eeq
Further, by solving the QDE we obtain the $S$-matrix at $\tau =\tau_Q$:
{\footnotesize
\begin{align}\label{quinticS} \quad&S^Q(z)^* = \id  +  \frac{1}{z}\begin{pmatrix}
 0 &  &   &  \\
 J_1 & 0  & &  \\
  &    \frac{J_2'}{J_1'}& 0  & \\
  &&J_1& 0
 \end{pmatrix}+ \frac{1}{z^2}\begin{pmatrix}
 0 &   && \\
  & 0 &   &\\
  J_2  &  & 0  &  \\
&  {\footnotesize \text{$\!\!\!\frac{J_2'}{J_1'}J_1\!-\!J_2$}}  && \quad 0
\end{pmatrix}+ \frac{1}{z^3}\begin{pmatrix}
 0 &   &  &  \\
  & 0 &   &  \\
  &  & 0  &  \\
 J_3 &&& 0
\end{pmatrix}. \
\end{align}}



\vspace{0.7cm}
\section{Staiblization of grahps}\label{stablization}

We consider a genus $g$ connected graph with $n$ legs, and with each vertex $v$ is labeled by its genus
$g_v\in\ZZ_{\ge 0}$.
A vertex $v$ with valence $n_v$ is called stable, if $2g_v-2 +n_v >0$. A graph is called stable if all its vertices are stable. 

Given such a graph $\Theta$ with $2g-2+n>0$, we can apply ``stabilization" to it to get a stable graph $ \Theta\ust$, by removing all (maximal rational) tails; replacing all (maximal rational) 
chains between two stable vertices by single edges; and by contracting all (maximal rational) 
chains between one stable vertices and an vertex with a leg incident to it.
In this paper, we will use ``tail", ``chain", and  ``end" to describe such three types of subgraphs just mentioned. (Here we follow \cite{CGT} for the terminologies of ``tail" and  ``end".)

Any vertex of $\Theta$ that remains a vertex after stabilization is called a ``stable-graph-vertex".
For any leg $l$ of $\Theta$, the stabilization associates to it a unique stable-graph-vertex $cl(\ell)\in V(\Theta)$
that is the vertex of $\Theta$ so the leg $l$ will be attached to after stabilization.
      
In this paper we apply stabilization to two kinds of decorated graphs: $\nmsp$ localization graphs, and bipartie graphs. 
%
%
%
\vspace{0.7cm}
\section{List of symbols}
\vspace{-0.5cm}
{
\begin{table}[H] \centering 
\begin{center}
 \noindent\begin{tabular}{cp{0.8\textwidth}}
\hline\\
 $\n$ & a large positive integer \black\\
 $p$ & the equivariant  hyperplane class $c_1(\sO_{\PP^{4+\n}}(1))$    
 \\
 $H$ & hyperplane class of quintic  \\
  $Q$ & the quintic $3$-fold \\
    $F$ & the base field $F=\QQ(\ft)$ \\
  $\aA$ & coefficient ring  $\aA=\QQ(\ft)[\![q]\!]$ of all cohomologies\\
  $\sH$ & the 
  $\nmsp$ 
  state space with twisted inner product and unit $\mathbf 1\in \sH$ \\
  $\sH_Q$ & the quintic state space with twisted inner product  and unit $\mathbf 1_Q\in H^0(Q)$ \\
  $\sH_{\pt_\alpha}$ & the point state space with twisted inner product and unit $\mathbf 1_\alp\in H^0(\pt)$ \\
  $\{\phi_i\}$ & the basis $\{\phi_i:=p^i\}_{i=0}^{\n+3}$ of $\sH$ with dual basis $\{\phi^{i}\}_{i=0}^{\n+3}$ \\
  $R\loc$ & the $R$-matrix by $\nmsp$ localization  $R\loc(z) \in \End(\sH )\otimes \aA[\![z]\!]$\\
  $R$ & the $R$-matrix by $\nmsp$ localization composited with the GRR formula
  \\
  $A_k,\!B_k$ & generators defined from quintic $I$-function, especially $A:=A_1$, $B:=B_1$\\
  $\sR$ & the ring of five generators  $\sR:=\QQ[A,B,B_2,B_3,Y]$\\
  $\mS^\alp_{a;i}$ & the specialized $S$-matrix $\mS_a^\alp:=   S\msp(z) 1^\alpha   \big|_{  z= \frac{5\ft_\alpha}{a}}, \and  \mS^\alp_{a;i}:= ( \mS^\alp_{a},  p^i)^\tw $ \\
  $\locg$ & regular localization graph or  $[0,1]$ localization graph\\
  $\Gamma$ &  stable graph\\
  $\bipg$ & decorated bipartite graph\\
  $G_{g,n,\bd}\ureg$ & the set of regular localization graphs\\
  $G_{g,n}^\n$ & the set of stable graphs with vertices decorated by $Q$ or $\pt_\alpha$\\
  $\Xi_{g,n}^r$ & the set of decorated bipartite graphs\\
\hline
\end{tabular}
\end{center}
\end{table}
}

\end{appendix}

\end{document}